\documentclass[12pt]{amsart}

\usepackage{amsmath,amssymb}
\usepackage[latin1]{inputenc}
\usepackage[dvips]{graphics}
\input{xy}
\xyoption{poly}
\xyoption{2cell}
\xyoption{all}
\usepackage{epsfig}  
\usepackage{pstricks}
\usepackage{color}

\newcommand{\someothergreekletter}{\alpha}
\newcommand{\somegreekletter}{\alpha}
\newcommand{\somediagonalgreekletter}{\gamma}

\newcommand{\overunder}[2]{
\!\begin{array}{c}
\scriptstyle{#1}\\[-.1in]
-\!\!\!-\!\!\!-\\[-.1in]
\scriptstyle{#2}
\end{array}
\!
}

\def\sgn{\operatorname{sgn}}
\def\TT{\mathbb{T}}

\def\Trop{\operatorname{Trop}}
\def\QQ{\mathbb{Q}}

\newcommand{\za}{\alpha}

\newcommand{\zD}{\Delta}

\newcommand{\zg}{\gamma}

\DeclareMathOperator{\Frac}{Frac}
\DeclareMathOperator{\Match}{Match}

\newcommand{\Sbar}{\widetilde{S}}
\newcommand{\Bbar}{\widetilde{B}}
\newcommand{\Mbar}{\widetilde{M}}
\newcommand{\Tbar}{\widetilde{T}}
\newcommand{\Rbar}{\widetilde{R}}

\newcommand{\Qbar}{\widetilde{Q}}
\newcommand{\Dbar}{\widetilde{\Delta}}
\newcommand{\Abar}{\widetilde{\A}}

\newcommand{\A}{\mathcal{A}}
\newcommand{\F}{\mathcal{F}}

\newcommand{\PP}{\mathbb{P}}
\newcommand{\x}{\mathbf{x}}

\def\Aprin{\Acal_\bullet}
\def\Xcal{\mathcal{X}}
\def\Acal{\mathcal{A}}
\def\Fcal{\mathcal{F}}

\def\yy{\mathbf{y}}

\def\xx{\mathbf{x}}

\def\ZZ{\mathbb{Z}}

\newcommand{\zLp}{\ell_p}
\newcommand{\zLq}{\ell_q}

\newtheorem{theorem}{Theorem}[section]

\newtheorem{lemma}[theorem]{Lemma}

\newtheorem{question}[theorem]{Question}

\newtheorem{prop}[theorem]{Proposition}
\newtheorem{conjecture}[theorem]{Conjecture}
\newtheorem{cor}[theorem]{Corollary}
\newtheorem{Cor}[theorem]{Corollary}

\theoremstyle{definition}
\newtheorem{definition}[theorem]{Definition}
\newtheorem{Def}[theorem]{Definition}

\theoremstyle{remark}
\newtheorem{remark}[theorem]{Remark}

\numberwithin{equation}{section}

\begin{document}
\title{Positivity for cluster algebras from surfaces}
\author{Gregg Musiker}
\address{Department of Mathematics, MIT, Cambridge, MA 02139}
\email{musiker@math.mit.edu}
\author{Ralf Schiffler} 
\address{Department of Mathematics, University of Connecticut, 
Storrs, CT 06269-3009}
\email{schiffler@math.uconn.edu}
\thanks{{The first and third authors are partially supported by  NSF Postdoctoral Fellowships.
The second author is partially
   supported by the NSF grant DMS-0700358.}}
\author{Lauren Williams}
\address{Department of Mathematics, Harvard University,
Cambridge, MA 02138}
\email{lauren@math.harvard.edu}

\date{}
\dedicatory{}

\begin{abstract}  We give combinatorial formulas for the Laurent expansion of 
any cluster variable in any cluster algebra
coming from a triangulated surface
(with or without punctures), with respect to an arbitrary seed.  Moreover,
we work in the generality of {\it principal coefficients}.
An immediate corollary of our formulas is a proof of 
the positivity conjecture of Fomin and Zelevinsky for cluster
algebras from surfaces, in geometric type.
\end{abstract}

\maketitle

\setcounter{tocdepth}{1}
\tableofcontents

\section{Introduction}\label{intro}
Since their introduction by Fomin and Zelevinsky in \cite{FZ1}, cluster
algebras have been shown to be related to diverse areas of mathematics
such as total positivity, quiver representations, Teichm\"uller theory,
tropical geometry, Lie theory, and Poisson geometry.  One of the
main outstanding conjectures about cluster algebras is the
{\it positivity conjecture}, which says that if one
fixes a cluster algebra $\A$ and
an {\it arbitrary} cluster $\x$, one can express each
cluster variable $x\in \A$
as a Laurent polynomial with {\it positive coefficients} in the variables of
$\x$.

There is a class of cluster algebras arising from {\it surfaces with marked
points}, introduced by Fomin, Shapiro, and Thurston in \cite{FST} 
(generalizing work of Fock and Goncharov \cite{FG1, FG2} and Gekhtman, Shapiro, and 
Vainshtein \cite{GSV}), 
and further developed  in \cite{FT}.  This class is quite large:
(assuming rank at least three) it has been shown \cite{FeSTu} that all but finitely many
skew-symmetric cluster algebras of {\it finite mutation type} come from this
construction.  Note that the class of cluster algebras of finite mutation type
in particular contains those of finite type.

The goal of this paper is to give a combinatorial expression
for the Laurent polynomial which expresses any cluster variable
in terms of any cluster, for any cluster algebra arising from a
surface.
An immediate corollary is the proof of the positivity conjecture
for all cluster algebras arising from surfaces.

A {\it cluster algebra} $\A$ of rank $n$ is a subalgebra of an
{\it ambient field} $\F$ isomorphic to a field of rational functions
in $n$ variables.  Each cluster algebra has a distinguished set of
generators called {\it cluster variables}; this set is a union of overlapping
algebraically independent $n$-subsets of $\F$ called {\it clusters},
which together have the structure of a simplicial complex called the {\it cluster
complex}.  See Definition \ref{def:cluster-algebra} for precise details.
The clusters are related to each other by birational transformations
of the following kind: for every cluster $\x$
and every cluster variable $x\in \x$, there is another cluster
$\x' = \x - \{x\} \cup \{x'\}$, with the new cluster variable
$x'$ determined by an {\it exchange relation} of the form
\begin{equation*}
x x' = y^+ M^+ + y^- M^-.
\end{equation*}
Here $y^+$ and $y^-$ belong to a {\it coefficient semifield} $\PP$,
while $M^+$ and $M^-$ are two monomials in the elements of
$\mathbf x - \{x\}$.  There are two dynamics at play in the exchange relations:
that of the monomials $M^+$ and $M^-$, which is encoded in the exchange matrix,
and that of the coefficients.

A classification of finite type cluster algebras -- those with finitely
many clusters -- was given by Fomin and Zelevinksy in \cite{FZ2}.  They
showed that this classification is parallel to the famous Cartan-Killing
classification of complex simple Lie algebras, i.e.\ finite type cluster
algebras either fall into  one of the infinite families
$A_n$, $B_n$, $C_n$, $D_n$, or are of one of the exceptional types
$E_6, E_7, E_8, F_4$, or $G_2$.  Furthermore, the {\it type} of a finite
type cluster algebra depends only on the dynamics of the corresponding
exchange matrices, and not on the coefficients.
However, there are many examples
of cluster algebras of geometric origin which -- despite having the same {\it type} --
have totally different systems of coefficients.
This motivated Fomin and Zelevinsky's work in \cite{FZ4}, in which they
studied the dependence of a cluster algebra structure on the choice of
coefficients.
One surprising result of \cite{FZ4} was that there is a special
choice of coefficients, the {\it principal coefficients}, which
have the property that computation of explicit expansion formulas
for the cluster variables in arbitrary cluster algebras can
be reduced to computation of explicit expansion formulas
in cluster algebras with principal coefficients.
A corollary of this work is that to prove
the positivity conjecture in geometric type, it suffices to prove
the positivity conjecture for the class of principal coefficients.

This takes us to the topic of the present work.  Our main results
are combinatorial formulas for cluster expansions of cluster variables
with respect to any seed, 
in any cluster algebra coming from a surface.
Our formulas are manifestly positive,
so as a consequence  
we obtain the following result.

\begin{theorem}\label{main}
Let $\mathcal A$ be any cluster algebra arising from a surface,
where the coefficient system is of geometric type, and
let $\xx$ be any choice of initial cluster.
Then the Laurent expansion of every cluster variable with respect
to the cluster $\xx$ has non-negative coefficients.
\end{theorem}

Our results generalize those obtained in \cite{S2}, 
where cluster algebras from the (much more restrictive) case of 
surfaces without punctures were considered.  This work  in turn  generalized
\cite{ST}, which treated
cluster algebras from unpunctured surfaces with a very limited coefficient system 
that was associated to the boundary of the surface. 
The very special case where the surface is a polygon and 
coefficients arise from the boundary was covered in \cite{S},
and also in unpublished work \cite{CP, FZ5}.  See also \cite{Propp}.
The recent paper \cite{MS} gave an alternative formulation of 
the results of \cite{S2}, using the language of perfect matchings
as opposed to {\it $T$-paths}.

Many other people have worked on the problem of finding
Laurent expansions of cluster variables,
and on the positivity conjecture.  
However, most of the results
so far obtained have strong restrictions on the 
cluster algebra, the choice of 
initial cluster (equivalently, initial seed), 
or on the system of coefficients.

In the case of rank 2 cluster algebras, 
the papers \cite{SZ,Z,MP} gave  
cluster expansion formulas in affine types.  Positivity in 
these cases was subsequently generalized to 
the coefficient-free rank 2 case in \cite{Dup}, using results of \cite{CalRein}.
In the case of finite type cluster algebras, the positivity conjecture 
with respect to a  bipartite seed follows from \cite[Corollary 11.7]{FZ4}.
Other work \cite{M} gave cluster expansions for coefficient-free
cluster algebras of
finite classical types with respect to a  bipartite seed.

A recent tool in understanding Laurent expansions of cluster variables
is the connection to quiver representations 
and the introduction
of the cluster category \cite{BMRRT} (see also \cite{CCS1} in type $A$).
More specifically, there is a geometric interpretation 
(found in \cite{CC} and generalized in \cite{CK})
of coefficients in Laurent expansions as 
Euler-Poincar\'e characteristics of appropriate Grassmannians of quiver 
representations.  
Using this approach, the works \cite{CC,CK,CK2} gave an expansion formula in the
case where the cluster algebra is acyclic and the initial cluster lies in an
acyclic seed (see also \cite{CZ} in rank 2); this was subsequenty generalized to arbitrary
clusters in an acyclic cluster algebra \cite{Palu}.  Note that these formulas
do not give information about the coefficients.
Later,  \cite{FK} generalized these results to cluster algebras  
with principal coefficients that admit a categorification by a 2-Calabi-Yau category \cite{FK};
by  \cite{A} and \cite{ABCP,LF}, such a categorification exists in the case of 
cluster algebras associated to surfaces with non-empty boundary.  
Recently \cite{DWZ} gave expressions for the $F$-polynomials in 
any skew-symmetric cluster algebra.   
However, since all of the above formulas are in terms of Euler-Poincar\'e characteristics
(which can be negative), they do not immediately imply the positivity conjecture.

The  work \cite{CalRein} used the above approach to prove the positivity conjecture
for coefficient-free acyclic cluster algebras, with respect to an acyclic seed.
Building on 
results of \cite{HL},  Nakajima recently  
used quiver varieties to prove the 
positivity conjecture for cluster algebras 
that have at least one bipartite seed, with respect to any cluster \cite{Nak}.
This is a very strong result, but it does not overlap very much
with our Theorem \ref{main}.
Note that a bipartite seed is in particular acyclic, 
but not every acyclic type 
has a bipartite seed; the affine type $\tilde A_2$, for example, does not.
And the only surfaces that give rise to acyclic cluster algebras
are the polygon, the polygon with one puncture, the annulus, and the polygon with two punctures 
(corresponding to the finite types $A$ and $D$, and 
the affine types $\tilde A$ and $\tilde D$).
All other surfaces yield non-acyclic cluster algebras, see \cite[Corollary 12.4]{FST}.

The paper is organized as follows. We give background on cluster algebras and cluster algebras from surfaces
in Sections \ref{sect cluster algebras} and \ref{sect surfaces}.  In Section \ref{sect main} we present our 
formulas for  Laurent expansions of cluster variables, and in Section \ref{ClusterExamples}
we give examples, as well as identities in the coefficient-free case.  As the proofs 
of our main results are 
rather involved, we give a detailed outline of 
the main argument in Section \ref{sec outline}, before giving the proofs themselves in Sections 
\ref{sect Sbar} to \ref{sec plain} and \ref{sec Finish}.    
In Section \ref{sect applications}, we combine our results with those 
of \cite{FK} and \cite{DWZ} to give formulas for 
F-polynomials, g-vectors, and 
Euler-Poincar\'e characteristics of quiver Grassmannians.

Recall that cluster variables in cluster algebras from surfaces correspond to
{\it ordinary arcs} as well as arcs with {\it notches} at one or two ends.  
We remark that working in the generality of principal coefficients
is much more difficult than working in the coefficient-free case.
Indeed, once we have proved positivity for cluster variables corresponding
to ordinary arcs, the proof of positivity for cluster variables corresponding
to tagged arcs {\it in the coefficient-free case} follows easily, see 
Proposition \ref{coeff-free} and Section \ref{sec quick}.  
Putting back principal coefficients requires much more elaborate arguments, see
Section \ref{sec Finish}.  A crucial tool here is the connection to 
laminations \cite{FT}.

Note that all cluster algebras coming from surfaces are skew-symmetric.  In a sequel
to this paper we will explain how to use folding arguments to give positivity results
for cluster algebras with principal coefficients which are not skew-symmetric.

\textsc{Acknowledgements:}
We are grateful to the organizers of the workshop on cluster
algebras in Morelia, Mexico, where we benefited from Dylan Thurston's
lectures.  We would also like to thank Sergey Fomin and Bernard Leclerc for 
useful discussions.

\section{Cluster algebras}\label{sect cluster algebras}

We begin by reviewing the definition of cluster algebra,
first introduced by Fomin and Zelevinsky in \cite{FZ1}.
Our definition follows the exposition in \cite{FZ4}.

\subsection{What is a cluster algebra?}

To define  a cluster algebra~$\Acal$ we must first fix its
ground ring.
Let $(\PP,\oplus, \cdot)$ be a \emph{semifield}, i.e.,
an abelian multiplicative group endowed with a binary operation of
\emph{(auxiliary) addition}~$\oplus$ which is commutative, associative, and
distributive with respect to the multiplication in~$\PP$.
The group ring~$\ZZ\PP$ will be
used as a \emph{ground ring} for~$\Acal$.
One important choice for $\PP$ is the tropical semifield; in this case we say that the
corresponding cluster algebra is of {\it geometric type}.
\begin{Def}
[\emph{Tropical semifield}]
\label{def:semifield-tropical}
Let $\Trop (u_1, \dots, u_{m})$ be an abelian group (written
multiplicatively) freely generated by the $u_j$.
We define  $\oplus$ in $\Trop (u_1,\dots, u_{m})$ by
\begin{equation}
\label{eq:tropical-addition}
\prod_j u_j^{a_j} \oplus \prod_j u_j^{b_j} =
\prod_j u_j^{\min (a_j, b_j)} \,,
\end{equation}
and call $(\Trop (u_1,\dots,u_{m}),\oplus,\cdot)$ a \emph{tropical
 semifield}.
Note that the group ring of $\Trop (u_1,\dots,u_{m})$ is the ring of Laurent
polynomials in the variables~$u_j\,$.
\end{Def}

As an \emph{ambient field} for
$\Acal$, we take a field $\Fcal$
isomorphic to the field of rational functions in $n$ independent
variables (here $n$ is the \emph{rank} of~$\Acal$),
with coefficients in~$\QQ \PP$.
Note that the definition of $\Fcal$ does not involve
the auxiliary addition
in~$\PP$.

\begin{Def}
[\emph{Labeled seeds}]
\label{def:seed}
A \emph{labeled seed} in~$\Fcal$ is
a triple $(\xx, \yy, B)$, where
\begin{itemize}
\item
$\xx = (x_1, \dots, x_n)$ is an $n$-tuple of elements of~$\Fcal$
forming a \emph{free generating set} over $\QQ \PP$,
\item
$\yy = (y_1, \dots, y_n)$ is an $n$-tuple
of elements of $\PP$, and
\item
$B = (b_{ij})$ is an $n\!\times\! n$ integer matrix
which is \emph{skew-symmetrizable}.
\end{itemize}
That is, $x_1, \dots, x_n$
are algebraically independent over~$\QQ \PP$, and
$\Fcal = \QQ \PP(x_1, \dots, x_n)$.
We refer to~$\xx$ as the (labeled)
\emph{cluster} of a labeled seed $(\xx, \yy, B)$,
to the tuple~$\yy$ as the \emph{coefficient tuple}, and to the
matrix~$B$ as the \emph{exchange matrix}.
\end{Def}

We obtain ({\it unlabeled}) {\it seeds} from labeled seeds
by identifying labeled seeds that differ from
each other by simultaneous permutations of
the components in $\xx$ and~$\yy$, and of the rows and columns of~$B$.

We  use the notation
$[x]_+ = \max(x,0)$,
$[1,n]=\{1, \dots, n\}$, and
\begin{align*}
\sgn(x) &=
\begin{cases}
-1 & \text{if $x<0$;}\\
0  & \text{if $x=0$;}\\
 1 & \text{if $x>0$.}
\end{cases}
\end{align*}

\begin{Def}
[\emph{Seed mutations}]
\label{def:seed-mutation}
Let $(\xx, \yy, B)$ be a labeled seed in $\Fcal$,
and let $k \in [1,n]$.
The \emph{seed mutation} $\mu_k$ in direction~$k$ transforms
$(\xx, \yy, B)$ into the labeled seed
$\mu_k(\xx, \yy, B)=(\xx', \yy', B')$ defined as follows:
\begin{itemize}
\item
The entries of $B'=(b'_{ij})$ are given by
\begin{equation}
\label{eq:matrix-mutation}
b'_{ij} =
\begin{cases}
-b_{ij} & \text{if $i=k$ or $j=k$;} \\[.05in]
b_{ij} + \sgn(b_{ik}) \ [b_{ik}b_{kj}]_+
 & \text{otherwise.}
\end{cases}
\end{equation}
\item
The coefficient tuple $\yy'=(y_1',\dots,y_n')$ is given by
\begin{equation}
\label{eq:y-mutation}
y'_j =
\begin{cases}
y_k^{-1} & \text{if $j = k$};\\[.05in]
y_j y_k^{[b_{kj}]_+}
(y_k \oplus 1)^{- b_{kj}} & \text{if $j \neq k$}.
\end{cases}
\end{equation}
\item
The cluster $\xx'=(x_1',\dots,x_n')$ is given by
$x_j'=x_j$ for $j\neq k$,
whereas $x'_k \in \Fcal$ is determined
by the \emph{exchange relation}
\begin{equation}
\label{exchange relation}
x'_k = \frac
{y_k \ \prod x_i^{[b_{ik}]_+}
+ \ \prod x_i^{[-b_{ik}]_+}}{(y_k \oplus 1) x_k} \, .
\end{equation}
\end{itemize}
\end{Def}

We say that two exchange matrices $B$ and $B'$ are {\it mutation-equivalent}
if one can get from $B$ to $B'$ by a sequence of mutations.
\begin{Def}
[\emph{Patterns}]
\label{def:patterns}
Consider the \emph{$n$-regular tree}~$\TT_n$
whose edges are labeled by the numbers $1, \dots, n$,
so that the $n$ edges emanating from each vertex receive
different labels.
A \emph{cluster pattern}  is an assignment
of a labeled seed $\Sigma_t=(\xx_t, \yy_t, B_t)$
to every vertex $t \in \TT_n$, such that the seeds assigned to the
endpoints of any edge $t \overunder{k}{} t'$ are obtained from each
other by the seed mutation in direction~$k$.
The components of $\Sigma_t$ are written as:
\begin{equation}
\label{eq:seed-labeling}
\xx_t = (x_{1;t}\,,\dots,x_{n;t})\,,\quad
\yy_t = (y_{1;t}\,,\dots,y_{n;t})\,,\quad
B_t = (b^t_{ij})\,.
\end{equation}
\end{Def}

Clearly, a cluster pattern  is uniquely determined
by an arbitrary  seed.

\begin{Def}
[\emph{Cluster algebra}]
\label{def:cluster-algebra}
Given a cluster pattern, we denote
\begin{equation}
\label{eq:cluster-variables}
\Xcal
= \bigcup_{t \in \TT_n} \xx_t
= \{ x_{i,t}\,:\, t \in \TT_n\,,\ 1\leq i\leq n \} \ ,
\end{equation}
the union of clusters of all the seeds in the pattern.
The elements $x_{i,t}\in \Xcal$ are called \emph{cluster variables}.
The 
\emph{cluster algebra} $\Acal$ associated with a
given pattern is the $\ZZ \PP$-subalgebra of the ambient field $\Fcal$
generated by all cluster variables: $\Acal = \ZZ \PP[\Xcal]$.
We denote $\Acal = \Acal(\xx, \yy, B)$, where
$(\xx,\yy,B)$
is any seed in the underlying cluster pattern.
\end{Def}

The remarkable {\it Laurent phenomenon} asserts the following.

\begin{theorem} \cite[Theorem 3.1]{FZ1}
\label{Laurent}
The cluster algebra $\Acal$ associated with a seed
$(\xx,\yy,B)$ is contained in the Laurent polynomial ring
$\ZZ \PP [\xx^{\pm 1}]$, i.e.\ every element of $\Acal$ is a
Laurent polynomial over $\ZZ \PP$ in the cluster variables
from $\xx=(x_1,\dots,x_n)$.
\end{theorem}

\begin{definition}
Let $\Acal$ be a cluster algebra, 
$\Sigma$ be a seed, and $x$
be a cluster variable of $\Acal$.  We denote by 
$[x]_{\Sigma}^{\Acal}$ the Laurent polynomial given by 
Theorem \ref{Laurent} which expresses $x$ in 
terms of the cluster variables from $\Sigma$.  We  refer to this
as the {\it cluster expansion} of $x$ in terms of $\Sigma$.
\end{definition}

The longstanding {\it positivity conjecture} \cite{FZ1} says that 
even more is true.

\begin{conjecture} {\em (Positivity Conjecture)}
For any cluster algebra $\Acal$, any seed $\Sigma$, and any cluster
variable $x$, the Laurent polynomial
$[x]_{\Sigma}^{\Acal}$ has coefficients which are 
non-negative 
integer linear combinations of elements in $\PP$.

\end{conjecture}

\begin{remark}\label{rectangular}
In cluster algebras whose ground ring is 
$\Trop(u_1,\dots, u_{m})$ (the tropical semifield), it is convenient to replace the
matrix $B$ by an $(n+m)\times n$ matrix $\tilde B=(b_{ij})$ whose upper part
is the $n\times n$ matrix $B$ and whose lower part is an $m\times
n$ matrix that encodes the coefficient tuple via
\begin{equation}\label{eq 20}
y_k = \prod_{i=1}^{m} u_i^{b_{(n+i)k}}.
\end{equation}  
Then the mutation of the coefficient tuple in equation (\ref{eq:y-mutation}) 
is determined by the mutation
of the matrix $\tilde B$ in equation (\ref{eq:matrix-mutation}) and the formula (\ref{eq 20}); and the
exchange relation (\ref{exchange relation}) becomes
\begin{equation}\label{geometric exchange}
 x_k'=x_k^{-1} \left( \prod_{i=1}^n x_i^{[b_{ik}]_+}
\prod_{i=1}^{m} u_i^{[b_{(n+i)k}]_+} 
+\prod_{i=1}^n x_i^{[-b_{ik}]_+}
\prod_{i=1}^{m} u_i^{[-b_{(n+i)k}]_+}
\right).
\end{equation}  
\end{remark}

\subsection{Finite type and finite mutation type classification}

We say that a cluster algebra is of {\it finite type} if it has finitely many seeds.
Fomin and Zelevinsky \cite{FZ2} showed that the classification of finite type cluster algebras
is parallel to the Cartan-Killing classification of complex simple Lie algebras.

More specifically, define the \emph{diagram} $\Gamma(B)$
associated to an $n \times n$ exchange matrix $B$ to
be a weighted directed graph on $n$ nodes $v_1,\dots,v_n$,
where $v_i$ is directed towards $v_j$ if and only if
$b_{ij}>0$.  In that case, we label this edge  by
the weight $| b_{ij} b_{ji}|.$ 
It was shown in \cite{FZ2} that  a cluster algebra
$\Acal = \Acal(\xx, \yy, B)$ is of finite type if and only
$\Gamma(B)$ is mutation-equivalent to an orientation of a finite type Dynkin diagram.
In this case, we say that $B$ and $\Gamma(B)$ are of {\it finite type.}

We say that a matrix $B$ (and the corresponding cluster algebra) has {\it finite mutation type} if its mutation
equivalence class is finite, i.e. only finitely many matrices can be obtained
from $B$ by repeated matrix mutations.
A classification of all cluster algebras of finite mutation 
type with skew-symmetric exchange matrices was recently given by Felikson, Shapiro, and Tumarkin \cite{FeSTu}.
In particular, all but 11 of them come from either cluster algebras of rank 2 or 
cluster algebras associated with {\it triangulations of surfaces} (see Section \ref{sect surfaces}).

\subsection{Cluster algebras with principal
    coefficients}\label{sect principal coefficients}

Fomin and Zelevinsky introduced in \cite{FZ4} a special type of
coefficients, called \emph{principal coefficients}.

\begin{Def}
[\emph{Principal coefficients}]
\label{def:principal-coeffs}
We say that a cluster pattern $t \mapsto (\xx_t, \yy_t,B_t)$ on $\TT_n$
(or the corresponding cluster algebra~$\Acal$)  has
\emph{principal coefficients at a vertex~$t_0$} if
$\PP= \Trop(y_1, \dots, y_n)$ and
$\yy_{t_0}= (y_1, \dots, y_n)$. \linebreak[3]
In this case, we denote $\Acal=\Aprin(B_{t_0})$.
\end{Def}

\begin{remark}
\label{rem:principal-tildeB}
Definition~\ref{def:principal-coeffs} can be rephrased
as follows: a cluster algebra~$\Acal$ has principal coefficients at a
vertex~$t_0$ if $\Acal$ is of geometric type,
and is associated with the matrix $\tilde B_{t_0}$ of order $2n \times n$
whose upper part is $B_{t_0}$,
and whose complementary (i.e., bottom) part is
the $n \times n$ identity matrix (cf.\ \cite[Corollary~5.9]{FZ1}).
\end{remark}
 
\begin{Def}
[\emph{The functions $X_{\ell; t}$ and  $F_{\ell,t}$}]
\label{def:Aprin}
Let~$\Acal$ be the cluster \linebreak[3]
algebra with principal coefficients at 
$t_0$, defined by the initial seed
$\Sigma_{t_0}=(\xx_{t_0}\,,\yy_{t_0}\,,B_{t_0})$ with
\begin{equation}
\label{eq:initial-seed}
\xx_{t_0} = (x_1, \dots, x_n), \quad
\yy_{t_0} = (y_1, \dots, y_n), \quad
B_{t_0} = B^0 = (b^0_{ij})\,.
\end{equation}
By the Laurent phenomenon, we
can express every cluster variable $x_{\ell;t}$ as a (unique)
Laurent polynomial in $x_1, \dots, x_n, y_1, \dots, y_n$; 
we denote this by 
\begin{equation}
\label{eq:X-sf}
X_{\ell;t} 
= X_{\ell;t}^{B^0;t_0}.
\end{equation}
Let $F_{\ell;t} =
F_{\ell;t}^{B^0;t_0}$
denote the Laurent polynomial obtained from $X_{\ell;t}$ by
\begin{equation}
\label{eq:F-def}
F_{\ell;t}(y_1, \dots, y_n) = X_{\ell;t}(1, \dots, 1; y_1, \dots, y_n).
\end{equation}
$F_{\ell;t}(y_1,\dots,y_n)$ turns out to be
a polynomial \cite{FZ4} and is called an \emph{F-polynomial}.
\end{Def}

Knowing the cluster expansions for a cluster algebra with principal
coefficients allows one to compute the cluster expansions for the
``same'' cluster algebra with an arbitrary coefficient system. 
To explain this, we need an additional notation.
If $F$ is a subtraction-free rational expression over $\QQ$
in several variables, $R$ a semifield,
and $u_1,\dots,u_r$ some elements of~$R$,
then we denote by $F|_R(u_1,\dots,u_r)$ the evaluation of $F$ at
$u_1,\dots,u_r\,$.

\begin{theorem}\cite[Theorem 3.7]{FZ4}
\label{th:reduction-principal}
Let $\Acal$ be a cluster algebra over an arbitrary semifield $\PP$
and contained in the ambient field $\Fcal$,
with a seed at an initial vertex $t_0$ given by 
$$((x_1, \dots, x_n), (y_1^*, \dots, y_n^*), B^0).$$
Then the cluster variables in~$\Acal$ can be expressed as follows:
\begin{equation}
\label{eq:xjt-reduction-principal}
x_{\ell;t} = \frac{X_{\ell;t}^{B^0;t_0}|_\Fcal (x_1, \dots, x_n;y_1^*, \dots, y_n^*)}
{F_{\ell;t}^{B^0;t_0}|_\PP (y_1^*, \dots, y_n^*)} \, .
\end{equation}
\end{theorem}

When $\PP$ is a tropical semifield,  the denominator of equation 
(\ref{eq:xjt-reduction-principal}) is a monomial.  
Therefore if 
the Laurent polynomial $X_{\ell;t}$ has positive coefficients, so does
$x_{\ell;t}$.

\begin{cor}
Let~$\Acal$ be the cluster 
algebra with principal coefficients at a
vertex~$t_0$, defined by the initial seed
$\Sigma_{t_0}=(\xx_{t_0}\,,\yy_{t_0}\,,B_{t_0})$.
Let $\hat{\Acal}$ be any cluster algebra of geometric type
defined by the 
same  exchange matrix $B_{t_0}$.
If the positivity conjecture
holds for $\Acal$, then it also holds for $\hat{\Acal}$.
\end{cor}

\section{Cluster algebras arising from 
    surfaces}\label{sect surfaces} 

Building on work of Fock and Goncharov \cite{FG1, FG2}, and of 
Gekhtman, Shapiro and Vainshtein \cite{GSV}, 
Fomin, Shapiro and Thurston \cite{FST} associated a cluster algebra
to any {\it bordered surface with marked points}.  In 
this section we will recall that construction, as well
as further results of Fomin and Thurston \cite{FT}.

\begin{Def}
[\emph{Bordered surface with marked points}]
Let $S$ be a connected oriented 2-dimensional Riemann surface with
(possibly empty)
boundary.  Fix a nonempty set $M$ of {\it marked points} in the closure of
$S$ with at least one marked point on each boundary component. The
pair $(S,M)$ is called a \emph{bordered surface with marked points}. Marked
points in the interior of $S$ are called \emph{punctures}.  
\end{Def}
 
Table \ref{table 1} gives examples of surfaces
(using notation of Lemma \ref{Ideal-Tri}).
For technical reasons, we require that $(S,M)$ is not
a sphere with one, two or three punctures;
a monogon with zero or one puncture; 
or a bigon or triangle without punctures.
\begin{table}
\begin{center}
  \begin{tabular}{ c | c | c | c | c || l  }
  \  b\ \  &\ \  g \ \   & \ \  c \ \ &\ \ p\ \ &n\ \    &\  surface \\ \hline
    0 & 1 & 0 & 1 &3&\ punctured torus\\
    0 & 0 & 0 & 4 & 6&\ sphere with 4 punctures\\\hline
    1 & 0 & n+3 &0& c-3& \ polygon \\ 
    1 & 0 & n &1& c&\ once punctured polygon \\ 
    1 & 0 & n-3 &2& c+3&\ twice punctured polygon \\ 
    1 & 1 & n-3 & 0&c+3&\ torus with disk removed \\
\hline 
    2 & 0 & n & 0&c&\ annulus\\
    2 & 0 & n-3 & 1& c+3&\ punctured annulus\\
    2 & 1 & n-6 & 0& c+6& \ torus with 2 disks removed \\ 
     \hline
    3 & 0 & n-3 & 0& c+3&\ pair of pants \\ \\
  \end{tabular}
\end{center}
\caption{Examples of  surfaces}\label{table 1}
\end{table}
 
\subsection{Ideal triangulations and tagged triangulations}

\begin{definition}
[\emph{Ordinary arcs}]
An \emph{arc} $\zg$ in $(S,M)$ is a curve in $S$, considered up
to isotopy, such that 
\begin{itemize}
\item[(a)] the endpoints of $\zg$ are in $M$;
\item[(b)] $\zg$ does not cross itself, except that its endpoints may coincide;
\item[(c)] except for the endpoints, $\zg$ is disjoint from $M$ and
  from the boundary of $S$,
\item[(d)] $\zg$ does not cut out an unpunctured monogon or an unpunctured bigon. 
\end{itemize}   
\end{definition}     

An arc whose endpoints coincide is called a {\it loop}.
Curves that connect two
marked points and lie entirely on the boundary of $S$ without passing
through a third marked point are \emph{boundary segments}.
By (c), boundary segments are not ordinary arcs.

\begin{Def}
[\emph{Crossing numbers and compatibility of ordinary arcs}]
For any two arcs $\zg,\zg'$ in $S$, let $e(\zg,\zg')$ be the minimal
number of crossings of 
arcs $\za$ and $\za'$, where $\za$ 
and $\za'$ range over all arcs isotopic to 
$\zg$ and $\zg'$, respectively.
We say that arcs $\zg$ and $\zg'$ are  \emph{compatible} if $e(\zg,\zg')=0$. 
\end{Def}

\begin{Def}
[\emph{Ideal triangulations}]
An \emph{ideal triangulation} is a maximal collection of
pairwise compatible arcs (together with all boundary segments). 
The arcs of a 
triangulation cut the surface into \emph{ideal triangles}. 
\end{Def}

\begin{lemma}\cite[Section 2]{FG3}
\label{Ideal-Tri}
Each ideal triangulation consists of 
$n=6g+3b+3p+c-6$ arcs,  where $g$ is the
genus of $S$, $b$ is the number of boundary components, $p$ is the number of punctures and $c=|M|-p$ is the
number of marked points on the boundary of $S$. 
The number $n$ is called the \emph{rank} of $(S,M)$.
The number of triangles in any triangulation is 
$4g+2b+2p+c-4.$
\end{lemma}

There are two types of ideal triangles: triangles that have three distinct sides and triangles that have only two. The latter are called \emph{self-folded} triangles.  Note that a self-folded triangle consists of 
a loop $\ell$, together with an arc $r$ to an enclosed puncture which we dub a 
\emph{radius}, see the left side of Figure \ref{figtags}.

\begin{Def}
[\emph{Ordinary flips}]
Ideal triangulations are connected to each other by sequences of 
{\it flips}.  Each flip replaces a single arc $\gamma$ 
in a triangulation $T$ by a (unique) arc $\gamma' \neq \gamma$
that, together with the remaining arcs in $T$, forms a new ideal
triangulation.
\end{Def}
Note that an arc $\gamma$ that lies inside a self-folded triangle
in $T$ cannot be flipped.

In \cite{FST}, the authors associated a cluster algebra to 
any bordered surface with marked points.  Roughly speaking,
the cluster variables correspond to arcs, the clusters
to triangulations, and the mutations to flips.  However,
because arcs inside self-folded triangles cannot be flipped,
the authors were led to introduce the slightly more general notion 
of {\it tagged arcs}.  They showed that ordinary arcs can
all be represented by tagged arcs and gave a notion of flip 
that applies to all tagged arcs.

\begin{Def}
[\emph{Tagged arcs}]
A {\it tagged arc} is obtained by taking an arc that does not 
cut out a once-punctured monogon and marking (``tagging")
each of its ends in one of two ways, {\it plain} or {\it notched},
so that the following conditions are satisfied:
\begin{itemize}
\item an endpoint lying on the boundary of $S$ must be tagged plain
\item both ends of a loop must be tagged in the same way.
\end{itemize}
\end{Def}

\begin{Def}
[\emph{Representing ordinary arcs by tagged arcs}]
One can represent an ordinary arc $\beta$ by 
a tagged arc $\iota(\beta)$ as follows.  If $\beta$ 
does not cut out a once-punctured monogon, then $\iota(\beta)$
is simply $\beta$ with both ends tagged plain.
Otherwise, $\beta$ is a loop based at some marked point $a$
and cutting out
a punctured monogon with the sole puncture $b$ inside it.
Let $\alpha$ be the unique arc connecting $a$ and $b$ and compatible
with $\beta$.  Then $\iota(\beta)$
is obtained by tagging $\alpha$ plain at $a$ and notched at $b$.
\end{Def}

\begin{Def}
[\emph{Compatibility of tagged arcs}]  \label{compatible}
Tagged arcs $\alpha$ and
$\beta$ are called {\it compatible} if and only if the following 
properties hold:
\begin{itemize}
\item the arcs $\alpha^0$ and $\beta^0$ obtained from 
   $\alpha$ and $\beta$ by forgetting the taggings are compatible; 
\item if $\alpha^0=\beta^0$ then at least one end of $\alpha$
  must be tagged in the same way as the corresponding end of $\beta$;
\item $\alpha^0\neq \beta^0$ but they share an endpoint $a$, 
 then the ends of $\alpha$ and $\beta$ connecting to $a$ must be tagged in the 
same way.
\end{itemize}
\end{Def}

\begin{Def}
[\emph{Tagged triangulations}]  A maximal (by inclusion) collection
of pairwise compatible tagged arcs is called a {\it tagged triangulation}.
\end{Def}

Figure \ref{figtags} gives an example of an ideal triangulation 
$T$ and the corresponding tagged triangulation $\iota(T)$. The notching is 
indicated by a 
bow tie.

\begin{figure}
\input{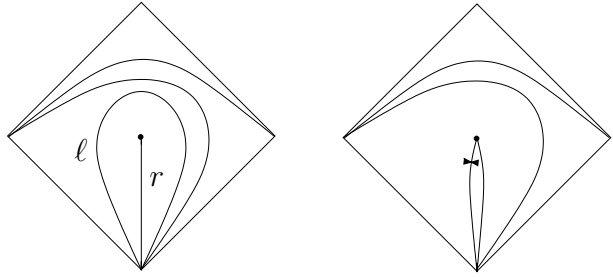}
\caption{Example of an ideal triangulation on the left and the corresponding  tagged triangulation on the right}\label{figtags}
\end{figure}

\subsection{From  surfaces to cluster algebras}
One can  associate an exchange
matrix and
hence a cluster algebra to any bordered  surface $(S,M)$
\cite{FST}.

\begin{Def}
[\emph{Signed adjacency matrix of an ideal triangulation}]
Choose any ideal triangulation
$T$, and let $\tau_1,\tau_2,\ldots,\tau_n$ be the $n$ arcs of
$T$.
For any triangle $\Delta$ in $T$ which is not self-folded, we define a matrix 
$B^\Delta=(b^\Delta_{ij})_{1\le i\le n, 1\le j\le n}$  as follows.
\begin{itemize}
\item $b_{ij}^\Delta=1$ and $b_{ji}^{\Delta}=-1$ in the following cases:
\begin{itemize}
\item[(a)] $\tau_i$ and $\tau_j$ are sides of 
  $\Delta$ with  $\tau_j$ following $\tau_i$  in the 
  clockwise order;
\item[(b)] $\tau_j$ is a radius in a self-folded triangle enclosed by a loop $\tau_\ell$, and $\tau_i$ and $\tau_\ell$ are sides of 
  $\Delta$ with  $\tau_\ell$ following $\tau_i$  in the 
clockwise order;
\item[(c)] $\tau_i$ is a radius in a self-folded triangle enclosed by a loop $\tau_\ell$, and $\tau_\ell$ and $\tau_j$ are sides of 
  $\Delta$ with  $\tau_j$ following $\tau_\ell$  in the 
clockwise order;
\end{itemize}
\item $b_{ij}^\Delta=0$ otherwise.
\end{itemize}
 
Then define the matrix 
$ B_{T}=(b_{ij})_{1\le i\le n, 1\le j\le n}$  by
$b_{ij}=\sum_\Delta b_{ij}^\Delta$, where the sum is taken over all
triangles in $T$ that are not self-folded. 
\end{Def}

Note that  $B_{T}$ is skew-symmetric and each entry  $b_{ij}$ is either
$0,\pm 1$, or $\pm 2$, since every arc $\tau$ is in at most two triangles. 

\begin{remark}\label{tagged-ideal}
As noted in \cite[Definition 9.2]{FST}, compatibility of tagged arcs is 
invariant with respect to a simultaneous change of all tags at a given
puncture.  So given a tagged triangulation $T$, let us perform such changes
at every puncture where all ends of $T$ are notched.  The resulting 
tagged triangulation $\hat{T}$ represents an ideal triangulation $T^0$
(possibly containing self-folded triangles): $\hat{T} = \iota(T^0)$.  This is
because the only way for a puncture $p$ to have two incident
arcs with two different taggings at $p$ is 
for those two arcs to be homotopic, see Definition \ref{compatible}.
But then for this to lie in some tagged triangulation,
it follows that $p$ must be a puncture in the interior of a bigon.  
See Figure \ref{figtags}.
\end{remark}

\begin{Def}
[\emph{Signed adjacency matrix of a tagged triangulation}] \label{GeneralB}
The signed adjacency matrix $B_T$ of a tagged triangulation $T$
is defined to be the signed adjacency matrix $B_{T^0}$, where 
$T^0$ is obtained from $T$ as in Remark \ref{tagged-ideal}.  
The index sets of the matrices (which correspond to tagged and ideal arcs,
respectively) are identified in the obvious way.
\end{Def}

\begin{theorem} \cite[Theorem 7.11]{FST} and \cite[Theorem 5.1]{FT}
\label{clust-surface}
Fix a bordered surface $(S,M)$ and let $\Acal$ be the cluster algebra associated to 
the signed adjacency matrix of a tagged triangulation as in Definition
\ref{GeneralB}.
Then the (unlabeled) seeds $\Sigma_{T}$ of $\Acal$ are in bijection 
with tagged triangulations $T$ of $(S,M)$, and 
the cluster variables are  in bijection
with the tagged arcs of $(S,M)$ (so we can denote each by 
$x_{\gamma}$, where $\gamma$ is a tagged arc).  Furthermore,
if a tagged triangulation $T'$ is obtained from another
tagged triangulation $T$ by flipping a tagged arc $\gamma\in T$
and obtaining $\gamma'$,
then $\Sigma_{T'}$ is obtained from $\Sigma_{T}$ by the seed mutation
replacing $x_{\gamma}$ by $x_{\gamma'}$.
\end{theorem}

\begin{remark}\label{ordinary-tagged}
By a slight abuse of notation, if $\gamma$ is an ordinary arc
which does not cut out a once-punctured monogon
(so that the tagged arc $\iota(\gamma)$ is obtained
from $\gamma$ by tagging both ends plain),  we will 
often write $x_{\gamma}$ instead of $x_{\iota(\gamma)}$.
\end{remark}

Given a surface $(S,M)$ with a puncture $p$ and a tagged arc $\gamma$,
we let both $\gamma^{(p)}$ and $\gamma^{p}$ 
denote the arc obtained from $\gamma$
by changing its notching at $p$.  (So if $\gamma$ is not incident to 
$p$, $\gamma^{(p)} = \gamma$.)  
If $p$ and $q$ are two punctures, we let 
$\gamma^{(pq)}$ denote the arc obtained from $\gamma$ by changing
its notching at both $p$ and $q$.
Given a tagged triangulation $T$
of $S$, we let $T^{p}$ denote the tagged triangulation obtained
from $T$ by replacing each $\gamma \in T$ by $\gamma^{(p)}$.

Besides labeling cluster variables of 
$\Acal(B_T)$ by $x_{\tau}$, where $\tau$ is a tagged arc of 
$(S,M)$, we will also make the following conventions:
\begin{itemize}
\item If $\ell$ 
is an unnotched loop with endpoints at $q$ cutting out a once-punctured monogon containing 
puncture $p$ and radius $r$, 
then we set 
$x_{\ell}= 
x_{r}x_{r^{(p)}}$.\footnote{There
is a corresponding 
statement on the level of {\it lambda lengths} of these three
arcs, see \cite[Lemma 7.2]{FT};  these conventions are compatible with both the Ptolemy
relations and the 
exchange relations among 
cluster variables \cite[Theorem 7.5]{FT}.}  
\item If $\beta$ is a boundary
segment, we set $x_{\beta} = 1$.
\end{itemize}

To prove the positivity conjecture for a cluster algebra associated 
to a surface, we must show that the Laurent expansion of each 
cluster variable with respect to {\it any cluster}
is positive.  The next result will allow us to restrict our
attention to clusters that correspond to 
ideal triangulations.

\begin{prop}\label{tag-change}
Fix $(S,M)$, $p$, $\gamma$,  $T=(\tau_1,\dots,\tau_n)$,
and 
$T^p = (\tau_1^p,\dots,\tau_n^p)$ as above.
Let 
$\Acal = \Aprin(B_T)$, and $\Acal^p = \Aprin(B_{T^p})$ be the cluster
algebras with principal coefficients  at the vertices
$\Sigma_T = (\xx,\yy, B_T)$ and 
$\Sigma_{T^p}=(\xx^p,\yy^p,B_{T^p})$, where 
$\xx = 
\{x_{\tau_i}\}$, $\yy=\{y_{\tau_i}\}$,
$\xx^p = \{x_{\tau_i^p}\}$, and $\yy^p = \{y_{\tau_i^p}\}.$
Then 
$$[x_{\gamma^{p}}]_{\Sigma_{T^p}}^{\Acal^p} = 
[x_{\gamma}]_{\Sigma}^{\Acal}\mid_{x_{\tau_i} \leftarrow x_{\tau_i^{p}},\ y_{\tau_i} \leftarrow y_{\tau_i^p}}.$$
That is, the cluster expansion of 
$x_{\gamma^{p}}$ with respect to  
$\xx^{p}$ in $\Acal^p$ is obtained from 
the cluster expansion of $x_{\gamma}$ with respect to $\xx$ in $\Acal$ by 
substituting 
 $x_{\tau_i}=
x_{\tau_i^{p}}$ and
$y_{\tau_i} = 
y_{\tau_i^{p}}.$ 
\end{prop}

\begin{proof}
By Definition \ref{GeneralB}, the rectangular exchange matrix 
$\tilde{B}_T$ is equal to $\tilde{B}_{T^{(p)}}$.
The columns of $\tilde{B}_T$ are indexed by $\{x_{\tau_i}\}$
and the columns of $\tilde{B}_T^p$ are indexed by 
$\{x_{\tau_i^p}\}$; the rows of $\tilde{B}_T$ are indexed by 
$\{x_{\tau_i}\} \cup \{y_{\tau_i}\}$ and the rows of $\tilde{B}_T^p$ are indexed by 
$\{x_{\tau_i^p}\} \cup \{y_{\tau_i^p}\}$. 

To compute the $\xx$-expansion of $x_{\gamma}$,  
we write down a sequence of flips $(i_1,\dots, i_r)$
(here $1 \leq i_j \leq n$) which transforms 
$T$ into a tagged triangulation $T'$
containing $\gamma$.
Applying the corresponding exchange relations then 
gives the $\xx$-expansion of $x_{\gamma}$ in $\Acal$.
By the description of tagged flips (\cite[Remark 4.13]{FT}),
performing the same sequence of flips on $T^p$  transforms 
$T^p$ into the tagged triangulation $T'^{p}$, which in particular
contains $\gamma^p$.  Therefore applying the corresponding 
exchange relations gives the $\xx^p$-expansion of 
$x_{\gamma^p}$ in $\Acal^p$.  

Since in both cases we start from the same exchange matrix and apply
the same sequence of mutations, 
the $\xx^p$-expansion of 
$x_{\gamma^p}$ in $\Acal^p$ will be equal to the  
$\xx$-expansion of $x_{\gamma}$ in $\Acal$ after relabeling variables, 
i.e. after substituting
 $x_{\tau_i}=
x_{\tau_i^{p}}$ and
$y_{\tau_i} = 
y_{\tau_i^{p}}.$ 

\end{proof}

\begin{cor}\label{reduce-to-ideal}
Fix a bordered surface $(S,M)$ and let $\Acal$ be the corresponding
cluster algebra.  Let $T$ be an arbitrary tagged triangulation.
To prove the positivity conjecture for 
$\Acal$ with respect to $\xx_{T}$,
it suffices to prove
positivity with respect to clusters of the form $\xx_{\iota(T^0)}$,
where $T^0$ is an ideal triangulation.
\end{cor}

\begin{proof}
As in Remark \ref{tagged-ideal},
we can perform simultaneous tag-changes at punctures to pass
from an arbitrary tagged triangulation $T$ to a tagged triangulation
$\hat{T}$ representing an ideal triangulation.  By a repeated application
of Proposition \ref{tag-change} -- which preserves positivity 
because it just involves a substitution of variables -- 
we can then express cluster expansions with respect to 
$\xx_T$ in terms of cluster expansions with respect to $\xx_{\hat{T}}$.
\end{proof}

The exchange relation corresponding to a flip in an ideal triangulation
is called 
a {\it generalized Ptolemy relation}.  It can be described as 
follows.
\begin{prop}\cite{FT}\label{Ptolemy}
Let $\alpha, \beta, \gamma, \delta$ be arcs (including loops) 
or boundary segments 
of $(S,M)$ which cut out a quadrilateral; we assume that the sides
of the quadrilateral, listed in cyclic order, are
$\alpha, \beta, \gamma, \delta$.  Let $\eta$ and $\theta$ 
be the two diagonals of this quadrilateral; see the left-hand-side of
Figure \ref{figflip}.
Then 
$$x_{\eta} x_{\theta} = Y x_{\alpha} x_{\gamma} + Y' x_{\beta} x_{\delta}$$
for some coefficients $Y$ and $Y'$.
\end{prop}

\begin{proof}
This follows from the interpretation of cluster variables as 
{\it lambda lengths}  and the 
Ptolemy relations for lambda lengths \cite[Theorem 7.5 and Proposition 6.5]{FT}.
\end{proof}

Note that 
some sides of the quadrilateral in Proposition \ref{Ptolemy}
may be glued to each other, changing the appearance of the relation.
There are also generalized Ptolemy relations for tagged triangulations,
see \cite[Definition 7.4]{FT}.
\begin{figure} \begin{center}
\scalebox{.8}{\input{figflip.pstex_t}} \hspace{6em} \scalebox{.8}{\input{figsz.pstex_t}}
\end{center}
\caption{}
\label{figflip}
\end{figure}

\subsection{Keeping track of coefficients using laminations}

So far we have not addressed the topic of 
coefficients for cluster algebras arising from bordered surfaces.
It turns out that 
W. Thurston's theory of measured laminations gives a 
concrete way to think about coefficients, 
as described in \cite{FT} (see also \cite{FG3}).

\begin{Def}
[\emph{Laminations}] 
A {\it lamination} on a bordered surface $(S,M)$ is a finite collection
of non-self-intersecting and pairwise non-intersecting curves in $S$,
modulo isotopy relative to $M$, subject to the following restrictions.
Each curve must be one of the following:
\begin{itemize}
\item a closed curve;
\item a curve connecting two unmarked points on the boundary of $S$;
\item a curve starting at an unmarked point on the boundary and,
at its other end, spiraling into a puncture (either clockwise
or counterclockwise);
\item a curve whose ends both spiral into punctures (not necessarily
distinct).
\end{itemize}
Also, we forbid curves that bound an unpunctured or once-punctured
disk, and curves with two endpoints on the boundary of $S$ which
are isotopic to a piece of boundary containing zero or one
 marked point.
\end{Def}

In \cite[Definitions 12.1 and 12.3]{FT}, the authors define shear coordinates and 
extended exchange matrices, with 
respect to a tagged triangulation.  For our purposes, it will be enough 
to make these definitions with respect to an ideal triangulation.

\begin{Def}
[\emph{Shear coordinates}]
Let $L$ be a lamination, and let $T$ be an ideal triangulation.
For each arc $\gamma \in T$, 
the corresponding {\it shear coordinate} of $L$ with respect to $T$, 
denoted by $b_{\gamma}(T,L)$, is defined as a sum of contributions
from all intersections of curves in $L$ with $\gamma$.
Specifically, such an intersection contributes $+1$ (resp., $-1$)
to $b_{\gamma}(T,L)$ if the corresponding segment of a curve in 
$L$ cuts through the quadrilateral surrounding $\gamma$
as shown in Figure \ref{figflip} in the middle (resp., right).
\end{Def}

\begin{Def}
[\emph{Multi-laminations and associated extended exchange matrices}]
A {\it multi-lamination} is a finite family of laminations.
Fix a multi-lamination $\mathbf{L}=(L_{n+1}, \dots, L_{n+m})$.  
For an ideal triangulation $T$ of $(S,M)$,
define the matrix $\tilde{B} = \tilde{B}(T,\mathbf{L}) = (b_{ij})$
as follows.  The top $n \times n$ part of $\tilde{B}$ is the signed
adjacency matrix $B(T)$, with rows and columns indexed by arcs
$\gamma \in T$ (or equivalently, by the tagged arcs
$\iota(\gamma) \in \iota(T)$).  The bottom $m$ rows
are formed by the shear coordinates of the laminations $L_i$
with respect to $T$:
$$b_{n+i,\gamma} = b_{\gamma}(T, L_{n+i}) \text{ if }1 \leq i \leq m.$$
\end{Def}

By  \cite[Theorem 11.6]{FT},
the matrices
$\tilde{B}(T)$ transform compatibly with mutation.

\begin{figure} \begin{center}
\input{figlaminate.pstex_t}
\end{center}
\caption{}
\label{laminate}
\end{figure}

\begin{Def}
[\emph{Elementary lamination associated with a tagged arc}]
Let $\gamma$ be a tagged arc in $(S,M)$.  Denote by 
$L_{\gamma}$ a lamination consisting of a single curve
defined as follows.  The curve $L_{\gamma}$ runs along 
$\gamma$ within a small neighborhood of it.  If $\gamma$
has an endpoint $a$ on a (circular) component $C$ of the boundary of $S$,
then $L_{\gamma}$ begins at a point $a'\in C$ located near $a$ in 
the counterclockwise direction, and proceeds along $\gamma$
as shown in Figure \ref{laminate} on the left.
If $\gamma$ has an endpoint at a puncture, then $L_{\gamma}$
spirals into $a$: counterclockwise if $\gamma$ is tagged plain
at $a$, and clockwise if it is notched.  
\end{Def}

The following result comes from \cite[Proposition 16.3]{FT}.

\begin{prop}
Let $T$ be an ideal triangulation with a signed adjacency matrix
$B(T)$. Recall that we can view $T$  as a tagged triangulation $\iota(T)$. 
Let $L_T = (L_{\gamma})_{\gamma \in \iota(T)}$
be the multi-lamination consisting of elementary laminations associated with 
the tagged arcs in $\iota(T)$.
Then the cluster algebra with principal coefficients
$\Aprin(B(T))$ is isomorphic to 
$\mathcal A(\tilde{B}(T, L_T))$.
\end{prop}

\section{Main results: cluster expansion formulas}\label{sect main}

In this section we present cluster expansion formulas 
for all cluster variables in a cluster algebra associated to a bordered
surface, with respect to a seed corresponding to an ideal triangulation; by
Proposition \ref{tag-change} and Corollary \ref{reduce-to-ideal}, 
one can use these formulas to compute cluster expansion formulas with respect to an {\it arbitrary} 
seed by an appropriate substitution of variables.
All of our formulas are manifestly positive, so this proves the positivity conjecture
for any cluster algebras associated to a bordered surface.  Moreover, since our 
formulas involve the system of principal coefficients, this proves positivity for
any such cluster algebra of geometric type.
 
We present three slightly different formulas, based on whether the 
cluster variable corresponds to a tagged arc with $0$, $1$, or $2$
notched ends.
More specifically, fix an ordinary arc $\gamma$ and 
a tagged triangulation $T=\iota(T^\circ)$ of $(S,M)$, where $T^{\circ}$ is an ideal triangulation. 
We recursively construct an edge-weighted graph $G_{T^{\circ},\gamma}$ 
by glueing together {\it tiles} based on the local configuration of the intersections between  $\gamma$ and $T^{\circ}$.
Our formula 
(Theorem \ref{thm main}) for $x_{\gamma}$ with respect to  $\Sigma_{T}$ 
is given in terms of perfect matchings of $G_{T^{\circ},\gamma}$.
This formula also holds for the cluster algebra element 
$x_{\ell} = x_r x_{r^{(p)}}$, where 
$\ell$ is a loop  cutting out a once-punctured monogon 
enclosing the puncture $p$ and radius $r$.  
In the case of $\gamma^{(p)}$, an arc between points $p$ and $q$ with a single notch at $p$, 
we build the graph $G_{T^{\circ}, \ell_p}$ associated to the loop
$\ell_p$  
such that $\iota(\ell_p) = \gamma^{(p)}$.  Our combinatorial
formula for $x_{\gamma^{(p)}}$ is then in terms of 
the so-called {\it $\gamma$-symmetric} matchings
of $G_{T^{\circ}, \ell_p}$.
In the case of $\gamma^{(pq)}$, an arc between points $p$ and $q$ which is notched
at both $p$ and $q$, we build the two graphs 
$G_{T^{\circ}, \ell_p}$ and 
$G_{T^{\circ}, \ell_q}$ associated to $\ell_p$ and 
$\ell_q$.  Our combinatorial 
formula for $x_{\gamma^{(pq)}}$ is then in terms of 
the  {\it $\gamma$-compatible pairs} of matchings
of $G_{T^{\circ}, \ell_p}$.
and $G_{T^{\circ}, \ell_q}$.

\subsection{Tiles} \label{sect tiles}
Let $T^\circ$ be an ideal triangulation of a bordered surface $(S,M)$ and let
$\zg$ be an ordinary arc in $(S,M)$ which is not in $T^\circ$. 
Choose an orientation on $\zg$, let $s\in M$ be its starting point, and let $t\in M$ be its endpoint.
We denote by
\[ s=p_0, p_1, p_2, \ldots, p_{d+1}=t
\]
the points of intersection of $\zg$ and $T^\circ$ in order.  
Let $\tau_{i_j}$ be the arc of $T^{\circ}$ containing $p_j$, and let 
$\zD_{j-1}$ and 
$\zD_{j}$ be the two ideal triangles in $T^{\circ}$ 
on either side of 
$\tau_{i_j}$. 

To each $p_j$ we associate a  \emph{tile} $G_j$,  
an edge-labeled triangulated quadrilateral (see the right-hand-side of Figure \ref{digoncross}), 
which is defined to be
the union of two edge-labeled triangles $\zD_1^j$ and $\zD_2^j$ glued at an edge labeled $\tau_{i_j}$. 
The triangles $\zD_1^j$ and $\zD_2^j$ are determined by  
$\zD_{j-1}$ and 
$\zD_{j}$ as follows.

If neither $\zD_{j-1}$ nor $\zD_{j}$ is self-folded, then they each have three distinct sides 
(though possibly fewer than three vertices), and we define $\zD_1^j$ and $\zD_2^j$ to be the ordinary triangles with edges labeled as in $\zD_{j-1}$ and $\zD_{j}$.  We glue $\zD_1^j$ and $\zD_2^j$ at the edge labeled $\tau_{i_j}$, so that the orientations of $\zD_1^j$ and $\zD_2^j$ both either agree or disagree with those of $\zD_{j-1}$ and $\zD_j$; this gives two possible planar embeddings
of a graph $G_j$ which we call an \emph{ordinary tile}.
\begin{figure}
\input{figTileEgs.pstex_t}
\caption{}
\label{digoncross}
\end{figure}

If one of $\zD_{j-1}$ or $\zD_{j}$ is  self-folded, then in fact $T^\circ$ must have a local configuration of a bigon (with sides $a$ and $b$) containing a radius $r$ incident to a puncture $p$ inscribed inside a loop $\ell$, see Figure \ref{triptile}.
Moreover, $\zg$ must either

\begin{itemize}
\item[(1)] intersect the loop $\ell$ and terminate at puncture $p$, or

\item[(2)] intersect the loop $\ell$, radius $r$ and then $\ell$ again.
\end{itemize}

In case (1), we associate to $p_j$ (the intersection point with  $\ell$) 
an {\it ordinary tile} $G_j$ consisting of a triangle with sides $\{a,b,\ell\}$ which 
is glued along diagonal $\ell$ to a triangle with sides $\{\ell,r,r\}$.
As before there are two possible planar embeddings of $G_j$.

In case (2), we associate to the triple of intersection points 
$p_{j-1}, p_j, p_{j+1}$ a union of tiles $G_{j-1} \cup G_j \cup G_{j+1}$,
which we call a \emph{triple tile},
based on whether $\zg$ enters and exits
through different sides of the bigon or through the same side.
These graphs are defined by Figure \ref{triptile} (each possibility
is denoted in boldface within a concatenation of five tiles).  Note that 
in each case there are two possible planar embeddings of the triple tile.
We call
the tiles $G_{j-1}$ and $G_{j+1}$ within the triple tile {\it ordinary tiles}.

\begin{figure}
\input{TripleTile.pstex_t}
\caption{}
\label{triptile}
\end{figure}

\begin{definition} [\emph{Relative orientation}]
Given a planar embedding $\tilde G_j$ 
of an ordinary tile $G_j$, we define the \emph{relative orientation} 
$\mathrm{rel}(\tilde G_j, T^\circ)$ 
of $\tilde G_j$ with respect to $T^\circ$ 
to be $\pm 1$, based on whether its triangles agree or disagree in orientation with those of $T^\circ$.  
\end{definition}
Note that in Figure \ref{triptile}, the lowest tile in each of the three graphs in the middle
(respectively, rightmost) column has relative orientation $+1$ (respectively, $-1$).
Also note that 
by construction, the planar embedding of a triple tile $\tilde G_{j-1} \cup \tilde G_j \cup \tilde G_{j+1}$ satisfies $\mathrm{rel}(\tilde G_{j-1},T^\circ) = \mathrm{rel}(\tilde G_{j+1},T^\circ)$.

\begin{definition}
Using the notation above, 
the arcs $\tau_{i_j}$ and $\tau_{i_{j+1}}$ form two edges of a triangle $\zD_j$ in $T^\circ$.  Define $\tau_{[\zg_j]}$ to be the third arc in this triangle if $\zD_j$ is not self-folded, and to be the radius in $\zD_j$ otherwise.
\end{definition}

\subsection{The graph ${G}_{T^\circ,\zg}$}\label{sect graph}

We now recursively glue together the tiles $G_1,\dots,G_d$
in order from $1$ to $d$, subject to the following conditions.
\begin{figure}
\input{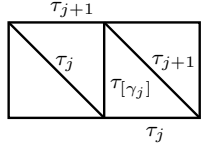}
\caption{Glueing tiles $\tilde G_j$ and $\tilde G_{j+1}$ along the edge labeled  $\tau_{[\zg_j]}$.}
\label{figglue}
\end{figure}
\begin{enumerate}
\item Triple tiles must stay glued together as  in Figure \ref{triptile}.
\item For two adjacent ordinary tiles, 
each of which may be an exterior tile of a triple tile, 
we glue  $G_{j+1}$ to $\tilde G_j$ along the edges 
labeled $\tau_{[\zg_j]}$, choosing a planar embedding $\tilde G_{j+1}$ for $G_{j+1}$
so that $\mathrm{rel}(\tilde G_{j+1},T^\circ) \not= \mathrm{rel}(\tilde G_j,T^\circ).$  See Figure \ref{figglue}.
\end{enumerate}

After glueing together the $d$ tiles, we obtain a graph (embedded in 
the plane),
which we denote 
$\overline{G}_{T^\circ,\zg}$.  Let $G_{T^\circ,\zg}$ be the graph obtained 
from $\overline{G}_{T^\circ,\zg}$ by removing the diagonal in each tile.
Figure \ref{triptile} gives examples of a 
dotted arc $\gamma$ and the corresponding graph 
$\overline{G}_{T^\circ,\zg}$.  Each $\gamma$ intersects $T^\circ$
five times,  so each 
$\overline{G}_{T^\circ,\zg}$ has five tiles. 

\begin{remark}
Even if $\gamma$ is a curve with self-intersections, 
our definition of $\overline{G}_{T^\circ,\gamma}$ makes sense.  
This is relevant to  our formula for the doubly-notched loop,
see Remark \ref{rem double}.
\end{remark}

\subsection{Cluster expansion formula for ordinary arcs}\label{sect cluster expansion formula}

Recall that if $\tau$ is a boundary segment then $x_{\tau} = 1$,
and if $\tau$ is a loop  cutting out a once-punctured
monogon with radius $r$ and puncture $p$, then $x_{\tau} = x_r x_{r^{(p)}}$.
Also see Remark \ref{ordinary-tagged}.
Before giving the next result, we need to introduce some
notation.

\begin{definition} [\emph{Crossing Monomial}]
If $\zg$ is an ordinary arc and  $\tau_{i_1}, \tau_{i_2},\dots, \tau_{i_d}$
is the sequence of arcs in $T^\circ$ which $\zg$ crosses, we define the \emph{crossing monomial} of $\gamma$ with respect to $T^\circ$ to be
$$\mathrm{cross}(T^\circ, \gamma) = \prod_{j=1}^d x_{\tau_{i_j}}.$$
\end{definition}

\begin{definition} [\emph{Perfect matchings and weights}]
A \emph{perfect matching} of a graph $G$ is a subset $P$ of the 
edges of $G$ such that
each vertex of $G$ is incident to exactly one edge of $P$. 
If the edges of  a perfect matching $P$ of  
$G_{T^\circ,\zg}$ are labeled $\tau_{j_1},\dots,\tau_{j_r}$, then 
we define the {\it weight} $x(P)$ of $P$ to be 
$x_{\tau_{j_1}} \dots x_{\tau_{j_r}}$.
\end{definition}

\begin{definition} [\emph{Minimal and Maximal Matchings}]  
By induction on the number of tiles it is easy to see that 
$G_{T^\circ,\zg}$  
has  precisely two perfect matchings which we call
the {\it minimal matching} $P_-=P_-(G_{T^\circ,\zg})$ and 
the {\it maximal matching} $P_+
=P_+(G_{T^\circ,\zg})$, 
which contain only boundary edges.
To distinguish them, 
if  $\mathrm{rel}(\tilde G_1,T^\circ)=1$ (respectively, $-1$),
we define 
$e_1$ and $e_2$ to be the two edges of 
$\overline{G}_{T^\circ,\zg}$ which lie in the counterclockwise 
(respectively, clockwise) direction from 
the diagonal of $\tilde G_1$.  Then  $P_-$ is defined as
the unique matching which contains only boundary 
edges and does not contain edges $e_1$ or $e_2$.  $P_+$
is the other matching with only boundary edges.
\end{definition}

For an arbitrary perfect matching $P$ of $G_{T^\circ,\gamma}$, we let
$P_-\ominus P$ denote the symmetric difference, defined as $P_-\ominus P =(P_-\cup P)\setminus (P_-\cap P)$.

\begin{lemma}\cite[Theorem 5.1]{MS}
\label{thm y}
The set $P_-\ominus P$ is the set of boundary edges of a 
(possibly disconnected) subgraph $G_P$ of $G_{T^\circ,\zg}$, which is a union of 
cycles.  These cycles enclose a set of tiles 
$\cup_{j\in J} G_{i_j}$,  where $J$ is a finite index set.
\end{lemma}

We use this decomposition to define \emph{height monomials} for
perfect matchings.
Note that the exponents in the height monomials defined below coincide
with the definiton of height functions given in
\cite{ProppLattice} for perfect matchings of bipartite
graphs, based on earlier work of \cite{ConwayLagarias}, 
\cite{EKLP}, and \cite{Thurston} for domino tilings. 

\begin{definition} [\emph{Height Monomial and Specialized Height Monomial}] \label{height} Let $T^\circ = \{\tau_1,\tau_2,\dots, \tau_n\}$ be an ideal triangulation of $(S,M)$ and $\gamma$ be an ordinary arc of $(S,M)$.  By Lemma \ref{thm y}, for any perfect matching $P$
of $G_{T^\circ,\zg}$, $P \ominus P_-$ encloses the union of tiles $\cup_{j\in J} G_{i_j}$.  
We  define the \emph{height monomial} $h(P)$ of $P$ by
\begin{equation*}
h(P) = \prod_{k=1}^n h_{\tau_{k}}^{m_k},
\end{equation*}
where $m_k$ is the number of tiles in $\cup_{j\in J} G_{i_j}$ whose
diagonal is labeled $\tau_k$.

We  define the \emph{specialized height monomial} $y(P)$ of $P$ to be
the specialization $\Phi(h(P))$, where $\Phi$ is defined below. 
\begin{eqnarray*} 
\label{eqn yspec} \Phi({h}_{\tau_i}) &=& \left\{\begin{array}{ll}
y_{\tau_i}\
&\textup{if $\tau_i$ is not a side of a self-folded triangle;}\\ \\
\dfrac{y_{r}}{y_{r^{(p)}}}
&\textup{if $\tau_i$ is a radius $r$ to puncture $p$ in a self-folded triangle;}\\ \\
y_{r^{(p)}}
&\textup{if $\tau_i$ is a loop in a self-folded triangle with radius $r$ to puncture $p$.}
\end{array}\right.
\end{eqnarray*}
\end{definition}

\begin{theorem}\label{thm main}
Let $(S,M)$ be a bordered surface with an ideal triangulation
$T^\circ$, and 
let $T=\{\tau_1,\tau_2,\dots, \tau_n\}= \iota(T^\circ)$ be the 
corresponding tagged triangulation.
Let $\mathcal{A}$ be
the corresponding cluster algebra with principal coefficients with respect to  $\Sigma_T=(\mathbf{x}_T,\mathbf{y}_T,B_T)$, and let
$\zg$ 
be an ordinary arc in $S$ (this may include a loop cutting out a once-punctured monogon).
Let $G_{T^\circ,\zg}$ be the graph constructed in Section \ref{sect graph}.  Then the Laurent expansion of $x_{\zg}$ with respect to $\Sigma_T$ is given by
$$[x_{\gamma}]_{\Sigma_{T}}^{\Acal} = \frac{1}{\mathrm{cross}(T^\circ,\zg)} \sum_P 
x(P) y(P),$$ where the sum is over all perfect matchings $P$ of $G_{T^\circ,\zg}$.
\end{theorem}
Sections \ref{sect Sbar} - \ref{sec phi-map} set up the auxiliary results which are used for the proof of Theorem \ref{thm main}, which is given in Section \ref{sec plain}.  See Section \ref{sec outline} for an outline of the proof.

\begin{remark}
This expansion as a Laurent polynomial does not necessarily yield a reduced fraction, which is why our denominators are defined in terms of crossing numbers as opposed to the intersection numbers $(\alpha|\beta)$ defined in Section 8 of \cite{FST}.
\end{remark}

\subsection{Cluster expansion formulas for tagged arcs with notches}\label{sect notched}

We now consider cluster variables of tagged arcs which 
have a notched end.  The following remark shows that if we want to compute
the Laurent expansion of a cluster variable associated to a tagged arc notched at $p$,
with respect to a tagged triangulation $T$,
there is 
no loss of generality in assuming that all arcs in $T$ are tagged plain at $p$.

\begin{remark} \label{assumption}
Fix a tagged triangulation $T$ of $(S,M)$  such that $T = \iota(T^\circ)$, where $T^\circ$ is an ideal triangulation. 
Let $p$ and $q$ (possibly $p=q$) be two marked points, and let 
$\zg$ denote an ordinary  arc between $p$ and $q$.
If $p$ is a puncture and we are interested in computing the Laurent expansion of 
of $x_{\zg^{(p)}}$ with respect to $T$, we may assume that no
tagged arc in $T$ is notched at $p$. 
Otherwise, 
by changing the tagging of $T$ and $\zg^{(p)}$ at $p$, and applying Proposition \ref{tag-change},
we could 
reduce the computation of the Laurent expansion of $x_{\zg^{(p)}}$
to our formula for cluster variables 
corresponding to ordinary arcs.  Note that if there is no tagged arc in $T$ which 
is notched at $p$, then there is no loop in $T^\circ$ cutting out a once-punctured
monogon around $p$.
Similarly, if $p$ and $q$ are punctures
and we are interested in computing the Laurent expansion of 
 $x_{\zg^{(pq)}}$ with respect to $T$, we may assume that no
tagged arc in $T$ is notched at either $p$ or $q$ (equivalently, there
are no loops in $T^\circ$ cutting out once-punctured monogons around 
$p$ or $q$).
We will make these assumptions throughout this section.
\end{remark}

Before giving our formulas, we must introduce some notation.
\begin{definition} [\emph{Crossing monomials for tagged arcs with notches}]\label{crossnotch}
If $p$ is a puncture, and $\zg^{(p)}$ is a tagged arc with a notch at $p$ but tagged plain at its other end, we define the associated crossing monomial as $$\mathrm{cross}(T^\circ,\zg^{(p)}) =\frac{\mathrm{cross}(T^\circ,\zLp)}{\mathrm{cross}(T^\circ,\zg)} = 
\mathrm{cross}(T^\circ,\gamma) \prod_\tau x_\tau,$$
where the product is over all ends of arcs $\tau$ of  $T^\circ$ that are incident to $p$.
If $p$ and $q$ are punctures and $\zg^{(pq)}$ is a tagged arc with a notch at $p$ and $q$, we define the associated crossing monomial as
$$\mathrm{cross}(T^\circ,\zg^{(pq)}) =\frac{\mathrm{cross}(T^\circ,\zLp) \,
\mathrm{cross}(T^\circ,\zLq) }{\mathrm{cross}(T^\circ,\zg)^3} = 
\mathrm{cross}(T^\circ,\gamma) \prod_\tau x_\tau,$$ where the 
product is
over all ends of arcs $\tau$  that are incident to $p$ or $q$.
\end{definition}

Our formula computing the Laurent expansion of 
a cluster variable $x_{\zg^{(p)}}$ with exactly one notched end (at the puncture $p$)
involves {\it $\gamma$-symmetric} matchings of the graph associated to the ideal arc
$\zLp$
corresponding to $\zg^{(p)}$ (so $\iota(\zLp)=\zg^{(p)}$). 
Note that $\zLp$ is a loop cutting out a 
once-punctured monogon around $p$.  

Our goal now is to define {\it $\gamma$-symmetric} matchings.
For a tagged arc $\tau \in T$ and a puncture $p$, 
let $e_p(\tau)$ denote the number of ends of $\tau$ incident to $p$ (so if $\tau$ is a loop with 
its ends at $p$, $e_p(\tau) = 2$).  
We let $e_p = e_p(T) = \sum_{\tau \in T} e_p(\tau)$.

Keeping the notation of Section \ref{sect tiles}, 
orient $\zg$ from $q$ to $p$, let $\tau_{i_1}, \tau_{i_2},\dots, \tau_{i_d}$ denote the arcs 
crossed by $\zg$ in order, and let 
$\Delta_0,\dots,\Delta_{d+1}$ be the sequence of ideal triangles in $T^\circ$ which $\zg$ passes through.
We let $\zeta_1$ and $\zeta_{e_p}$ denote the sides of 
triangle $\zD_{d+1}$ not crossed by $\zg$ (by Remark \ref{assumption},
$\zeta_1 \neq \zeta_{e_p}$), so that $\tau_{i_d}$ follows $\zeta_{e_p}$ and 
$\zeta_{e_p}$ follows $\zeta_1$ in clockwise order around $\Delta_{d+1}$.  
Let $\zeta_2$ through $\zeta_{e_p-1}$ denote the labels of the other arcs incident to 
puncture $p$ in order as we follow $\ell_p$ clockwise around $p$.
Note that if  $T^\circ$ contains a loop $\tau$ based at a puncture $p$, 
then $\tau$ appears twice in the multiset $\{\zeta_1,\dots, \zeta_{e_p}\}$.  
Figure \ref{Bicycle} shows some possible local configurations around a puncture.

\begin{figure}
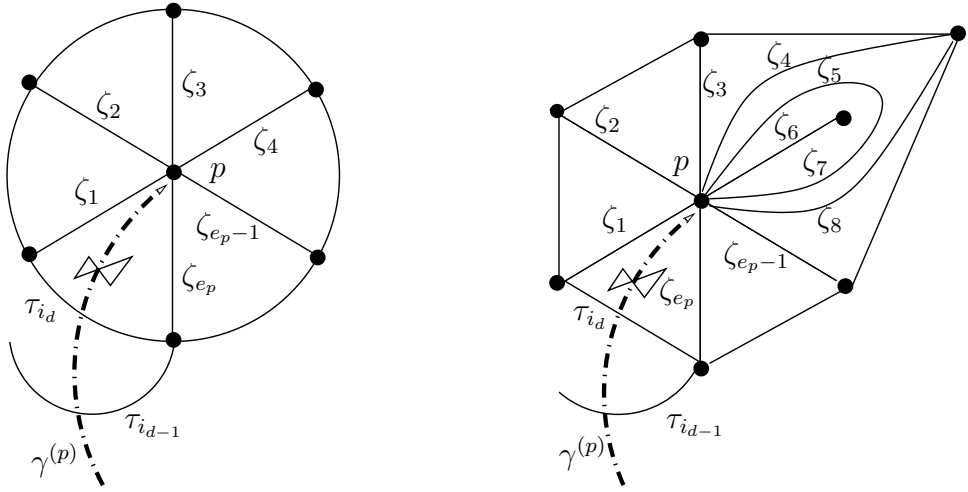

\input{figBicyclenew.pstex_t} \hspace{6em}  \input{figBicycle2new.pstex_t}
\caption{Possible local configurations around a puncture}
\label{Bicycle}
\end{figure}

\begin{definition} [\emph{Subgraphs $G_{T^\circ,\zg_p,1}$, $G_{T^\circ,\zg_p,2}$,
$H_{T^\circ,\zg_p,1}$, and $H_{T^\circ,\zg_p,2}$ of $G_{T^\circ, \zLp}$}]
Since $\zLp$ is a loop cutting out a once-punctured monogon with radius $\zg$ and puncture $p$,
the graph $G_{T^\circ,\zLp}$  contains two disjoint connected subgraphs, one on each end, 
both of which are isomorphic to $G_{T^\circ,\zg}$.  Therefore each  
subgraph consists of a union of tiles $G_{\tau_{i_1}}$ through $G_{\tau_{i_d}}$;
we let $G_{T^\circ,\zg_p,1}$ and $G_{T^\circ,\zg_p,2}$ denote these two subgraphs.

Let $v_1$ and $v_2$ be the two vertices of tiles $G_{\tau_{i_d}}$ in $G_{T^\circ, \zLp}$ incident to the edges labeled $\zeta_1$ and $\zeta_{e_p}$.  For $i\in \{1,2\}$, we let $H_{T^\circ,\gamma_p,i}$ be the connected subgraph of $G_{T^\circ,\gamma_p,i}$ which is obtained by deleting $v_i$ and the two edges incident to $v_i$.  See Figure \ref{LoopSnakeGraph}.
\end{definition}

\begin{figure}
\input{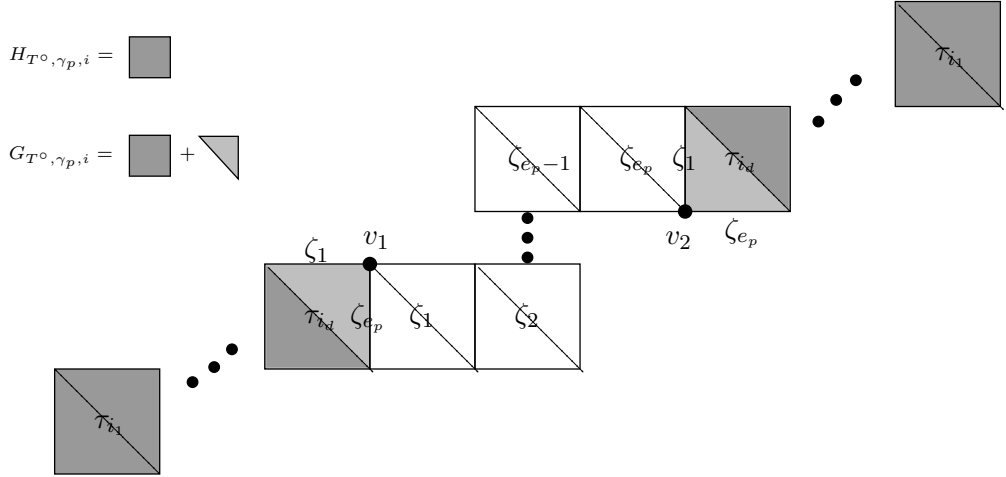}
\caption{$\overline{G}_{T^\circ,\zLp}$, with subgraphs $G_{T^\circ,\gamma_p,i}$ and
$H_{T^\circ,\gamma_p,i}$ shaded as indicated.}
\label{LoopSnakeGraph}
\end{figure}

\begin{definition} [\emph{$\gamma$-symmetric matching}]
Having fixed an ideal triangulation $T^\circ$ and an ordinary arc $\gamma$ between $p$ and $q$,
we call a perfect matching $P$ of $G_{T^\circ,\zLp}$ $\zg$-\emph{symmetric} 
if the restrictions of $P$ to the two ends satisfy $P|_{H_{T^\circ,\gamma_p,1}}
\cong P|_{H_{T^\circ,\gamma_p,2}}$.
\end{definition}

\begin{definition} [\emph{Weight and Height Monomials of a $\gamma$-symmetric matching}]
Fix a $\gamma$-symmetric matching $P$ of $G_{T^\circ,\zLp}$.  By Lemma \ref{twoends}, 
$P$ restricts to a perfect matching of (without loss of generality) $G_{T^\circ,\gamma_p,1}$.  
Therefore the following definitions of weight and (specialized) height monomials $\overline{x}(P)$ 
and $\overline{y}(P)$ are well-defined:
\begin{eqnarray*}\overline{x}(P) = \frac{x(P)}{x(P|_{G_{T^\circ,\zg,1}})},~~
\overline{y}(P) = \frac{y(P)}{y(P|_{G_{T^\circ,\zg,1}})}.
\end{eqnarray*}
\end{definition}

We are now  ready to state our result for tagged arcs with one notched end.
\begin{theorem} \label{thm single}
Let $(S,M)$ be a bordered surface with puncture $p$ and tagged triangulation $T= \{\tau_1,\tau_2,\dots, \tau_n\} = \iota(T^\circ)$ where  $T^\circ$ is an ideal triangulation.  
Let $\mathcal{A}$ be
the corresponding 
cluster algebra with principal coefficients with respect to  $\Sigma_T$.
Let $\zg$ be an ordinary arc  with one end incident to $p$,
and let $\zLp$ be the ordinary arc corresponding to 
$\zg^{(p)}$ (so $\iota(\zLp)=\zg^{(p)}$).  
Without loss of generality we can assume that $T$ contains no arc notched at $p$
and that $\zg \notin T$ (see Remarks \ref{assumption} and \ref{rem:enough}).
Let $G_{T^\circ,\zLp}$ be the graph constructed in Section \ref{sect graph}.  
Then the Laurent expansion of $x_{\zg^{(p)}}$ with respect to  $\Sigma_T$ is given by
$$[x_{\gamma^{(p)}}]_{\Sigma_{T}}^{\Acal} = \frac{1}{\mathrm{cross}(T^\circ,\zg^{(p)})} \sum_P \overline{x}(P)\, \overline{y}(P),$$ where the sum is over all $\zg$-symmetric matchings $P$ of $G_{T^\circ,\zLp}$.
\end{theorem}

\begin{remark}\label{rem:enough}
Note that if $\zg$ is in $T$ (so $x_{\zg}$ is an initial cluster variable), 
then since $x_{\zg} x_{\zg^{(p)}} = x_{\zLp}$, we obtain the Laurent expansion of 
$x_{\zg^{(p)}}$ with respect to $\Sigma_T$ by dividing the Laurent expansion of $x_{\zLp}$ 
(which is given by Theorem \ref{thm main}) by 
$x_{\zg}$.
\end{remark}

We prove Theorem \ref{thm single} in Section \ref{SingleNotchedProofs}.  For the case of a tagged arc with notches at both ends, we need two more definitions in the same spirit as the above notation.

\begin{definition} [$\gamma$-compatible pair of matchings]
Assume that the tagged triangulation $T$ does not contain either $\zg$, $\zg^{(p)}$, or $\zg^{(q)}$.  Let $P_p$ and $P_q$ be $\gamma$-symmetric matchings of  $G_{T^\circ,\zLp}$ and $G_{T^\circ,\zLq}$, respectively.  By Lemma \ref{twoends}, without loss of generality, $P_p|_{G_{T^\circ,\zg_p,1}}$ and
$P_q|_{G_{T^\circ,\zg_q,1}}$ are perfect matchings.  We say that $(P_p,P_q)$ is a $\zg$-\emph{compatible} pair if the restrictions satisfy $$P_p|_{G_{T^\circ,\gamma_p,1}}
\cong P_q|_{G_{T^\circ,\gamma_q,1}}.$$
\end{definition}

\begin{definition} [Weight and Height Monomials for $\gamma$-compatible  matchings]
Fix a $\gamma$-compatible pair of matchings $(P_p,P_q)$ of $G_{T^\circ,\zLp}$ and $G_{T^\circ,\zLq}$.
We define the weight and height monomial, respectively $\overline{\overline{x}}(P_p,P_q)$ and $\overline{\overline{y}}(P_p,P_q)$, as
\begin{eqnarray*}\overline{\overline{x}}(P_p,P_q) = \frac{x(P_p)\, x(P_q)}{x(P_p|_{G_{T^\circ,\zg,1}})^3},~~
\overline{\overline{y}}
(P_p,P_q) = \frac{y(P_p)\, y(P_q)}{y(P_p|_{G_{T^\circ,\zg,1}})^3}.
\end{eqnarray*}
\end{definition}

For technical reasons, we require the $(S,M)$ is not a closed surface with exactly $2$ marked points for Theorem \ref{thm double} and Proposition \ref{coeff-free}.

\begin{theorem} \label{thm double}
Let $(S,M)$ be a bordered surface with punctures $p$ and $q$ 
and tagged triangulation $T= \{\tau_1,\tau_2,\dots, \tau_n\} = \iota(T^\circ)$ where  $T^\circ$ is an ideal triangulation.
Let $\zg$ be an ordinary arc between $p$ and $q$.
Assume $\gamma \notin T$, and without loss of generality assume $T$ does not 
contain an arc notched at $p$ or $q$.
Let $\mathcal{A}$ be
the corresponding cluster algebra with principal coefficients with 
respect to $\Sigma_T$.
Let $\zLp$ and $\zLq$ be the two ideal arcs corresponding to 
$\zg^{(p)}$ and $\zg^{(q)}$.  
Let $G_{T^\circ,\zLp}$ and $G_{T^\circ,\zLq}$ be the graphs constructed in Section \ref{sect graph}.  Then the Laurent expansion of $x_{\zg^{(pq)}}$ with respect to $\Sigma_T$ is given by
$$[x_{\gamma^{(pq)}}]_{\Sigma_{T}}^{\Acal} = \frac{1}{\mathrm{cross}(T^\circ,\zg^{(pq)})} \sum_{(P_p,P_q)} \overline{\overline{x}}(P_p,P_q)\, \overline{\overline{y}}(P_p,P_q),$$ where the sum is over all 
$\zg$-compatible pairs of matchings $(P_p,P_q)$ of $(G_{T^\circ,\zLp}, G_{T^\circ,\zLq})$.
\end{theorem}

\begin{prop} \label{prop double-special}
Let $(S,M)$, $p$, $q$, $T$, $\mathcal{A}$, $\gamma$ be as in Theorem \ref{thm double},
except that we assume that $\gamma \in T$.
Then $[x_{\gamma^{(pq)}}]_{\xx_{T}}^{\Acal}$ is a positive Laurent polynomial.
\end{prop}

We prove this theorem and proposition in Section \ref{DoubleNotchedProofs}.

\begin{remark}\label{rem double}
If in Theorem \ref{thm double} the two endpoints $p$ and $q$ of $\gamma$ coincide, i.e. $\gamma$ is a loop, then we let $\ell_p$ and $\ell_q$ denote the loops (with self-intersections) displayed in Figure \ref{DNLoopReplacements} for the purpose of the formula for $[x_{\gamma^{(pp)}}]_{\Sigma_{T}}^{\Acal}$.
\end{remark}

\begin{figure}
\input{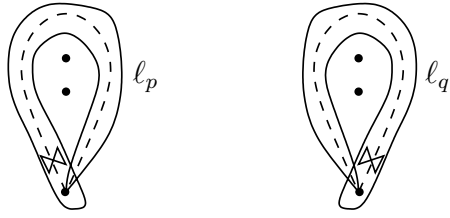} 
\caption{Analogues of $\ell_p$ and $\ell_q$ for a loop notched at its basepoint.}
\label{DNLoopReplacements}
\end{figure}

\section{Examples of results, and identities in the coefficient-free case} \label{ClusterExamples}

\subsection{Example of a Laurent expansion for an ordinary arc}

\begin{figure}
\input{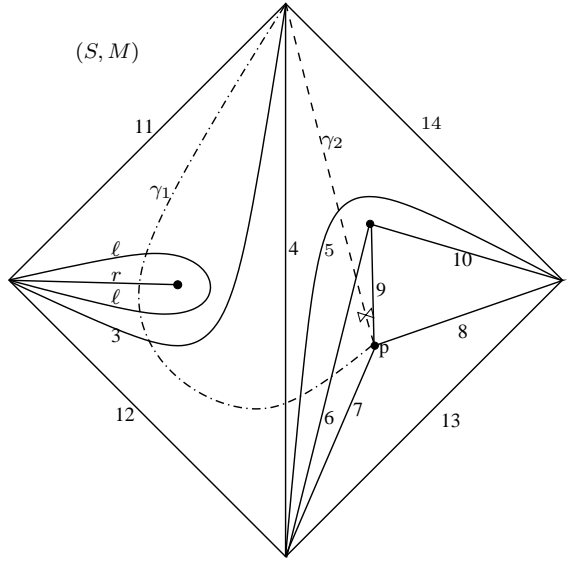}
\caption{Ideal Triangulation $T^\circ$ of $(S,M)$}
\label{fig ExT_0}
\end{figure}

Consider the ideal triangulation  in Figure \ref{fig ExT_0}.  
We have labeled the loop of the ideal triangulation $T^\circ$ as $\ell$ 
and the radius as $r$.  The corresponding tagged triangulation has two arcs, 
both homotopic to $r$: we denote by $\tau_1$ the one which is notched at the puncture,
and by $\tau_2$ the one which is tagged plain at the puncture. 
The graph $\overline{G}_{T^\circ,\zg_1}$ corresponding to the  arc $\zg_1$ is shown 
on the left of Figure \ref{matching graphs}.  It is drawn so that 
the relative orientation of the first tile $\mathrm{rel}(G_{\ell},T^\circ)$ is equal to $-1$.  
$G_{T^\circ,\zg_1}$  has 
$19$ perfect matchings.

\begin{figure}
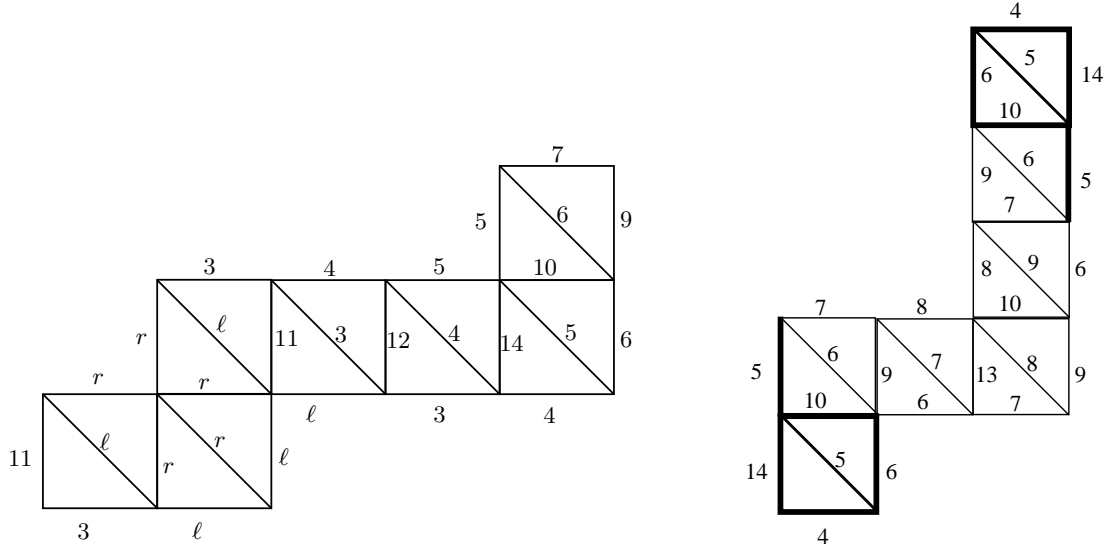

\input{FigSnGrTake4.pstex_t} \hspace{3em} \input{figSnGr2.pstex_t}
\caption{The graphs $\overline{G}_{T^\circ,\zg_1}$ and $\overline{G}_{T^\circ,\ell_p}$.}
\label{matching graphs}
\end{figure}

Applying Theorem \ref{thm main}, we make the specialization
$x_{\ell} = x_1x_2$, $x_r=x_2$, $y_{\ell} = y_1$, $y_r= y_2/y_1$, 
and $x_{11}=x_{12}=x_{13}=x_{14}=1$.  We find that $x_{\gamma_1}$ is equal to:

\begin{eqnarray*} 
 && \frac{1}{x_1 x_2 x_3 x_4 x_5 x_6} \big(
 x_1 x_2 x_4^2 x_5 x_9
+ {\color{red} y_3} ~~ x_4 x_5 x_9
+ {\color{red} y_6} ~~ x_1 x_2 x_4^2 x_7
+ {\color{red}y_1 y_3} ~~ x_3 x_4 x_5 x_9
+ {\color{red} y_3 y_6} ~~ x_4 x_{10} x_7
\\
&+& {\color{red} y_5 y_6} ~~x_1x_2 x_4 x_6 x_7
+ {\color{red} y_2 y_3} ~~ x_3 x_4 x_5 x_9
+ {\color{red} y_1 y_3 y_6} ~~ x_3 x_4 x_{10} x_7
+ {\color{red} y_3 y_5 y_6} ~~x_6 x_7
+ {\color{red} y_1 y_2 y_3} ~~ x_3^2 x_4 x_5 x_9
\\
&+& {\color{red} y_2 y_3 y_6} ~~ x_3 x_4 x_{10} x_7
+ {\color{red} y_1 y_3 y_5 y_6} ~~ x_3 x_6 x_7
+ {\color{red} y_3 y_4 y_5 y_6} ~~ x_3 x_5 x_6 x_7
+ {\color{red} y_1 y_2 y_3 y_6} ~~ x_3^2 x_4 x_{10} x_7
\\
&+& {\color{red} y_2 y_3 y_5 y_6} ~~x_3 x_6 x_7
+ {\color{red} y_1 y_3 y_4 y_5 y_6} ~~ x_3^2 x_5 x_6 x_7
+ {\color{red} y_1 y_2 y_3 y_5 y_6} ~~ x_3^2 x_6 x_7
+ {\color{red} y_2 y_3 y_4y_5 y_6} ~~ x_3^2 x_5 x_6 x_7 \\
&+& {\color{red} y_1 y_2 y_3 y_4 y_5 y_6} ~~ x_3^3 x_5 x_6 x_7
\big).
\end{eqnarray*}

\subsection{Example of a Laurent expansion for a singly-notched arc}\label{sec single}

To compute the Laurent expansion of $x_{\gamma_2}$ 
(the notched arc in 
Figure \ref{fig ExT_0}), we draw the graph $\overline{G}_{T^\circ,\ell_p}$ associated to the 
loop $\ell_p$, where $\ell_p$ is the ideal arc associated to $\gamma_2$.  
Figure \ref{matching graphs} depicts this graph, embedded so that 
the relative orientation of the tiles with diagonals labeled $5$ is $+1$.  
We need to enumerate $\gamma$-symmetric matchings of 
$\overline{G}_{T^\circ,\ell_p}$, 
i.e. those matchings which 
have isomorphic restrictions to the two bold subgraphs.

Splitting up the set of $\gamma$-symmetric matchings into three classes, 
corresponding to the configuration of the perfect matching on the restriction to $G_{\gamma}$, we obtain
\begin{eqnarray*} [x_{\gamma_2}]_{\xx_{T}}^{\Acal} &=& \frac{1}{x_5 x_6 x_7 x_8 x_9}
\big( x_4 x_5 ( x_9 x_6 x_8 + {\color{red} y_7} ~~x_9  x_9 + {\color{red} y_7 y_8}~~x_9 x_7 x_{10}) \\
&+& {\color{red} y_6y_7}~~x_4 x_{10}(x_9 x_7 + {\color{red} y_8} ~~x_7 x_{10} x_7 + {\color{red} y_8 y_9}~~x_7 x_8 x_6) \\
&+& {\color{red} y_5 y_6y_7}~~ x_6 (  x_9 x_7 + {\color{red} y_8} ~~x_7 x_{10} x_7 + {\color{red} y_8 y_9}~~x_7 x_8 x_6)\big).
\end{eqnarray*}
Since all the initial variables and coefficients appearing in this sum correspond to ordinary arcs, no 
specialization of $x$-weights or $y$-weights was necessary in this case 
(except for the boundaries $x_{10}=x_{11}=1$).

\subsection{Example of a Laurent expansion for a doubly-notched arc}\label{sec double}

We close with an example of a cluster expansion formula for a tagged arc with notches at both endpoints.  
\begin{figure}
\input{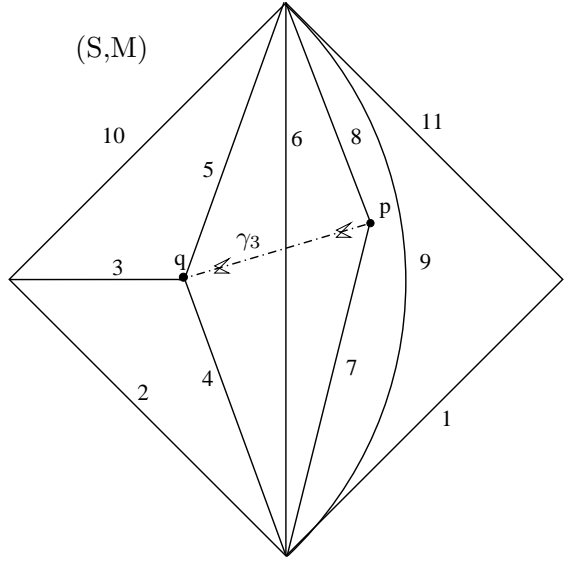}
\caption{Ideal triangulation $T^\circ$ and doubly-notched arc $\gamma_3$.} 
\label{DoubNotchEg}
\end{figure}
We build two 
graphs associated to the doubly-notched arc $\zg_3$ in Figure \ref{DoubNotchEg}: each graph 
corresponds to a loop $\ell_p$ or $\ell_q$
tracing out a once punctured monogon around an endpoint of $\zg_3$.
Note that in the planar embeddings of Figure \ref{DoubNotchEgSG}, the relative orientations of 
the first tiles are both $+1$.
So each  minimal matching uses the lowest edge
in  $\overline{G}_{T^\circ,\ell_p}$ and
$\overline{G}_{T^\circ,\ell_q}$, respectively. 
\begin{figure}
\scalebox{1.2}{\input{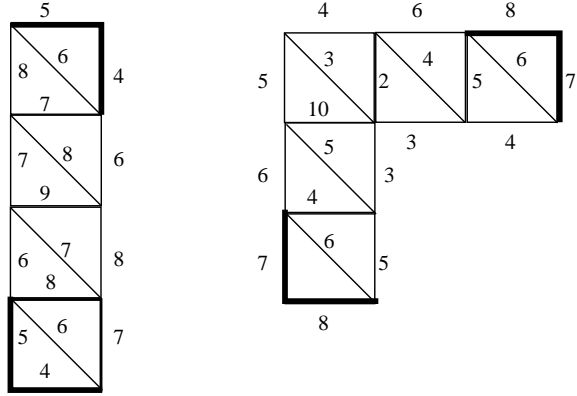}}
\caption{Graphs $\overline{G}_{T^\circ,\ell_p}$ and $\overline{G}_{T^\circ,\ell_q}$ corresponding to ideal arcs $\ell_p$, 
$\ell_q$. } 
\label{DoubNotchEgSG}
\end{figure}
To write down the Laurent expansion for $x_{\zg_3}$, we need to 
enumerate the $\gamma$-compatible pairs of perfect matchings of these graphs. 
There are $12$ pairs of $\gamma$-compatible perfect matchings in all, 
yielding the $12$ monomials in the expansion of $x_{\gamma_3}$:
\begin{eqnarray*}
[x_{\gamma_3}]_{\xx_{T}}^{\Acal}
&=&\frac{1}{x_3 x_4 x_5 x_6 x_7 x_8}\big(x_3 x_4 x_6^2 x_8 + {\color{red} y_5}~~ x_4^2 x_6 x_8 + {\color{red} y_7}~~ x_3 x_4 x_6 x_8 x_9  \\
&+&  {\color{red} y_3 y_5}~~ x_2 x_4 x_5 x_6 x_8 +  {\color{red} y_5 y_7}~~ x_4^2 x_8 x_9 +    {\color{red} y_3 y_5 y_7}~~ x_2 x_4 x_5 x_8 x_9 \\
&+&  {\color{red} y_5 y_6 y_7}~~ x_4 x_5 x_7 x_9 +  {\color{red} y_3 y_5 y_6 y_7}~~ x_2 x_5^2 x_7 x_9 +    {\color{red} y_5 y_6 y_7 y_8}~~ x_4 x_5 x_6 x_7 \\
&+&  {\color{red} y_3 y_4 y_5 y_6 y_7}~~ x_3 x_5 x_6 x_7 x_9 +  {\color{red} y_3 y_5 y_6 y_7 y_8 }~~ x_2 x_5^2 x_6 x_7 + {\color{red} y_3 y_4 y_5 y_6 y_7 y_8}~~
x_3 x_5 x_6^2 x_7\big).
\end{eqnarray*}

\subsection{Identities for cluster variables in the coefficent-free case}

\begin{remark}
Note that if 
we set all the $y_i$'s equal to $1$ in Section \ref{sec single}, then 
$x_{\gamma_2}$ factors as
$$\big(\frac{x_{10} {x_7} + x_6 {x_8} +  {x_9}  }{x_7 x_8 x_9}\big)
\big(\frac{x_6 x_7  + x_4 x_7 x_{10} + x_4 x_5 x_9}{x_5 x_6}\big).$$  
The first term depends only on the local configuration around the puncture $p$ 
(at which end $\gamma_2$ is notched).
The second term is exactly the coefficient-free cluster variable 
associated to the ordinary arc homotopic to $\gamma_2$.

Also, if we set all the $y_i$'s equal to $1$ in Section \ref{sec double},
then $x_{\gamma_3}$ factors as 
$$\big(\frac{x_3 x_6+x_4 + x_2 x_5}{x_3 x_4 x_5}\big) \big(\frac{x_6 + x_9}{x_7 x_8}\big) 
\big(\frac{x_4 x_8 + x_5 x_7}{x_6}\big).$$
The first and second terms correspond to horocycles around 
the punctures $q$ and $p$, respectively, and the third term
is exactly the coefficient-free cluster variable associated 
to the ordinary arc homotopic to $\gamma_3$.
\end{remark}

These examples hint at a general phenomenon 
in the coefficient-free case, as we now describe.

\begin{definition} \label{zpDef}
Fix a bordered surface $(S,M)$ and a tagged triangulation $T = \iota(T^\circ)$ of $S$.
For any puncture $p$ we construct a  Laurent polynomial with positive coefficients
that only depends on the local neighborhood of $p$.  
Let $\tau_{1}, \tau_{2},\dots, \tau_{h}$ denote the ideal arcs of $T^\circ$ incident to $p$ in 
clockwise order, assuming that $h\geq 2$.  
(If a loop is incident to $p$, it appears 
twice in this list, once for each end.)
Let $[\tau_i,\tau_{i+1}]$ denote the unique arc 
in an ideal triangle containing $\tau_i$ and $\tau_{i+1}$, such that 
$[\tau_i,\tau_{i+1}]$ 
is in 
the clockwise direction from $\tau_i$;
here the indices in $[\tau_i,\tau_{i+1}]$ are considered modulo $h$.
  We set 
$$z_p = \frac{\sum_{i=0}^{h-1} \sigma^i ( x_{[\tau_1,\tau_2]}  x_{\tau_3} x_{\tau_4} 
\cdots x_{\tau_h})}{x_{\tau_1} x_{\tau_2}\cdots  x_{\tau_h}},
$$ where $\sigma$ is the cyclic permutation $(1,2,3,\dots, h)$ acting on subscripts.

In the case that $p$ has exactly one ideal arc $r$ incident to it, 
then the tagged triangulation contains exactly two tagged arcs $r$ and $r^{(p)}$ 
(technically $\iota(r)$ and $\iota(r)^{(p)}$)
incident to $p$. In this case,
$$z_p = \frac{x_{r^{(p)}}}{x_r}.$$ 
\end{definition}

\begin{prop} \label{coeff-free}
Fix $(S,M)$ and $T$ as above, let $\A$ be the corresponding coefficient-free cluster algebra,
and let
$\gamma$ be an ordinary arc between distinct marked points $p$ and $q$, or a loop
which does not cut out a once-punctured monogon.
Then if $p \neq q$ and $p$ is a puncture, $$x_{\gamma^{(p)}} = z_p \cdot x_{\gamma},$$ and if both 
$p$ and $q$ are punctures, $$x_{\gamma^{(pq)}} = z_p z_q \cdot x_{\gamma.}$$
Finally if $\gamma$ is a loop so that $p=q$ and $\gamma^{(pp)}$ is a doubly-notched loop, then 
$$x_{\gamma^{(pp)}} = z_p^2 \cdot x_{\gamma}.$$
\end{prop}

We will prove Proposition \ref{coeff-free} in Section \ref{sec quick}.


\section{Outline of the proof of the cluster expansion formulas}\label{sec outline}

As the proofs in this paper are rather involved, we present
here a detailed outline.
\begin{enumerate}
\item[Step 1.] Fix a bordered surface with marked points $(S,M)$.
The seeds of $\A=\A(S,M)$ are in bijection with tagged triangulations,
so to prove the positivity conjecture for $\A$, we must prove 
positivity with respect to every seed $\Sigma_T$ where $T$ is a 
tagged triangulation.
By Proposition \ref{reduce-to-ideal}, it is enough to prove
positivity with respect to every seed $\Sigma_T$
where $T=\iota(T^\circ)$ for some ideal triangulation.\\

\item[Step 2.]  Fix an ideal triangulation $T^\circ=(\tau_1,\dots,\tau_n)$
of $(S,M)$, with 
boundary segments denoted $\tau_{n+1},\dots,\tau_{n+c}$.  
Fix also an ordinary arc $\gamma$, which crosses $T$
$d$ times; we would like to understand the Laurent expansion of 
$x_{\gamma}$ with respect to $\Sigma_T$.
We build a triangulated polygon $\Sbar_{\gamma}$
which comes with a ``lift" $\tilde{\gamma}$ of $\gamma$.
The triangulation $\Tbar_{\gamma}$ of $\Sbar_{\gamma}$ 
has $d$ internal arcs labeled $\sigma_1,\dots,\sigma_d$,
and $d+3$ boundary segments labeled $\sigma_{d+1},\dots,\sigma_{2d+3}$.
We have a surjective map $\pi: \{\sigma_1,\dots,\sigma_{2d+3}\} \to 
\{\tau_1,\dots,\tau_{n+c}\}$.  This step will be addressed in 
Section \ref{sect Sbar}.  \\

\item[Step 3.]  We build a type $A_d$ cluster algebra $\Abar_{\gamma}$
associated to $\widetilde{S}_{\gamma}$, with a $(3d+3) \times d$
extended exchange matrix.  This is obtained from the $(2d+3) \times d$ extended exchange
matrix associated to 
$(\Sbar_{\gamma}, \Tbar_{\gamma})$ (with rows indexed by interior arcs
and boundary segments), and appending a $d \times d$
identity matrix below.  It is clear from the construction
that the initial cluster is {\it acyclic}.\\

\item[Step 4.] We construct a map 
$\phi_{\gamma}$ from $\Abar_{\gamma}$ to the fraction field $\Frac(\A)$, 
such that for each $\sigma \in \Tbar_{\gamma}$,
$\phi_{\gamma}(x_{\sigma}) = x_{\pi(\sigma)}$.  We check that 
$\phi_{\gamma}$ is a well-defined homomorphism,
using the fact that $\Abar_{\gamma}$ is acyclic,
and \cite[Corollary 1.21]{BFZ}.  Steps 3 and 4 will be addressed in 
Section \ref{sect Abar}.\\

\item[Step 5.] 
We identify a quadrilateral $Q$ in $S$ with simply-connected
interior containing $\gamma$ as a diagonal, whose other diagonal
and sides (denoted $\gamma',\alpha_1,\alpha_2,\alpha_3,\alpha_4$)
cross $T$ fewer times than $\gamma$ does.  To do so we use
(a slight generalization of) a lemma of \cite{ST}, which 
will be stated and proved in Section \ref{sec simply connected}. \\

\item[Step 6.] We check that $\phi_{\gamma}(x_{\tilde{\gamma}}) = x_{\gamma}$,
by induction on the number of crossings of $\gamma$ and $T$.
To do so, we use Step 5 to produce $Q$, which we lift
to a quadrilateral $\Qbar$ in a larger 
triangulated polygon $\widehat{S}$ 
containing $\Sbar_{\gamma}$. 
By induction, the cluster expansions of each of 
$x_{\gamma'}, x_{\alpha_1}, x_{\alpha_2}, x_{\alpha_3}$, and $x_{\alpha_4}$
are given by matching formulas using the combinatorics of $\widehat{S}$. 
By comparing the exchange relations corresponding to the flip
in $\Qbar$
and the flip in $Q$, and using the fact that cluster expansion formulas
are known in type A, we deduce
that $\phi_{\gamma}(x_{\tilde{\gamma}}) = x_{\gamma}$.\\

\item[Step 7.] 
In type A, the matching formula giving  
the Laurent expansion of $x_{\tilde{\gamma}}$ in $\Abar_{\gamma}$
with respect to $\Sigma_{\Tbar_{\gamma}}$ is known.
Since $\phi_{\gamma}(x_{\tilde{\gamma}}) = x_{\gamma}$,
and $\phi_{\gamma}$ is a homomorphism, we can compute the Laurent expansion
of $x_{\gamma}$ in terms of $\Sigma_T$.
Here we use the fact that for every arc 
$\sigma_i\in \Tbar_{\gamma}$, $\phi_{\gamma}(x_{\sigma_i}) = x_{\pi(\sigma_i)}$.
This proves our main theorem for cluster variables
corresponding to ordinary arcs {\it and} loops $\ell$ cutting out 
once-punctured monogons.
Steps 6 and 7 will be addressed in Section \ref{sec plain}.\\

\item[Step 8.] We prove our combinatorial formula
for a singly notched
arc by using the identity
 $x_{\ell} = x_{r} x_{r^{(p)}}$ (where $\ell$ cuts out
a once-punctured monogon with radius $r$ and puncture $p$),
and the fact that we now have proved our combinatorial formula
for $x_{\ell}$ and $x_{r}$.  For doubly-notched arcs we use
an analogous strategy, using a more complicated identity
(Theorem \ref{double-identity}). The proof for doubly-notched loops
is the same as for doubly-notched arcs, but 
we need to make sense of the cluster algebra element
corresponding to a singly-notched loop (see Definition \ref{SLoop}).  Step 8 is addressed in
Section \ref{sec Finish}.
\end{enumerate}

\section{Construction of a triangulated polygon and a lifted arc}\label{sect Sbar}

Let $T=\{\tau_1,\ldots,\tau_n,\tau_{n+1},\ldots,\tau_{n+c}\}$ be an ideal
triangulation of the surface $(S,M)$, where
$\tau_1,\ldots,\tau_n$ are arcs and $\tau_{n+1},\ldots,\tau_{n+c}$ are
boundary segments.
Let $\zg$ be an ordinary  (not notched) arc in $(S,M)$ that  crosses $T$ exactly $d$ times.
We now explain how to associate a triangulated polygon $\Sbar_{\gamma}$ 
to $\zg$,
as well as a lift $\tilde{\gamma}$ of $\gamma$, which
we shall use later to compute the cluster expansion of $x_\zg$.

We
fix an orientation for $\zg$ and we denote its starting point by $s$
and its endpoint by $t$, with $s,t\in M$. Let
$s=p_0,p_1,\ldots,p_d,p_{d+1}=t$
be the intersection  points of $\zg$ and $T$ in order of occurrence on
$\zg$, hence $p_0, p_{d+1}\in M$ and each $p_i$ with $1\le i\le d$
lies in the interior of $S$.
 Let $i_1,i_2,\ldots,i_d$ be such that $p_k$ lies on the arc
$\tau_{i_k}\in T$, for $k=1,2,\ldots,d$. Note that $i_k$ may be equal
to $i_j$ even if $k\ne j$.

For $k=0,1,\ldots,d$, let $\zg_k$ denote the segment of the path $\zg$
from  the point $p_k$ to the point $p_{k+1}$. Each $\zg_k$ lies in
exactly one ideal triangle $\zD_k$ in $T$. If $1\le k\le d-1$, then the triangle$\zD_k$ is formed by the arcs $\tau_{i_k},
\tau_{i_{k+1}}$ and a third arc that we denote by
$\tau_{[\zg_k]}$. If the triangle is self-folded then $\tau_{[\zg_k]}$ is equal to either  $\tau_{i_k}$ or $
\tau_{i_{k+1}}$. Note however, that   $\tau_{i_k}$ can't be equal to $
\tau_{i_{k+1}}$, since $\zg$ crosses them one after the other.

The idea now is to construct our triangulated polygon by glueing together
triangles which are modeled after  $\zD_0,\zD_1,\ldots,\zD_d$.
But some of $\zD_0, \zD_1,\dots, \zD_d$ may be self-folded, and we do not want to have self-folded triangles in the polygon. So we will unfold the self-folded triangles in a precise way, before glueing them back together.

Let $s_j$ denote the common endpoint of $\tau_{i_j}$ and $\tau_{i_{j+1}}$ such that
the triangle with vertices $s_j, p_j, p_{j+1}$ and with sides contained
in $\tau_{i_j}$, $\tau_{i_{j+1}}$, and $\gamma_j$ has simply connected 
interior, see Figure \ref{figs}.
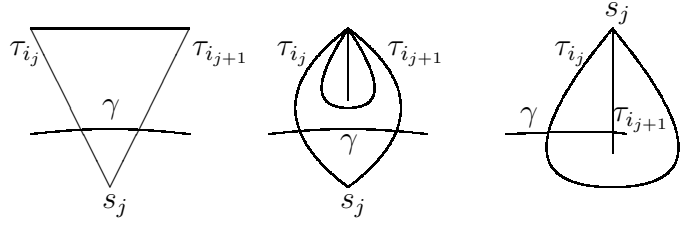
\begin{figure}\begin{center}
\setlength{\unitlength}{10pt}
\begin{picture}(10,15)(4,-4)
\qbezier(-3,2)(0,2.4)(3,2)
\put(-3,6){\line(1,0){6}}
\put(0,0){\line(-1,2){3}}
\put(0,0){\line(1,2){3}}
\put(-0.3,2.8){$\gamma$}
\put(-3.8,5){$\tau_{i_j}$}
\put(3.1,5){$\tau_{i_{j+1}}$}
\put(-0.3,-0.8){$s_j$}
\qbezier(9,0)(5,3)(9,6)
\qbezier(9,0)(13,3)(9,6)
\qbezier(9,6)(7,3)(9,3)
\qbezier(9,6)(11,3)(9,3)
\qbezier(9,6)(9,3.3)(9,3.3)
\qbezier(6,2)(9,2.4)(12,2)
\put(8.7,1.5){$\gamma$}
\put(6.3,5){$\tau_{i_j}$}
\put(10.5,5){$\tau_{i_{j+1}}$}
\put(8.7,-0.8){$s_j$}
\qbezier(19,6)(14,0)(19,0)
\qbezier(19,6)(24,0)(19,0)
\qbezier(19,6)(19,1.3)(19,1.3)
\qbezier(15,2)(19,2.2)(19.5,2)
\put(15.5,2.5){$\gamma$}
\put(16.8,5){$\tau_{i_j}$}
\put(19,2.5){$\tau_{i_{j+1}}$}
\put(18.7,6.5){$s_j$}
\end{picture}
\caption{Definition of the point $s_j$}\label{figs}\end{center}
\end{figure}
Let  $M(\zg)=\{s_j\mid 1 \leq j \leq d-1\}$.

We now partition the $s_j$'s into subsets of consecutive elements
which coincide.  That is, we define
integers $0=a_0< a_1<\ldots<a_{\ell-1}< a_\ell=d-1$, by requiring that
\[\begin{array}{ccccccccc}
 s_1&=&s_2&=&\cdots&=&s_{a_1}&\ne& s_{a_1+1}\\
s_{a_1+1}&=&s_{a_1+2}&=&\cdots & =& s_{a_2}&\ne& s_{a_2+1}\\
\vdots& &\vdots& && &\vdots& &\vdots\\
s_{a_{\ell-1}+1}&=&s_{a_{\ell-1}+2}&=&\cdots&=&s_{a_\ell}&=&s_{d-1}.
\end{array}\]
In the example in Figure \ref{figSbar}, we have
\begin{center}
  \begin{tabular}{ c | c | c | c | c | c | c | c | c  }
  \  $a_0$\ \  &\ \ $ a_1$ \ \   & \ \ $ a_2 $\ \ &\ \ $a_3$\ \ &\ \ $ a_4 $\ \ &\ \ $ a_5$\ \
  &\ \ $a_6$\ \ &\ \ $ a_7 $\ \ &\ \ $ a_8 $  \\ \hline
0&3&4&7&9&10&12&13&14.
\end{tabular}
\end{center}
We define $t_1=s_{a_1}$, $t_2=s_{a_2},\ldots,t_\ell=s_{d-1}$.
Note that $M(\zg)=\{t_1,t_2,\ldots,t_\ell\},$
and that $t_i$ may be equal to $t_j$ even if $i\ne j$.

We now construct a triangulated polygon $\Sbar_{\gamma}$ which is a union of
fans $F_1,\ldots,F_\ell$, where each $F_h$ consists of
$a_{h}-a_{h-1}+2$ triangles that all share the vertex $t_h$.
\begin{figure}
\scalebox{0.8}{\input{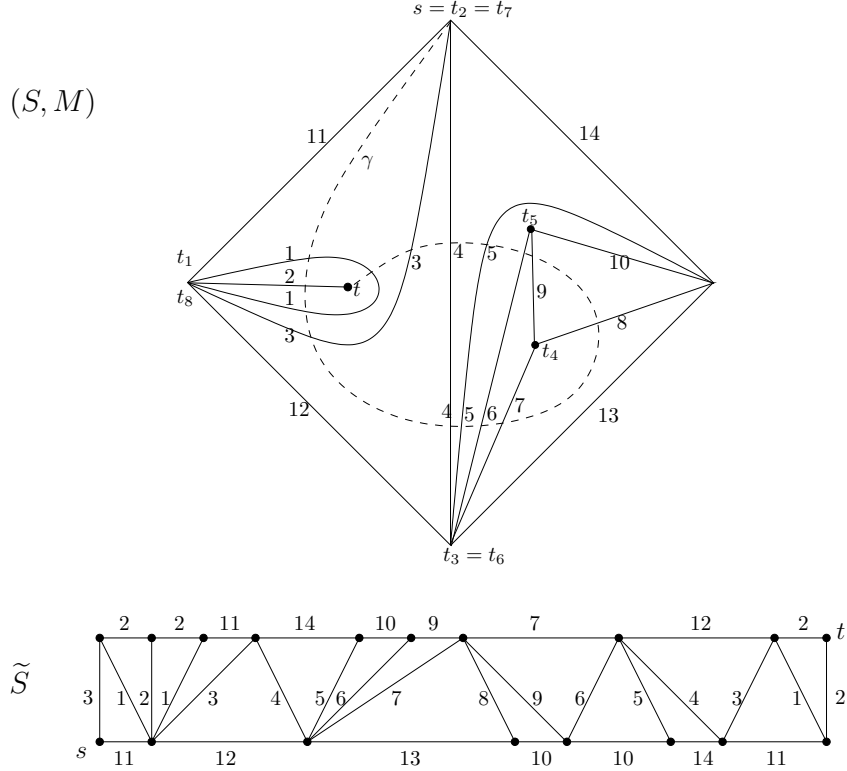}}
\caption{Construction of $\Sbar_{\gamma}$ in  a thrice-punctured square. The arcs of the triangulation $T$ are labeled 1,2,\ldots,14 for simplicity,
and the arcs of $\Tbar_{\gamma}$ are labeled
according to their images under $\pi$.
The arc $\zg$ is the dotted curve. There are $d=15$ crossings between $\zg$ and $T$, and $M(\zg)=\{t_1,\ldots,t_8\}$, where $t_1=t_8, t_2=t_7$ and $t_3=t_6$.}\label{figSbar}
\end{figure}
We will describe this precisely below;
see Figure \ref{figSbar}.
\begin{itemize}
\item[Step 1:] Plot a rectangle  with vertices $(0,0),(0,1),(d-1,1),(d-1,0)$.
\item[Step 2:] Label $(0,0)$, $(1,0)$, and $(0,1)$ by  $s$, $t_1$, and
$t_0$, respectively.
For $a_{2h}+1\le k\le a_{2h+1}$,
plot the points $(k,1)$ and label $(a_{2h+1},1)$ by $t_{2h+2}$.
For $a_{2h+1}+1 \leq k \leq a_{2h+2}$,
plot the points $(k,0)$,
and label $(a_{2h+2},0)$ by $t_{2h+3}$.
\item[Step 3:]
Connect $t_{2h}$ by a line segment with each point $(k,0)$
that lies between (and including) $t_{2h-1} $ and $t_{2h+1}$, for $1\le h<\ell/2$.  Connect $t_{2h+1}$ by a line segment with each point $(k,1)$ that lies between $t_{2h} $ and $t_{2h+2}$, for $1\le h<(\ell-1)/2$.
\item[Step 4:] If $\ell$ is odd, label $(d-1,0)$ by $t$
and otherwise label $(d-1,1)$ by $t$.
\item[Step 5:] Label the interior arcs of the polygon
by $\sigma_1,\dots,\sigma_d$,
in the order
that a curve from $s$ to $t$ (which intersects each only once) would
cross them.  Set $\pi(\sigma_1) = \tau_{i_1}, \dots, \pi(\sigma_d) =
\tau_{i_d}.$
Label the boundary segments of the polygon by $\sigma_{d+1},\dots,
\sigma_{2d+3}$, starting at $s$ and going counterclockwise
around the boundary of $\Sbar_{\gamma}.$
\item[Step 6:] Each boundary segment $\sigma_j$
not incident to $s$ or $t$ is the side of a unique triangle in the polygon, whose other sides project via $\pi$ to $\tau_{i_k},\tau_{i_{k+1}}$, for some $k$. 
If the ideal triangle $\zD_k$ has three distinct sides,
set $\pi(\sigma_j) = \tau_{[\zg_k]}$.  Otherwise
$\zD_k$ is self-folded: define $\pi(\sigma_j)$ to be
the label
of the radius in $\Delta_k$.
\item[Step 7:] If $\sigma_j$ is the segment between $s$ and $t_1$,
define $\pi(\sigma_j)$ to be
the label of the arc between $s$ and $t_1=s_1$ in $\zD_0$ in the surface $S$ (note
$s \ne s_1$).
If $\sigma_j$ is the other edge incident to $s$, define $\pi(\sigma_j)$
to be
the label of the third side of $\zD_0$  if $\zD_0$ is not self-folded;
and otherwise, define it to be the label of the radius in $\zD_0$.
\item[Step 8:] If $\sigma_j$ is  the
segment between $t$ and  $t_\ell$, define
$\pi(\sigma_j)$ to be the label of the arc between $t$ and  $t_\ell$
in $\zD_d$ in $S$ (note that $t_\ell\ne t$).
If $\sigma_j$ is the other edge incident to $t$, define $\pi(\sigma_j)$
to be
the label of the third side of $\zD_d$, if $\zD_d$ is not self-folded;
and otherwise define it to be the label
of the radius in $\zD_d$.
\item[Step 9:] Each of the triangles in this construction corresponds to an ideal
triangle in $T$.  If the ideal triangle is not self-folded, then the
constructed triangle may have the same orientation as the ideal triangle or the
opposite one, but if the orientations do not match for one such pair of
triangles then it does not match for any such pair of triangles.  In the latter case,
we reflect the whole polygon at the horizontal axis.
\item[Step 10:] We will use $\tilde{\gamma}$ to denote
the arc in $\Sbar_{\gamma}$ from $s$ to $t$; we call this the {\it lift} of $\gamma$.

\end{itemize}

The result is a polygon $\Sbar_{\gamma}$ 
with set of vertices $\Mbar$  and triangulation
$\Tbar_{\gamma}$.
Its internal arcs are labeled $\sigma_1,\dots,\sigma_d$, and
the boundary segments are labeled $\sigma_{d+1},\dots,\sigma_{2d+3}$.
Moreover, each triangle $\Dbar_i$ in $\Tbar_{\gamma}$ corresponds to an ideal triangle in $T$, and, if the ideal triangle is not self-folded, then the orientations of
the two triangles match.

\section{Construction of $\Abar_{\gamma}$ and the map $\phi_{\gamma}$}\label{sect Abar}

Let $(S,M)$ be a bordered surface with marked points, fix an ideal
triangulation $T$ with internal arcs
$\{\tau_1,\dots, \tau_n\}$ and boundary segments 
$\{\tau_{n+1},\dots,\tau_{n+c}\}$, and 
let $\mathcal A$ be the associated cluster algebra with principal coefficients.
Note that the initial cluster variables of $\mathcal A$ are 
$\{x_{\tau_i} \ \vert \ 1 \leq i \leq n \}$.
Using the construction of $\Sbar_{\gamma}$ and $\Tbar_{\gamma}$ in the previous section,
we will construct a related type A cluster algebra
$\Abar_{\gamma}$, 
and define a homomorphism $\phi_{\gamma}$
from $\Abar_{\gamma}$ to the fraction field $\Frac(\Acal)$.

\subsection{Construction of a type $A$ cluster algebra}

To this end,
let $\Sbar_{\gamma}$ be the polygon with triangulation $\Tbar_{\gamma}$ 
constructed in Section \ref{sect Sbar}. Recall that its internal arcs are labeled $\sigma_1,\dots,\sigma_d$, and its boundary segments 
are labeled
$\sigma_{d+1},\dots,\sigma_{2d+3}$.

We define a $(3d+3) \times d$  exchange matrix $\Bbar$ as follows.
The first $2d+3$ rows are the signed adjacency matrix of the triangulation 
$\Tbar_{\gamma}$ together with its boundary segments.
The bottom $d$ rows are a copy of the $d \times d$ identity matrix.
We let $\Abar_{\gamma} = \A(\Bbar)$, and denote the initial cluster by $\xx_{\Tbar_{\gamma}}$. 
We denote the coefficient variables by 
$\{x_{\sigma_{d+1}},\dots,x_{\sigma_{2d+3}}\} \cup 
 \{y_{\sigma_1}, \dots, y_{\sigma_d}\}$.  We let 
$\PP = \Trop(x_{\sigma_{d+1}},\dots,x_{\sigma_{2d+3}}, y_{\sigma_1},\dots, y_{\sigma_d})$
be the tropical semifield.  

The following lemma is obvious.

\begin{lemma}
The $2d+3$ coefficient variables of $\Abar_{\gamma}$ are encoded by both the boundary segments
of $\Sbar_{\gamma}$ and elementary laminations associated to the internal arcs of 
$\Sbar_{\gamma}$.
\end{lemma}

For each $k=1,2,\ldots,d$, denote by ${x}'_{\,\sigma_k}$ the cluster variable
obtained by mutation from $\xx_{\Tbar_{\gamma}}$ in direction $k$. 

\begin{prop}\label{prop acyclic}
$\Abar_{\gamma}$ is a cluster algebra of type $A_d$, and its initial seed is  acyclic. It follows that  
$\Abar_{\gamma}$ is generated over $\mathbb{Z}\PP$ by the initial $d$ cluster variables and their first mutations, that is, 
the set $\{{x}^{}_{\sigma_1},\ldots,{x}^{}_{\sigma_d},{x}'_{\sigma_1},\ldots,{x}'_{\sigma_d}\}$.
Finally, the ideal of relations among these variables is generated by 
the $d$ exchange relations expressing
${x}_{\sigma_i} {x}'_{\sigma_i}$ in terms of other cluster variables.
\end{prop}
\begin{proof}
$\Abar_{\gamma}$ is of type $A_d$ with acyclic initial seed, because $\Sbar_{\gamma}$ 
is a polygon with $d+3$ vertices, and  each triangle in $\Tbar_{\gamma}$ 
has at least one side on the boundary of $\Sbar_{\gamma}$.
 The last two statements now
  follows from  \cite[Theorem 1.20 and Corollary 1.21]{BFZ}.
\end{proof}

\subsection{The map $\phi_{\gamma}$}

We will now define a homomorphism $\phi_{\gamma}$ of $\mathbb{Z}$-algebras from the cluster algebra $\Abar_{\gamma}$ 
to the field of fractions $\Frac(\mathcal{A})$ of the cluster algebra $\mathcal{A}$. 
We will define $\phi_{\gamma}$ on a set of generators of $\Abar_{\gamma}$ and 
then show that it is a well-defined homomorphism, by checking that the image
of the $d$ exchange relations from Proposition \ref{prop acyclic} are 
relations in $\Frac(\mathcal{A})$.

\begin{subsubsection}{Definition of $\phi_{\gamma}$ on the variables corresponding
to arcs of $\Tbar_{\gamma}$} 

If $\sigma_j$ is an internal arc or boundary segment of $\Tbar_{\gamma}$ (so $1\le j\le 2d+3$),  define
\begin{equation}\label{eq def 1}
\phi_{\gamma}({x}_{\,\sigma_j})= x_{\pi(\sigma_j)}.
\end{equation}
We make the convention that if $\pi(\sigma_j)$ is a boundary segment of $S$, then 
$x_{\pi(\sigma_j)} = 1$.  
Also recall that if $\pi(\sigma_j)$ is  a loop in a self-folded triangle then the notation $x_{\pi(\sigma_j)}$ stands for the product $x_{ r} x_{r^{(p)}}$, where $r$ is the radius  and $p$ is the puncture in the self-folded triangle.
\end{subsubsection}

\begin{subsubsection}{Definition of $\phi_{\gamma}$ on the first mutations of the initial cluster variables}
Define

\begin{equation}\label{eq def 2}
\phi_{\gamma}({x}'_{\,\sigma_j})= \left\{\begin{array}{ll}
{x'_{\pi(\sigma_j)}}
& \textup{if $\pi(\sigma_j)$ is not a loop or a radius;}\\ \\
{x_e} & \textup{if $\pi(\sigma_j)$ is  a loop, where $e$ is obtained by flipping $\pi(\sigma_j)$;}\\ \\
\left(1+\dfrac{y_{r}}{y_{r^{(p)}}}
\right) {x_{r} x_{r^{(p)}}}
& \textup{if $\pi(\sigma_j)$ is a radius $r$ to a puncture $p$}.
\end{array}\right.
\end{equation}

\end{subsubsection}

\begin{subsubsection}{Definition of $\phi_{\gamma}$ on the coefficients $y_{\sigma_j}$} 

Define 

\begin{equation}\label{eq def 3}
\phi_{\gamma}({y}_{\,\sigma_j})= \left\{\begin{array}{ll}
y_{\pi(\sigma_j)}\
&\textup{if $\pi(\sigma_j)$ is not a loop or a radius;}\\ \\

\dfrac{y_{r}}{y_{r^{(p)}}}
&\textup{if $\pi(\sigma_j)$ is a radius $r$ to a puncture $p$;}\\ \\
y_{r^{(p)}}
&\textup{if $\pi(\sigma_j)$ is  a loop enclosing the radius $r$ and puncture $p$.}
\end{array}\right.
\end{equation}
\end{subsubsection}

\begin{subsubsection}{Definition of $\phi_{\gamma}$ on the whole cluster algebra} 
By Proposition \ref{prop acyclic}, defining $\phi_{\gamma}$ on 
the cluster variables and their first mutations, as well as on the 
generators of the coefficient group, is enough to define 
a homomorphism of $\ZZ$-algebras $\phi_{\gamma}$
from $\Abar_{\gamma}$, 
provided that $\phi_{\gamma}$ is well-defined.
Note that 
$\phi_{\gamma}$ is a map
from $\Abar_{\gamma}$ to the field of fractions
$\Frac(\mathcal{A})$ of $\mathcal{A}$, rather than a map to 
$\mathcal{A}$ itself.

\begin{prop}
The map $\phi_{\gamma}$ is a well-defined homomorphism of $\mathbb{Z}$-algebras
\[ \phi_{\gamma}:\,\Abar_{\gamma}\,\to \Frac(\mathcal{A}).\]
\end{prop}

\begin{proof}
By 
Proposition \ref{prop acyclic}, it suffices to show that  $\phi_{\gamma}$ maps
the $d$ exchange relations involving 
${x}_{\sigma_j} {x}'_{\sigma_j}$ to relations in 
$\mathcal{A}$.
We prove this by checking three cases:  $\pi(\sigma_j)$
is not a loop or radius;  $\pi(\sigma_j)$ is a loop enclosing a radius $r$; 
and  $\pi(\sigma_j)$ is a radius $r$.

In all cases, the exchange relation in $\Abar_{\gamma}$ is determined by 
the quadrilateral in $\Tbar_{\gamma}$ with diagonal $\sigma_j$, which projects
via $\pi$ to the quadrilateral in $T$ with diagonal $\pi(\sigma_j)$.  
Note that in all cases, the exchange relation in $\Abar_{\gamma}$ has the form 
\begin{equation} \label{er}
{x}_{\sigma_j} {x}'_{\sigma_j} = {y}_{\sigma_j} 
     \prod_b {x}_b + \prod_c {x}_c , 
\end{equation}
where $b$ ranges over all arcs in ${T}$ following $\sigma_j$
in clockwise order, and $c$ ranges over all arcs in ${T}$
following $\sigma_j$ in counterclockwise order.

In the first case (when $\pi(\sigma_j)$ is not a loop or radius),
the local configuration of the triangulation 
is either that of Figure \ref{figquad-case1a} or 
Figure \ref{figquad-case1b}.
\begin{figure}
\begin{center}
\input{figquad-case1a.pstex_t}
\end{center}
\caption{}
\label{figquad-case1a}
\end{figure}
\begin{figure}
\begin{center}
\input{figquad-case1b.pstex_t}
\end{center}
\caption{}
\label{figquad-case1b}
\end{figure}
The image of the exchange relation under $\phi_{\gamma}$ is 
$$x_{\pi(\sigma_j)} {x'_{\pi(\sigma_j)}} = 
y_{\pi(\sigma_j)}  \prod_b x_{\pi(b)} + \prod_c x_{\pi(c)}.$$  
This is exactly the corresponding exchange relation (``Ptolemy relation") 
in $\mathcal A$.

Note that in theory we also need to 
consider configurations such as that in 
Figure \ref{quad-caseextra}, where one or both of the arcs 
\begin{figure}
\begin{center}
\input{figquad-caseextra.pstex_t}
\end{center}
\caption{}
\label{quad-caseextra}
\end{figure}
$\pi(\sigma_{j-1})$ and $\pi(\sigma_{j+1})$ are loops cutting
out once-punctured monogons with puncture $p$ and radius $r$.  
If say $\pi(\sigma_{j-1})$ is such a 
loop, then 
the image of the exchange relation in ${\A}$ contains 
$x_{\pi(\sigma_{j-1})} = x_{r} x_r^{(p)}$.
However, the resulting relation will still be an exchange relation 
in $\mathcal A$ (a ``generalized Ptolemy relation"), by 
\cite[Proposition 6.5, Lemma 7.2, and Definition 7.4]{FT}. 

Now consider the case that $\pi(\sigma_j)$ is a loop 
enclosing the radius $r$ and puncture $p$.  See Figure \ref{quad-case2}
\begin{figure} \begin{center}
\input{figquad-case2.pstex_t}
\end{center}
\caption{}
\label{quad-case2}
\end{figure}
for the local configuration.  Note that 
$\pi(\sigma_{j-1})=\pi(\alpha_1) = r$.
In this case 
the exchange relation (\ref{er}) 
is equal to $${x}_{\sigma_j} {x}'_{\sigma_j} = 
{y}_{\sigma_j} {x}_{\sigma_{j-1}} x_{\alpha_2} +{x}_{\sigma_{j+1}}x_{\alpha_1}$$ and its image 
under $\phi_{\gamma}$ is 
$$x_{\pi(\sigma_j)} x_e = y_{r^{(p)}} 
  x_{\pi(\alpha_2)} x_r + x_{\pi(\sigma_{j+1})} x_r,$$  
where $e$ is the arc obtained by 
flipping $\pi(\sigma_j)$.
Since $x_{\pi(\sigma_j)} = x_r x_{r^{(p)}},$
dividing  by $x_r$ yields 
exactly the exchange relation for 
$x_{r^{(p)}} x_e$ in $\mathcal A$, 
see equation (7.1) of \cite{FT}.

Finally consider the case that $\pi(\sigma_j)$ is a radius $r$
to a puncture $p$;
let $\ell$ denote the corresponding loop around the puncture.
See Figure \ref{quad-case3} for the local configuration.
Note that the two boundary segments on the left-hand-side 
of the figure project to $\pi(\sigma_j)$.
\begin{figure} \begin{center}
\input{figquad-case3.pstex_t}
\end{center}
\caption{}
\label{quad-case3}
\end{figure}
In this case the image of the exchange relation (\ref{er})
under $\phi_{\gamma}$ is 
$$x_{\pi(\sigma_j)}(1+\frac{y_r}{y_{r^{(p)}}})
{x_r x_{r^{(p)}}} =
\frac{y_r}{y_{r^{(p)}}} x_{\ell} x_{\pi(\sigma_j)}+ 
x_{\ell} x_{\pi(\sigma_j)}.$$  
Since $x_{\ell} = x_r x_{r^{(p)}}$,
this is an identity.
This completes the proof.
\end{proof}
\end{subsubsection}


\section{Quadrilateral lemma}\label{sec simply connected}

\begin{lemma}\label{lem exchange}
Let $T=\{\tau_1,\ldots,\tau_{n+c}\}$  be an ideal triangulation of $(S,M)$, and let
  ${\somediagonalgreekletter}$ be an arc in $(S,M)$ which is not in $T$.
 Let $e({\somediagonalgreekletter},T)$ be the number of crossings
  between ${\somediagonalgreekletter}$ and $T$.
Then there exist five, not necessarily distinct, arcs or boundary arcs
  $\someothergreekletter_1,\,\someothergreekletter_2,\,\somegreekletter_3,\,\somegreekletter_4$ and ${\somediagonalgreekletter'}$ in $(S,M)$ such that
\begin{itemize}
\item[(a)] each of $\someothergreekletter_1,\,\someothergreekletter_2,\,\somegreekletter_3, \,\somegreekletter_4$  and ${\somediagonalgreekletter'}$ crosses $T$  fewer than $e({\somediagonalgreekletter},T)$ times,
\item[(b)] $\someothergreekletter_1,\,\someothergreekletter_2,\,\somegreekletter_3,\,\somegreekletter_4$ are the sides of an ideal
  quadrilateral with simply connected interior
in which ${\somediagonalgreekletter}$ and ${\somediagonalgreekletter'}$ are the diagonals.
\end{itemize}
\end{lemma}

\begin{proof} We prove the lemma by induction on $k=e({\somediagonalgreekletter},T)$. If
  $k=1$, then let ${\somediagonalgreekletter'}\in T$ be the unique arc that crosses
  ${\somediagonalgreekletter}$.
  Then $\somediagonalgreekletter'$ is a side of exactly two triangles in $T$. We distinguish three cases according to how many of these triangles are self-folded, see Figure \ref{figtriangles}.
\begin{figure}
\begin{center}
\input{figtriangles.pstex_t}
\caption{Configurations of the ideal triangles incident to $\somediagonalgreekletter'$.}
\label{figtriangles}
\end{center}
\end{figure}
 \begin{enumerate}
\item If none of the two triangles is self-folded, then let $\someothergreekletter_1,\somegreekletter_1$ and $\somediagonalgreekletter'$ denote the three sides of one of them, and let $\someothergreekletter_3,\somegreekletter_4$ and $\somediagonalgreekletter'$ denote the three sides of the other, such that $\someothergreekletter_1$ and $\someothergreekletter_3$ (and hence also $\somegreekletter_2$ and $\somegreekletter_4$) are opposite sides in the quadrilateral formed by the union of the two triangles. Then these arcs satisfy (a) and (b), see the
left of Figure \ref{figtriangles}.
 \item If one of the two triangles is self-folded, then let $\somegreekletter_2$ and $\somediagonalgreekletter'$ denote the two sides of the self-folded triangle, and let $\someothergreekletter_1,\somegreekletter_3$ and $\somediagonalgreekletter'$ denote the three sides of the other triangle.  Since $\somediagonalgreekletter$ crosses $\somediagonalgreekletter'$ but not $ \somegreekletter_2$, it follows that $\somediagonalgreekletter'$ is the loop of the self-folded triangle and $  \somegreekletter_2$ is its radius. Setting $\someothergreekletter_2=\somegreekletter_2$, we obtain five arcs that satisfy conditions (a) and (b),
see the middle of Figure \ref{figtriangles}.

\item The case where both triangles are self-folded is actually impossible, because two self-folded triangles that share a side can only occur on the sphere with three punctures, but this surface is not allowed, see the right of Figure
\ref{figtriangles}.
\end{enumerate}

  Suppose that $k\ge 2$. Choose an orientation of ${\somediagonalgreekletter}$ and denote its
starting point by $a$ and its endpoint by $b$ (note that $a$ and $b$
may coincide). Label the $k$ crossing points of ${\somediagonalgreekletter}$ and $T$
by $1,2,\ldots,k$ according to their order on ${\somediagonalgreekletter}$, such that the
crossing point $1$ is the closest to $a$.
Let $h$ be the middle crossing point, more precisely, 
let $h=\lceil k/2 \rceil$. 
 Denote by $\tau$ the unique arc of the
triangulation $T$ that crosses $\somediagonalgreekletter$ at the point with label $h$, and let $r=e(\tau,\somediagonalgreekletter)$ be the number of crossings between $\tau$ and $\somediagonalgreekletter$. Choose an orientation of $\tau$ and denote its
starting point by $c$ and its endpoint by $d$ (note again that $c$ and $d$
may coincide).
As before with
${\somediagonalgreekletter}$, label the $r$ crossing
points of $\tau$ and $\somediagonalgreekletter$ by $j_1,j_2,\ldots,j_r$ according to
their order on $\tau$  (see Figure \ref{figtau}). Thus $r\le k$,
$\{j_1,j_2,\ldots,j_r\}\subset\{1,2,\ldots,k\}$. Note that $s<t$
does \emph{not} imply $j_s<j_t$.
Choose $\ell$ so that $j_\ell=h$ is the middle crossing point.

\begin{figure}
\centering
\input{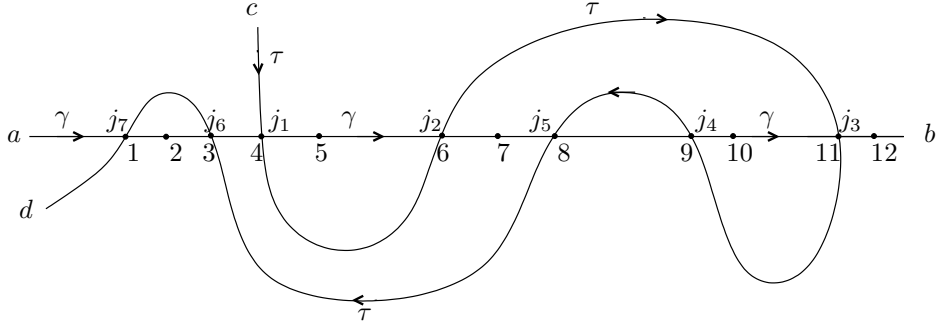}
\caption{Labeling of the crossing points of $\somediagonalgreekletter$ and $\tau$. Here $k=12$ and $j_\ell=j_2=h=6$.}\label{figtau}
\end{figure}

We will use $\tau$ and ${\somediagonalgreekletter}$ to construct the five arcs (or boundary arcs) of the
lemma. Let ${\somediagonalgreekletter}^-$ (respectively  $\tau^-$) denote the curve ${\somediagonalgreekletter}$
(respectively $\tau$) with the opposite orientation.
We will distinguish four cases:
\begin{enumerate}

\item  \emph{$(\ell = 1$ or $j_{\ell-1}<j_\ell)$ and $( \ell =r $ or $j_{\ell+1}> j_\ell)$.}
 We define the
  arcs below, and we illustrate them as the dashed arcs in Figure \ref{figrho'}, continuing  the example of Figure \ref{figtau}. Suppose first that $1<\ell < r$.

Let
\[{\somediagonalgreekletter'}=(a,j_{\ell-1},j_{\ell+1},b\mid {\somediagonalgreekletter},\tau,{\somediagonalgreekletter})
\]
be the arc
  that starts  at $a$ and is homotopic to ${\somediagonalgreekletter}$
  up to the crossing point $j_{\ell-1}$, then, from $j_{\ell-1}$ to
  $j_{\ell+1}$, ${\somediagonalgreekletter'} $ is homotopic to $\tau$, and from $j_{\ell+1}$ to
  $b$, ${\somediagonalgreekletter'}$ is homotopic to ${\somediagonalgreekletter}$. Note that ${\somediagonalgreekletter'}$ and ${\somediagonalgreekletter}$ cross
  exactly once, namely at  $j_\ell$.

\begin{figure}
\centering
\input{figrhoprime.pstex_t}

\input{figalpha.pstex_t}

\input{figbeta.pstex_t}
\caption{Construction of $\somediagonalgreekletter',\someothergreekletter_1,\someothergreekletter_2,\somegreekletter_3$ and $\somegreekletter_4$ in case (1).}\label{figrho'}
\end{figure}

In a similar way, we define
\[
\begin{array}{cc}
 \someothergreekletter_1=(a,j_{\ell-1},j_{\ell},a\mid {\somediagonalgreekletter},\tau,{\somediagonalgreekletter}^-)
&\someothergreekletter_3=(b,j_{\ell+1},j_{\ell},b\mid {\somediagonalgreekletter}^-,\tau^-,{\somediagonalgreekletter}) \\
 \somegreekletter_2=(a,j_{\ell},j_{\ell+1},b\mid {\somediagonalgreekletter},\tau,{\somediagonalgreekletter})
&\somegreekletter_4=(b,j_{\ell},j_{\ell-1},a\mid {\somediagonalgreekletter}^-,\tau^-,{\somediagonalgreekletter}^-).
\end{array}
\]
In the special case where $\ell=1$, (respectively $l=r$), we define
\[
\begin{array}{cc}
{\somediagonalgreekletter'}=(c,j_{\ell+1},b\mid \tau,{\somediagonalgreekletter}) & (\textup{respectively }
{\somediagonalgreekletter'}=(a,j_{\ell-1},d\mid {\somediagonalgreekletter},\tau)\\
\someothergreekletter_1=(c,j_{\ell},a\mid \tau,{\somediagonalgreekletter}^-) & (\textup{respectively }
\someothergreekletter_3=(d,j_{\ell},b\mid \tau^-,{\somediagonalgreekletter})\\
\somegreekletter_4=(b,j_{\ell},c\mid {\somediagonalgreekletter}^-,\tau^-) & (\textup{respectively }
\somegreekletter_2=(a,j_{\ell},d\mid {\somediagonalgreekletter},\tau),
\end{array}  \]
where $c$ is the starting point of $\tau $ and $d$ is its endpoint.  In particular, if $\ell=r=1$ then $\rho'=\tau$.

Then $\someothergreekletter_1,\somegreekletter_2,\someothergreekletter_3,\somegreekletter_4$ form a quadrilateral with simply connected
interior
such that
$\someothergreekletter_1$ and $\someothergreekletter_3$ are opposite sides, $\somegreekletter_2$ and $\somegreekletter_4$ are
opposite sides, and ${\somediagonalgreekletter}$ and ${\somediagonalgreekletter'}$ are the diagonals.
The topological type of this quadrilateral
is as in the left-hand-side of Figure \ref{quadtypes}.
\begin{figure}
\centering

\input{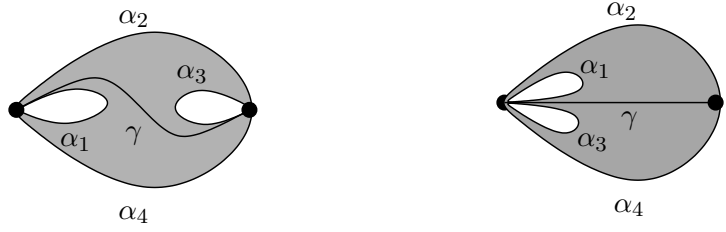}
\caption{Different topological types of quadrilaterals}\label{quadtypes}
\end{figure}
This shows (b).

 It remains to show (a). By hypothesis, we have
$j_{\ell-1}<j_\ell=h$ and $j_{\ell+1}>j_\ell=h$. Moreover, since the crossing points $j_{\ell-1}$, and $j_\ell$ both lie on the same arc $\tau$ of the ideal triangulation, the arc $\somediagonalgreekletter$ must cross some other arc between the two crossings at $j_{\ell-1} $ and $j_\ell$; in other words,  $j_{\ell-1}<j_\ell -1=h-1$. Similarly $j_{\ell+1}>j_\ell+1=h+1$. Also recall that $k\le 2h\le k+1$.
Then,

\[\begin{array}{ccccccc}
e({\somediagonalgreekletter'},T)&=&(j_{\ell-1}-1)+(k-j_{\ell+1}+1) &< & h-2+h+1&\le& k \\
e(\someothergreekletter_1,T)&=&(j_{\ell-1}-1)+j_{\ell} &<&h-2+h&\le &k \\
e(\someothergreekletter_3,T)&=&(k-j_{\ell+1})+(k-j_{\ell}+1) &<& k-h-1+k-h+1& \le& k\\
e(\somegreekletter_2,T)&=&(j_{\ell}-1)+(k-j_{\ell+1})& <&h-1+k-h-1&\le & k \\
e(\somegreekletter_4,T)&=&(k-j_{\ell})+(j_{\ell-1}-1) &<& k-h+ h-2 &\le& k .
\end{array}  \]
 In the case where $\ell=1$, we have
\[
\begin{array}{ccccl}
e({\somediagonalgreekletter'},T)&=& k-j_{\ell+1} &<& k\\
e(\someothergreekletter_1,T)&=& j_{\ell} -1 &<& k\\
e(\somegreekletter_4,T)&=& k-j_{\ell} &<& k,
\end{array}  \]
and in the case where $\ell=r$, we have
\[
\begin{array}{ccccl}
e({\somediagonalgreekletter'},T)&=& j_{\ell-1}-1 &<& k\\
e(\someothergreekletter_3,T)&=& k-j_{\ell}  &<& k\\
e(\somegreekletter_2,T)&=& j_{\ell}-1 &<& k.
\end{array}  \]

This shows (a).

\item \emph{$(\ell = 1$ or $j_{\ell-1}<j_\ell)$ and $( \ell =r $ or $j_{\ell+1}< j_\ell)$}. This case is illustrated in Figure \ref{figrho2}. Suppose first that $1<\ell < r$.

\begin{figure}
\centering
\input{figrho2.pstex_t}
\caption{Construction of $\somediagonalgreekletter',\someothergreekletter_1,\someothergreekletter_2,\somegreekletter_3$ and $\somegreekletter_4$ in case (2).}\label{figrho2}
\end{figure}

Let ${\somediagonalgreekletter'}=(a,j_{\ell-1},j_{\ell+1},a\mid {\somediagonalgreekletter},\tau,{\somediagonalgreekletter}^-)$  be the arc
  that starts at at $a$ and is homotopic to ${\somediagonalgreekletter}$
  up to the crossing point $j_{\ell-1}$, then, from $j_{\ell-1}$ to
  $j_{\ell+1}$, ${\somediagonalgreekletter'} $ is homotopic to $\tau$, and from $j_{\ell+1}$ to
  $a$, ${\somediagonalgreekletter'}$ is homotopic to ${\somediagonalgreekletter^-}$. Note that ${\somediagonalgreekletter'}$ and ${\somediagonalgreekletter}$ cross
  exactly once, namely at the point $j_\ell$.
In a similar way, let
\[
\begin{array}{cc}
 \someothergreekletter_1=(a,j_\ell,j_{\ell-1},a\mid {\somediagonalgreekletter},\tau^-,{\somediagonalgreekletter}^-)
&\someothergreekletter_3=(b,j_\ell,j_{\ell+1},a\mid {\somediagonalgreekletter}^-,\tau,{\somediagonalgreekletter}^-) \\
 \somegreekletter_2=(a,j_{\ell-1},j_{\ell},b\mid {\somediagonalgreekletter},\tau,{\somediagonalgreekletter})
&\somegreekletter_4=(a,j_{\ell+1},j_{\ell},a\mid {\somediagonalgreekletter},\tau^-,{\somediagonalgreekletter}^-)
\end{array}
\]
In the special case where $\ell=1$, (respectively $\ell=r$), we define
\[
\begin{array}{cc}
{\somediagonalgreekletter'}=(c,j_{\ell+1},a\mid \tau, {\somediagonalgreekletter}^-) & (\textup{respectively }
{\somediagonalgreekletter'}=(a,j_{\ell-1},d\mid {\somediagonalgreekletter}, \tau)\\
\someothergreekletter_1=(a,j_{\ell},c\mid {\somediagonalgreekletter},\tau^-) & (\textup{respectively }
\someothergreekletter_3=(b,j_{\ell},d\mid {\somediagonalgreekletter}^-,\tau)\\
\somegreekletter_2=(c,j_{\ell},b\mid \tau,{\somediagonalgreekletter}) & (\textup{respectively }
\somegreekletter_4=(d,j_{\ell},a\mid \tau^-,{\somediagonalgreekletter}^-),
\end{array}  \]
where $c$ is the starting point of $\tau $ and $d$ is its endpoint. Note again that $\rho'=\tau $ if $\ell=r=1$.

Then $\someothergreekletter_1,\somegreekletter_2,\someothergreekletter_3,\somegreekletter_4$ form a quadrilateral with simply connected
interior
such that
$\someothergreekletter_1$ and $\someothergreekletter_3$ are opposite sides, $\somegreekletter_2$ and $\somegreekletter_4$ are
opposite sides, and ${\somediagonalgreekletter}$ and ${\somediagonalgreekletter'}$ are the diagonals.
The topological type of this quadrilateral
is as in the right-hand-side of Figure \ref{quadtypes}.
This shows (b).
It remains to show (a).
By hypothesis, we have
$j_{\ell-1}<j_\ell=h$ and $j_{\ell+1}<j_\ell=h$.
As in case (1), the crossing points $j_{\ell-1}$, and $j_\ell$ both lie on the same arc $\tau$ of the ideal triangulation, and thus the arc $\somediagonalgreekletter$ must cross some other arc between the two crossings at $j_{\ell-1} $ and $j_\ell$; in other words,  $j_{\ell-1}<j_\ell -1=h-1$. Similarly $j_{\ell+1}<j_\ell-1=h-1$. Also recall that $k\le 2h\le k+1$. Then
\[\begin{array}{ccccccc}
e({\somediagonalgreekletter'},T)&=&(j_{\ell-1}-1)+(j_{\ell+1}-1) &<& h-2+h-2 &\le& k \\
e(\someothergreekletter_1,T)&=&(j_{\ell}-1) +j_{\ell-1}&<& h-1+h-1 &\le& k \\
e(\someothergreekletter_3,T)&=&(k-j_{\ell}) + (j_{\ell+1}-1)&<& k-h+h-2  &\le& k \\
e(\somegreekletter_2,T)&=&(j_{\ell-1}-1)+(k-j_{\ell}) &<& h-2 +k-h &\le& k \\
e(\somegreekletter_4,T)&=&(j_{\ell+1}-1)+j_{\ell} &<&  h-2+h &\le& k.
\end{array}  \]
 In the case where $\ell=1$, we have
\[
\begin{array}{ccccl}
e({\somediagonalgreekletter'},T)&=& j_{\ell+1}-1 &<& k\\
e(\someothergreekletter_1,T)&=& j_{\ell}-1 &< &k\\
e(\somegreekletter_2,T)&=& k-j_{\ell} &< &k,
\end{array}  \]
and in the case where $\ell=r$, we have
\[
\begin{array}{ccccl}
e({\somediagonalgreekletter'},T)&=& j_{\ell-1}-1 &<& k\\
e(\someothergreekletter_3,T)&=&k-j_{\ell} &<& k\\
e(\somegreekletter_4,T)&=&  j_{\ell}-1 &<& k .
\end{array}  \]

This shows (a).\\
\item \emph{$j_{\ell-1}>j_\ell$ and $j_{\ell+1}<j_\ell$.} This case
  follows from the case (1) by symmetry.

\item \emph{$j_{\ell-1}>j_\ell$ and $j_{\ell+1}>j_\ell$.} This case
  follows from the case (2) by symmetry.

\end{enumerate}

\end{proof}

\section{The proof of the expansion formula for ordinary arcs} \label{sec plain}

The main technical lemma we need in order to complete the proof of our expansion formula
for ordinary arcs is that $\phi_{\gamma}(x_{\tilde{\gamma}}) = x_{\gamma}$.
Once we have this, the proof of our expansion formula for ordinary arcs
will follow easily.

\subsection{The proof that $\phi_{\gamma}(x_{\tilde{\gamma}}) = x_{\gamma}$.}\label{sec phi-map}

In this section we show that the constructions of 
$\Sbar_{\gamma}$ and $\Tbar_{\gamma}$ in Section \ref{sect Sbar} are
compatible with the map $\phi_{\gamma}$ defined in Section \ref{sect Abar}
in a sense which we make precise in Theorem \ref{phi-map}.

Fix a bordered surface $(S,M)$, an ideal triangulation 
$T=(\tau_1,\dots,\tau_n)$, 
and let $\Acal$ be the corresponding cluster algebra with principal
coefficients with respect to $T$.
Also fix an arc $\gamma$ in $S$.  This gives
rise to a polygon
$\Sbar_{\gamma}$ with a triangulation $\Tbar_{\gamma} = 
(\sigma_1^{\gamma},\dots,\sigma_d^{\gamma})$, 
a lift $\tilde{\gamma}$ of $\gamma$
in $\Sbar_{\gamma}$,
a cluster algebra
$\Abar_{\gamma}$, a projection 
$\pi: \Tbar_{\gamma} \to T$, and a homomorphism 
$$\phi_{\gamma}: \Abar_{\gamma} \to \Frac(\Acal),$$ such that 
$\phi_{\gamma}(x_{\sigma_j^{\gamma}}) = x_{\pi(\sigma_j^{\gamma})}.$

\begin{theorem}\label{phi-map}
Using the notation of the previous paragraph, we have that 
$$\phi_{\gamma}(x_{\tilde{\gamma}}) = x_{\gamma}.$$
\end{theorem}

\begin{proof}
We prove Theorem \ref{phi-map} by induction on the number of crossings of $\gamma$
and $T$.  When this number is zero, there is nothing to prove.  
Otherwise,
by Lemma \ref{lem exchange}, 
there exists a quadrilateral $Q$ in $S$ with simply-connected interior,
which has diagonals $\gamma$ and $\gamma'$, and sides $\alpha_1, \alpha_2, \alpha_3, \alpha_4$.
Moreover, each of $\gamma'$ and the four sides
crosses $T$ fewer times than $\gamma$ does.
See Figure \ref{SurfaceQuad} for an example of a surface, an arc
$\gamma$ (bold), and the corresponding quadrilateral with sides 
$\alpha_1, \dots, \alpha_4$ and diagonals
$\gamma$ and $\gamma'$, as well as the triangulated polygons
$\Sbar_{\gamma}$ and $\Sbar_{\gamma'}$.
\begin{figure}
\centering
\input{Surface-Quad.pstex_t} \hspace{3em} \input{Tri-Polygons.pstex_t}
\caption{}\label{SurfaceQuad}
\end{figure}

By induction, we have five triangulated polygons 
$\Sbar_{\gamma'}$, $\Sbar_{\alpha_1}, \dots ,\Sbar_{\alpha_4}$,
five lifts
$\tilde{\gamma'}$, $\tilde{\alpha_1},\dots ,\tilde{\alpha_4}$, 
in the respective
polygons, five associated cluster algebras, 
and five different homomorphisms
$\phi_{\gamma'}$, 
$\phi_{\alpha_1},\dots,
\phi_{\alpha_4}$, 
such that 
\begin{equation}\label{inductive-step}
\phi_{\gamma'}(x_{\gamma'}) = x_{\gamma'}, \
\phi_{\alpha_i}(x_{\alpha_i}) = x_{\alpha_i},\ 
\text{ for }i=1\dots4.
\end{equation}

We let 
$\Tbar_{\gamma'}=\{\sigma_j^{\gamma'}\}_j$
and $\Tbar_{\alpha_i}=\{\sigma_j^{\alpha_i}\}_j$ denote
the triangulations of 
$\Sbar_{\gamma'}$ and $\Sbar_{\alpha_i}$ 
for $1\leq i \leq 4$.  Because 
$\Sbar_{\gamma}$, $\Sbar_{\gamma'}$, and $\Sbar_{\alpha_i}$ 
are polygons, the corresponding cluster algebras 
$\Abar_{\gamma}, \Abar_{\gamma'}$ and $\Abar_{\alpha_i}$ 
are of type A.  In type A, the cluster expansion formulas
(in terms of weights and heights of matchings) of 
Section \ref{sect cluster expansion formula} are already known
to be true \cite{MS}.  Therefore,
letting $\Match_{S,T,\gamma}(x_{\tau},y_{\tau})$ denote the formula
given by Theorem \ref{thm main}, 
 for $i=1,2,3,4$ we have
\begin{align}
[x_{\tilde{\gamma}}]_{\xx_{\Tbar_{\gamma}}}^{\Abar_{\gamma}} &= 
\Match_{\Sbar_{\gamma}, \Tbar_{\gamma}, \tilde{\gamma}} 
(x_{\sigma_j^{\gamma}}, 
y_{\sigma_j^{\gamma'}}) \\
[x_{\tilde{\gamma'}}]_{\xx_{\Tbar_{\gamma'}}}^{\Abar_{\gamma'}} &= 
\Match_{\Sbar_{\gamma'}, \Tbar_{\gamma'}, \tilde{\gamma'}} 
(x_{\sigma_j^{\gamma'}}, 
y_{\sigma_j^{\gamma'}}) \\
[x_{\tilde{\alpha_i}}]_{\xx_{\Tbar_{\alpha_i}}}^{\Abar_{\alpha_i}} &= 
\Match_{\Sbar_{\alpha_i}, \Tbar_{\alpha_i}, \tilde{\alpha_i}} 
(x_{\sigma_j^{\alpha_i}}, 
y_{\sigma_j^{\alpha_i}}).
\end{align}

Using  (\ref{inductive-step}) and the fact that 
$\phi_{\gamma}$, $\phi_{\gamma'}$ and
$\phi_{\alpha_i}$
are homomorphisms, we have
\begin{align}\label{align1}
 \phi_{\gamma}(x_{\tilde{\gamma}}) &=
\Match_{\Sbar_{\gamma}, \Tbar_{\gamma}, \tilde{\gamma}} 
(x_{\pi(\sigma_j^{\gamma})}, 
\phi_{\gamma}(y_{\sigma_j^{\gamma}})) \\
x_{\gamma'} = \phi_{\gamma'}(x_{\tilde{\gamma'}}) &=
\Match_{\Sbar_{\gamma'}, \Tbar_{\gamma'}, \tilde{\gamma'}} 
(x_{\pi(\sigma_j^{\gamma'})}, 
\phi_{\gamma'}(y_{\sigma_j^{\gamma'}})) \\
x_{\alpha_i} = \phi_{\alpha_i}(x_{\tilde{\alpha_i}}) &=\label{align2}
\Match_{\Sbar_{\alpha_i}, \Tbar_{\alpha_i}, \tilde{\alpha_i}} 
(x_{\pi(\sigma_j^{\alpha_i})}, 
\phi_{\alpha_i}(y_{\sigma_j^{\alpha_i}})) 
\end{align}

The matching formulas above use the combinatorics of 
six different triangulated polygons.  We would like to view them 
all  inside one polygon.  To this end, consider
the triangulated polygons $\Sbar_{\gamma}$ and $\Sbar_{\gamma'}$.
Because $\gamma$ and $\gamma'$ intersect (exactly once) in $S$,
the local neighborhoods around the corresponding points
in $\Sbar_{\gamma}$ and $\Sbar_{\gamma'}$ coincide (there are at least
two triangles in common and perhaps more).  Therefore we can glue
$(\Sbar_{\gamma}, \Tbar_{\gamma})$ and $(\Sbar_{\gamma'}, \Tbar_{\gamma'})$
together at this point, getting a larger polygon
$\widehat{S}$ with triangulation $\widehat{T}=\{\sigma_j\}_j$.
Figure \ref{SurfaceQuad} shows 
the triangulated polygons 
$\Sbar_{\gamma}$ and $\Sbar_{\gamma'}$
associated to $\gamma$ and $\gamma'$, and Figure \ref{Big-Polygon}
shows the polygon $\widehat{S}$ obtained by glueing them together.
\begin{figure}
\centering
\input{Big-Polygon.pstex_t}
\caption{}\label{Big-Polygon}
\end{figure}
Abusing notation, we denote the projection 
$\widehat{T} \to T$ by $\pi$.

Up to isotopy, there are now four unique arcs which form the boundary of a 
quadrilateral $\tilde{Q}$ with diagonals
$\tilde{\gamma}$ and $\tilde{\gamma'}$, and these four arcs together with 
the set of triangles they pass through are
isomorphic to $\tilde{\alpha_i}$ in $\Sbar_{\alpha_i}$. 
Therefore we can view the polygons
$\Sbar_{\gamma'}$ and $\Sbar_{\alpha_i}$ 
and the arcs $\tilde{\gamma'}$ and $\tilde{\alpha_i}$
as sitting inside $\widehat{S}$.
Using this observation together with 
equations (\ref{align1}) to (\ref{align2}), we get
\begin{align}
\phi_{\gamma}(x_{\tilde{\gamma}})  &=
\Match_{\widehat{S},\widehat{T}, \tilde{\gamma}} \label{eq useful}
(x_{\pi(\sigma_j)}, 
\phi(y_{\sigma_j})) \\
x_{\gamma'}  &= \label{align4}
\Match_{\widehat{S},\widehat{T}, \tilde{\gamma'}} 
(x_{\pi(\sigma_j)}, 
\phi(y_{\sigma_j})) \\
x_{\alpha_i}   &=\label{align5}
\Match_{\widehat{S}, \widehat{T}, \tilde{\alpha_i}} 
(x_{\pi(\sigma_j)}, 
\phi(y_{\sigma_j})), 
\end{align}
where $\phi$ denotes the specialization of $y$-variables from 
equation (\ref{eq def 3}).

Because $\gamma, \gamma'$ and the $\alpha_i$ form a quadrilateral
in $S$, we have a generalized Ptolemy relation in $\Acal$ of the form
\begin{equation} \label{S-Ptolemy}
x_{\gamma} x_{\gamma'} = Y_+ x_{\alpha_1} x_{\alpha_3} + Y_- x_{\alpha_2} x_{\alpha_4},
\end{equation}
where $Y_+$ and $Y_-$ can be computed using the elementary laminations
associated to the arcs of the triangulation $T$.

On the other hand, since $\tilde{\gamma}, \tilde{\gamma'}, 
\tilde{\alpha_i}$ form a quadrilateral in 
$\widehat{S}$, we have a generalized Ptolemy relation 
in $\Aprin(\widehat{S})$ of the form 
\begin{equation} \label{Lift-Ptolemy}
x_{\tilde{\gamma}} x_{\tilde{\gamma'}} = \tilde{Y}_+ x_{\tilde{\alpha_1}} x_{\tilde{\alpha_3}} + \tilde{Y}_- x_{\tilde{\alpha_2}} x_{\tilde{\alpha_4}},
\end{equation}
where again $\tilde{Y}_+$ and $\tilde{Y}_-$ can be computed using 
the elementary laminations associated to the arcs of the 
triangulation $\widehat{T}$.

Because the cluster expansion formulas hold in type A, we can 
rewrite  (\ref{Lift-Ptolemy}) as 
\begin{align*}
\Match_{\widehat{S}, \widehat{T}, \tilde{\gamma}} 
(x_{\sigma_j}, y_{\sigma_j}) 
{\Match_{\widehat{S}, \widehat{T}, \tilde{\gamma'}}
(x_{\sigma_j}, y_{\sigma_j})} &=
\tilde{Y}_+ \Match_{\widehat{S}, \widehat{T}, \tilde{\alpha_1}}
(x_{\sigma_j}, y_{\sigma_j})
\Match_{\widehat{S}, \widehat{T}, \tilde{\alpha_3}}(x_{\sigma_j}, y_{\sigma_j})\\
&+ \tilde{Y}_- \Match_{\widehat{S}, \widehat{T}, \tilde{\alpha_2}}
(x_{\sigma_j}, y_{\sigma_j})
\Match_{\widehat{S}, \widehat{T}, \tilde{\alpha_4}}(x_{\sigma_j}, y_{\sigma_j}).
\end{align*} 

Now we specialize variables in the above formula, substituting 
$x_{\pi(\sigma_j)}$ for each $x_{\sigma_j}$ and 
substituting $\phi(y_{\sigma_j})$ for each 
$y_{\sigma_j}$ as 
in equation (\ref{eq def 3}). Using equations
(\ref{align4}) to (\ref{align5}), this gives us
\begin{equation}\label{almost}
\Match_{\widehat{S}, \widehat{T}, \tilde{\gamma}} 
(x_{\pi(\sigma_j)}, \phi(y_{\sigma_j})) \ x_{\gamma'} = 
\phi(\tilde{Y}_+) x_{\alpha_1} x_{\alpha_3} + 
\phi(\tilde{Y}_-) x_{\alpha_2} x_{\alpha_4}.
\end{equation}

We now state a lemma, which we will use immediately and prove 
momentarily.
\begin{lemma}\label{lemma-project}
$\phi(\tilde{Y_+}) = Y_+$ and 
$\phi(\tilde{Y_-}) = Y_-$. 
\end{lemma} 

Using Lemma \ref{lemma-project} and comparing equation (\ref{almost}) to equation (\ref{S-Ptolemy}), we see that 
\begin{equation*} 
x_{\gamma} = 
\Match_{\widehat{S}, \widehat{T}, \tilde{\gamma}} 
(x_{\pi(\sigma_j)}, \phi(y_{\sigma_j})). 
\end{equation*}
But now by equation (\ref{eq useful}), we get the desired result
$$\phi_{\gamma}(x_{\tilde{\gamma}}) = x_{\gamma}.$$ 
\end{proof}

We now turn to the proof of Lemma \ref{lemma-project}.

\begin{proof}
The monomials $Y_{\pm}$ and $\widetilde{Y}_{\pm}$ are defined by
equations (\ref{S-Ptolemy}) and (\ref{Lift-Ptolemy})
and are computed by analyzing how the laminations associated 
to the arcs of $T$ and $\widehat{T}$ cut across the quadrilaterals
$Q \subset S$ and $\tilde{Q} \subset \widehat{S}$.

By the definition of shear coordinate, the only laminations
which can make a contribution to the $Y$'s (respectively,
$\widetilde{Y}$'s) are those intersecting $\gamma$ and two opposite
sides of $Q$
(respectively, $\tilde{\gamma}$ and two opposite sides of $\Qbar$).
In particular, these laminations must be a subset of the laminations 
$L_{\tau_{i_1}},\dots,L_{\tau_{i_d}}$ (respectively, $L_{\sigma_1},\dots,L_{\sigma_d}$, where $\sigma_1,\dots,\sigma_d$ are arcs of 
the triangulation of $\Sbar_{\gamma} \subset 
\widehat{S}$).

We claim that for every such arc $\tau_{i_k}$ in $T$
which is {\it not} a loop or radius, the lamination 
$L_{\tau_{i_k}}$ has the same local configuration in $Q$
as $L_{\sigma_k}$ does in $\Qbar$.  (Recall that 
$\pi(\sigma_k) = \tau_{i_k}$.)   To see why this is true,
recall that 
$\Sbar_{\gamma}$ is constructed by following $\gamma$ in 
$S$, keeping track of which arcs it is crossing, and glueing
together a sequence of triangles accordingly.  In $S$, 
we can imagine applying an isotopy to $\gamma'$, so that it follows
$\gamma$ as long as possible without introducing unnecessary crossings
with arcs of the triangulation, 
before leaving $\gamma$ to travel
along a different side of the quadrilateral. Recall that each 
elementary laminate
$L_{\tau_{i_k}}$ is obtained by taking the corresponding arc $\tau_{i_k}$
and simply rotating its endpoints a tiny amount counterclockwise.
So a laminate 
$L_{\tau_{i_k}}$ will make a nonzero contribution 
to the shear coordinates if and only if 
it crosses a side of $Q$ (say $\alpha_2$), then
 $\gamma$ and $\gamma'$, then the opposite side $\alpha_4$ of $Q$, 
without crossing $\alpha_1$ or $\alpha_3$ in between.
(The corresponding arc $\tau_{i_k}$ will either have exactly
the same crossings with $Q$, or it may have an endpoint coinciding
with an endpoint of $\alpha_2$.)
In this case the lift $\sigma_k$ of $\tau_{i_k}$ will be an arc
of $\widehat{S}$ which is an internal arc 
common to both $\Sbar_{\gamma}$ and $\Sbar_{\gamma'}$;
it is clear by inspection that it will cut across the two
opposite sides $\tilde{\alpha}_2$ and $\tilde{\alpha}_4$ of 
$\Qbar$, see Figure \ref{Big-Polygon}.
 
Therefore the corresponding contributions to the shear
coordinates will be the same from both the arc 
$\tau_{i_k}$ and 
its lift $\sigma_k$.  
Since 
$\phi(y_{\sigma_j}) = y_{\pi(\sigma_j)}$ if 
$\pi(\sigma_j)$ is not a loop or radius, we can henceforth 
ignore the contributions to the $Y$-monomials which come 
from such arcs $\tau_{i_k}$ and their lifts $\sigma_k$.

What remains is to 
analyze the contribution to the shear coordinates
from a self-folded triangle in $T$, and the contributions
to the shear coordinates from its lift in $\Tbar$.  
We will carefully analyze a  representative example, and 
then explain what happens in the remaining cases.

\begin{figure}
\begin{center}
\input{Lamin.pstex_t}
\end{center}
\caption{}
\label{fig lamin}
\end{figure}

The leftmost figure in Figure \ref{fig lamin} shows the quadrilateral
$Q$ in $S$; $\gamma$ is the arc bisecting it.  We've also displayed a
self-folded triangle in $T$ 
with a loop $\tau_{i_1}$ and radius $\tau_{i_2}$ to a puncture $p$.  
Just to the right of 
this is the same quadrilateral, and the elementary laminations 
$L_{\tau_{i_1}}$ and $L_{\tau_{i_2}}$.  To the right of that is 
the quadrilateral $\Qbar$, bisected by the arc $\tilde{\gamma}$.  
Here, $\sigma_1$, $\sigma_2$, and $\sigma_3$ are the lifts of 
$\tau_{i_1}$ and $\tau_{i_2}$ in $\Qbar$;
$\pi(\sigma_1) = \pi(\sigma_3) = \tau_{i_1}$ and 
$\pi(\sigma_2) = \tau_{i_2}$.
The rightmost figure in Figure \ref{fig lamin} shows $\Qbar$ together with the elementary 
laminations $L_{\sigma_1}$, $L_{\sigma_2}$, and $L_{\sigma_3}$.

Computing shear coordinates, we get 
$b_{\gamma}(T, L_{\tau_{i_1}}) = 
b_{\gamma}(T,L_{\tau_{i_2}}) = -1$, and also 
$b_{\tilde{\gamma}}(\Tbar, L_{\sigma_1}) = 
b_{\tilde{\gamma}}(\Tbar, L_{\sigma_2}) = 
b_{\tilde{\gamma}}(\Tbar, L_{\sigma_3}) = -1$.
Therefore the $Y_-$ monomial in $R$ gets a contribution of 
$y_{\tau_{i_1}} y_{\tau_{i_2}} = 
y_{\tau_{i_2}^{(p)}} y_{\tau_{i_2}}$, and the 
$\widetilde{Y}_-$ monomial in $\Rbar$ gets a contribution of 
$y_{\sigma_1} y_{\sigma_2} y_{\sigma_3}$.  
Applying $\phi$ to this gives
$\phi(y_{\sigma_1} y_{\sigma_2} y_{\sigma_3}) = 
y^2_{\tau_{i_2}^{(p)}} \frac{y_{\tau_{i_2}}}{y_{\tau_{i_2}^{(p)}}} = 
y_{\tau_{i_2}} y_{\tau_{i_2}^{(p)}}$, as desired. 

\begin{figure}
\begin{center}
\input{Lamin2.pstex_t}
\end{center}
\caption{}
\label{fig lamin2}
\end{figure}
Figure \ref{fig lamin2} shows two more ways that a self-folded triangle
from $T$ might interact with the quadrilateral $Q$.  Each row of the figure
displays the self-folded triangle and  the corresponding elementary 
laminations, and the lift of the self-folded triangle in $\Tbar$ and the 
corresponding elementary laminations.   
In the example of the top row, we have 
$b_{\gamma}(T, L_{\tau_{i_1}}) = 0$,
$b_{\gamma}(T,L_{\tau_{i_2}}) = -1$,  
$b_{\tilde{\gamma}}(\Tbar, L_{\sigma_1}) = 0$,
$b_{\tilde{\gamma}}(\Tbar, L_{\sigma_2}) = -1$,
and $b_{\tilde{\gamma}}(\Tbar, L_{\sigma_3}) = -1$.
In the example of the second row, we have 
$b_{\gamma}(T, L_{\tau_{i_1}}) = -1$,
$b_{\gamma}(T,L_{\tau_{i_2}}) = 0$,  
$b_{\tilde{\gamma}}(\Tbar, L_{\sigma_1}) = -1$,
$b_{\tilde{\gamma}}(\Tbar, L_{\sigma_2}) = 0$,
and $b_{\tilde{\gamma}}(\Tbar, L_{\sigma_3}) = 0$.

All other configurations of a self-folded triangle from $T$ either
make no contribution to the shear coordinates indexed by $\gamma$
(in which case the same is true for the lift of that self-folded triangle),
or come from either
rotating or reflecting one of the configurations from Figure \ref{fig lamin2}.
We leave it as an exercise for the reader to check that just as in the example
of Figure \ref{fig lamin}, the monomials corresponding to the shear coordinate
$b_{\tilde{\gamma}}(\Tbar, L_{\sigma_1} \cup L_{\sigma_2} \cup L_{\sigma_3})$
map via $\phi$ to the monomials corresponding to the shear coordinate
$b_{\gamma}(T, L_{\tau_{i_1}} \cup L_{\tau_{i_2}})$.

It may seem that  our arguments and figures rely on the assumption that 
the quadrilateral $Q$ has four distinct edges and four distinct vertices.
However, one can always slightly deform a quadrilateral with some identified
vertices or edges to get a quadrilateral with four distinct edges and vertices;
see Figure \ref{opening}.
It is not hard to see that the shear coordinates of a lamination 
with respect to a given arc are unchanged if we work instead with this deformation,
so our arguments extend to this situation.

\begin{figure}
\includegraphics[height=1.6in]{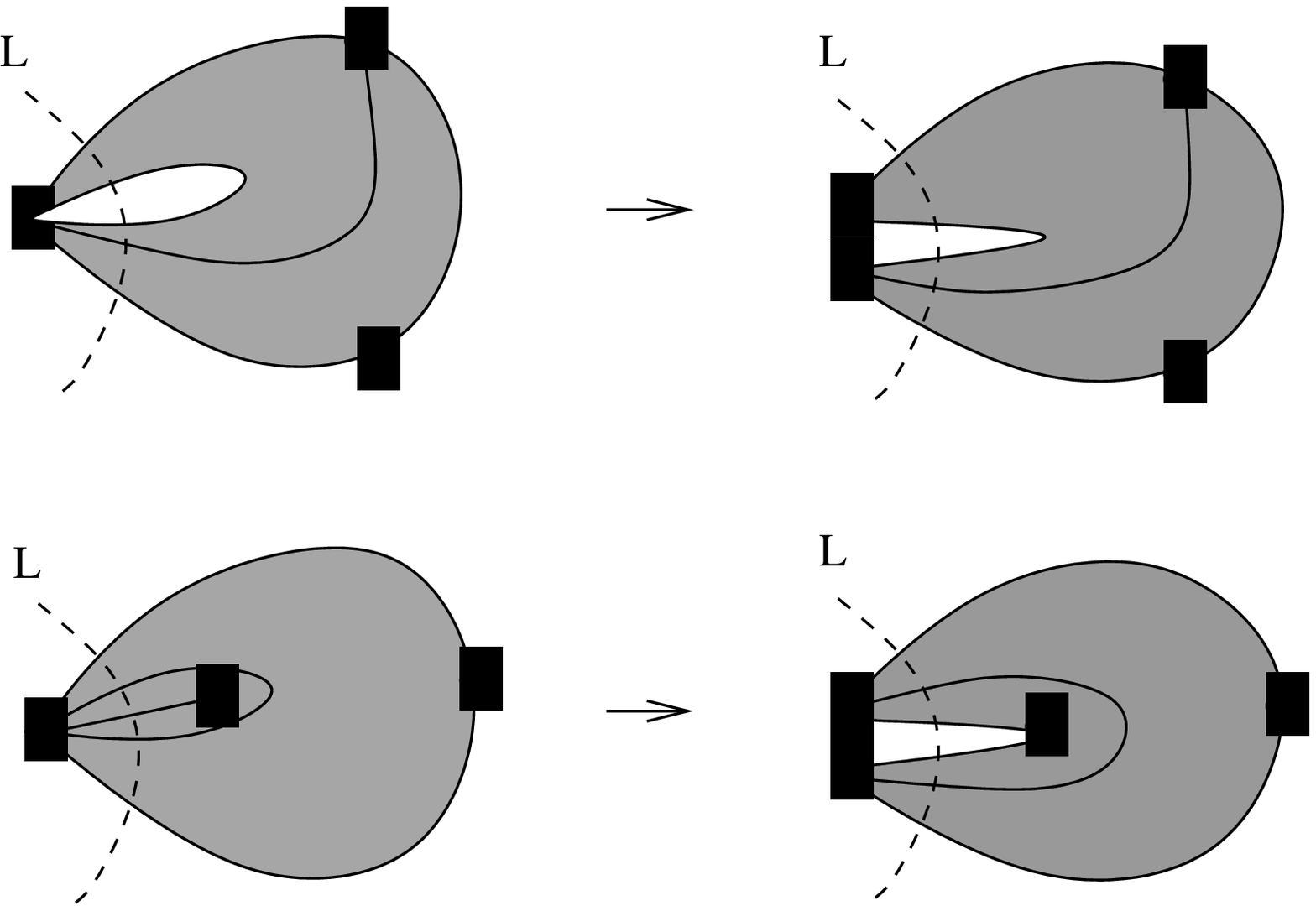}
\caption{}
\label{opening}
\end{figure}

This completes the proof of the claim, and hence the theorem.
\end{proof}

We are now ready to prove Theorem \ref{thm main}.

\begin{proof}
We have fixed a bordered surface $(S,M)$, an ordinary arc
$\gamma$, and an ideal 
triangulation $T$ with internal arcs $\tau_1,\dots,\tau_n$
and boundary segments $\tau_{n+1},\dots,\tau_{n+c}$.
This determines a cluster algebra $\A$ with principal coefficients
with respect to $\Sigma_T$.
From $(S,M)$, $T$, and $\gamma$ we have built a polygon
$\Sbar_{\gamma}$ with a ``lift" $\tilde{\gamma}$ of $\gamma$,
together with a triangulation 
$\Tbar_{\gamma}$ with internal arcs 
$\sigma_1, \dots, \sigma_d$ and boundary segments
$\sigma_{d+1}, \dots, \sigma_{2d+3}.$
We have a surjective map 
$\pi:\{\sigma_1,\dots,\sigma_{2d+3}\} \to \{\tau_1,\dots,\tau_{n+c}\}$.
Furthermore, 
we have associated a type $A_d$ cluster algebra
$\Abar_{\gamma}$ to $\Sbar_{\gamma}$, and a homomorphism
$\phi_{\gamma}$ from $\Abar_{\gamma}$ to $\Frac(\A)$.
This map has the property that
for each $\sigma_i \in \{\sigma_1,\dots, \sigma_{2d+3}\}$, $\phi_{\gamma}(x_{\sigma_i}) = x_{\pi(\sigma_i)}$.
Additionally, by Theorem \ref{phi-map}, 
$\phi_{\gamma}(x_{\tilde{\gamma}}) = x_{\gamma}$.

Since $\Abar_{\gamma}$ is a type A cluster algebra, we can compute the Laurent expansion
of $x_{\tilde{\gamma}}$ with respect to $\Sigma_{\Tbar_{\gamma}}$.
More specifically, \cite{MS} proved  Theorem \ref{thm main} for unpunctured
surfaces, which in particular includes polygons.  At this point the reader may worry 
that Theorem \ref{thm main} cannot be applied to $\Abar_{\gamma}$, as
$\Abar_{\gamma}$ is not simply a cluster algebra with principal coefficients associated to 
a triangulation -- it has 
extra coefficient variables corresponding to 
the boundary segments of $\Tbar_{\gamma}$.  However, consider the triangulated polygon
$(\Sbar'_{\gamma}, \Tbar'_{\gamma})$ that we obtain from 
$(\Sbar_{\gamma}, \Tbar_{\gamma})$ by adding $c$ triangles around
the boundary, each one with an edge at a boundary segment, and 
consider the corresponding cluster algebra $\Abar'_{\gamma}$ with principal coefficients.  
This is still a type A cluster algebra so we can use the result of \cite{MS} to 
apply Theorem \ref{thm main} to expand the cluster variable corresponding to 
$\tilde{\gamma}$ with respect to  $\Sigma_{\Tbar'_{\gamma}}$ in 
$\Abar'_{\gamma}$.  And clearly the formula giving the Laurent expansion of 
$x_{\tilde{\gamma}}$ with respect to 
$\Sigma_{\Tbar_{\gamma}}$ in $\Abar_{\gamma}$ is identical to the formula giving the 
Laurent expansion of the cluster variable corresponding to $\tilde{\gamma}$
with respect to 
$\Sigma_{\Tbar'_{\gamma}}$ in $\Abar'_{\gamma}$. 

Therefore we can apply Theorem \ref{thm main} to get the cluster expansion of 
$x_{\tilde{\gamma}}$ with respect to $\Sigma_{\Tbar_{\gamma}}$ in $\Abar_{\gamma}$: in other words,
we build a graph $G_{\Tbar_{\gamma}, \tilde{\gamma}}$, and the cluster expansion
is given as a generating function for perfect matchings of this graph.
The variables in the expansion are $x_{\sigma_1}, \dots, x_{\sigma_{2d+3}}$
and $y_{\sigma_1},\dots, y_{\sigma_d}$. Therefore, since 
$\phi_{\gamma}$ is a homomorphism  
such that 
$\phi_{\gamma}(x_{\sigma_i}) = x_{\pi(\sigma_i)}$ for $1 \leq i \leq 2d+3$, and 
$\phi_{\gamma}(x_{\sigma}) = x_{\pi(\sigma)}$, 
computing the Laurent expansion for $x_{\gamma}$ with respect to $T$
in $\A$ amounts to applying a specialization of variables 
to the generating function for matchings in $G_{\Tbar_{\gamma}}$.

It follows from the construction in Section \ref{sect Sbar} that 
the unlabeled graph $G_{\Tbar_{\gamma}, \tilde{\gamma}}$ is equal to the 
unlabeled graph $G_{T,\gamma}$: this is because the 
triangulated polygon $(\Sbar_{\gamma}, \Tbar_{\gamma})$ is built so that the local 
configuration of triangles that $\tilde{\gamma}$ passes through is the same
as the local configuration of triangles that $\gamma$ passes through in $T$.
Additionally, an edge of $G_{\Tbar_{\gamma},\tilde{\gamma}}$ labeled 
$\sigma_i$ corresponds to an edge of $G_{T, \gamma}$ labeled $\pi(\sigma_i)$.

Comparing the definition of $\phi_{\gamma}$ on the coefficients
$y_{\sigma_j}$ (Equation \ref{eq def 3}) to the definition of 
the specialized height monomial (Definition \ref{height}), we see now 
that applying $\phi_{\gamma}$ to the generating function for matchings
in $G_{\Tbar_{\gamma}, \tilde{\gamma}}$ yields exactly the formula of Theorem \ref{thm main}
applied to $S$, $T$, and $\gamma$.  This completes the proof of the theorem.
\end{proof}

\section{Positivity for notched arcs in the coefficient-free case}\label{sec quick}

In this section we will use Proposition \ref{coeff-free} together
with our positivity result for ordinary arcs to prove the positivity result 
for notched arcs (in the coefficient-free case).
We begin by proving Proposition \ref{coeff-free}.
\begin{proof}
Fix a bordered surface $(S,M)$ and an ideal triangulation $T$ of $S$.
Let $\A$ be the associated coefficient-free cluster algebra.
Consider a puncture $p$, a different marked point $q$, and an ordinary 
arc $\gamma$ between $p$ and $q$.  Consider a third marked point $s$
and an ordinary arc $\rho$ between $p$ and $s$.  Let $\alpha$ and $\beta$
be the two ordinary arcs between $q$ and $s$ which are sides of a bigon
so that the triangles with sides $\alpha, \gamma, \rho$ and 
$\beta, \gamma, \rho$ have simply-connected interior.  See the left-hand-side of 
Figure \ref{fig quadrilaterals}.

\begin{figure}
\input{TwoQuads.pstex_t} 
\caption{}
\label{fig quadrilaterals}
\end{figure}

Then in $\A$, $x_{\gamma} x_{\rho^{(p)}} = x_{\alpha} + x_{\beta} = x_{\gamma^{(p)}} x_{\rho}$, 
which implies that $\frac{x_{\gamma^{(p)}}}{x_{\gamma}} = \frac{x_{\rho^{(p)}}}{x_{\rho}}.$
In other words, the ratio 
$\frac{x_{\gamma^{(p)}}}{x_{\gamma}}$ is an invariant which 
we will call $z_p$, which 
depends only on $p$, and not the choice of ordinary arc $\gamma$ incident to $p$.

If we take the same bigon with sides $\alpha$ and $\beta$ and notch all three arcs emanating
from $q$, we get
$x_{\gamma^{(pq)}} x_{\rho} = x_{\alpha^{(q)}} + x_{\beta^{(q)}} = 
   z_q x_{\alpha} + z_q x_{\beta} = z_q (x_{\gamma^{(p)}} x_{\rho}).$
Therefore $x_{\gamma^{(pq)}} = z_q x_{\gamma^{(p)}} = z_p z_q x_{\gamma}$.

So far we have treated the case where $\gamma$ has two distinct endpoints.
Now suppose that $\gamma$ is an ordinary loop based at $p$
which does not cut out a once-punctured monogon.
Then choose a marked point $q$ in the interior of $\gamma$, another marked 
point $s$, and an ordinary arc from $q$ to $s$ which crosses $\gamma$ once.
Let $\alpha_1$ and $\alpha_2$ be ordinary arcs between $p$ and $s$,
and $\alpha_3$ and $\alpha_4$ be ordinary arcs between $p$ and $q$
so that the $\alpha_i$ form the sides of a quadrilateral with simply-connected
interior.  See the right-hand-side of Figure \ref{fig quadrilaterals}.  Then 
$x_{\gamma} x_{\rho} = x_{\alpha_1} x_{\alpha_3} + x_{\alpha_2} x_{\alpha_4}$ and  
$x_{\gamma^{(pp)}} x_{\rho} = x_{\alpha_1^{(p)}} x_{\alpha_3^{(p)}} + 
x_{\alpha_2^{(p)}} x_{\alpha_4^{(p)}}$, where $x_{\alpha_i^{(p)}} = z_p x_{\alpha_i}$.
It follows that $x_{\gamma^{(pp)}} = z_p^2 x_{\gamma}$.

What remains is to give an explicit expression for the quantity $z_p$. 
For $\gamma$ an ordinary arc with distinct endpoints, we know that  
$z_p = \frac{x_{\gamma^{(p)}}}{x_{\gamma}}$ does not depend on the choice of $\gamma$, 
so we make the simplest possible choice.
Choose $\tau_1$ to be any arc of $T$ which is incident to $p$, so that 
$x_{\tau_1}$ is in the initial cluster associated to $T$.  Let 
$q$ denote the other endpoint of $\tau_1$, and let $\ell$ be the loop based at $q$
cutting out 
a once-punctured monogon around $p$.  Then $x_{\ell} = x_{\tau_1} x_{\tau_1^{(p)}}$, so 
$z_p = \frac{x_{\tau_1^{(p)}}}{x_{\tau_1}} = \frac{x_{\ell}}{x_{\tau_1}^2}$.  
The variable $x_{\tau_1}$ is an initial cluster variable 
and we can compute the Laurent expansion of $x_{\ell}$ using 
Theorem \ref{thm main}. 

It is easy to see that the graph $\overline{G}_{\ell,T}$ consists of 
$h-1$ tiles with diagonals $\tau_2,\dots,\tau_h$, where 
$\tau_1, \tau_2,\dots, \tau_h$ are the arcs of $T$ emanating from $p$
(say in clockwise order around $p$).  The tiles are glued in an alternating fashion
so as to form a ``zig-zag" shape, see Figure \ref{zig-zag}.  Also, 
$\tau_1$ is the label of the two outer edges of $\overline{G}_{\ell,T}$.  
Now a straightforward induction on $h$ shows that applying Theorem \ref{thm main} to 
$\ell$ gives 
$$x_{\ell} = \frac{x_{\tau_1} \sum_{i=0}^{h-1} \sigma^i ( x_{[\tau_1,\tau_2]}  x_{\tau_3} x_{\tau_4} 
\cdots x_{\tau_h})}{x_{\tau_2}\cdots  x_{\tau_h}},
$$ where $\sigma$ is the cyclic permutation $(1,2,3,\dots, h)$ acting on subscripts.
Dividing this expression by $x_{\tau_1}^2$ gives the desired expression for $z_p$.
\begin{figure}
\input{zigzag.pstex_t} 
\caption{}
\label{zig-zag}
\end{figure}
This completes the proof.
\end{proof}

\begin{Cor}
Fix a bordered surface $(S,M)$, a tagged triangulation $T$ of the form
$\iota(T^\circ)$ where $T^\circ$ is an ideal triangulation,  and let $\A$ be the 
corresponding coefficient-free cluster algebra.  Then
the Laurent expansion of a cluster variable corresponding to a notched arc
with respect to $\Sigma_T$ is positive.
\end{Cor}

\begin{proof}
This follows immediately from our positivity result for cluster variables
corresponding to ordinary arcs, together with Proposition \ref{coeff-free}.
\end{proof}


\section{Proofs of Laurent expansions for notched arcs} \label{sec Finish}

In this section, we prove the results of Section \ref{sect notched}, giving cluster expansion formulas for cluster variables corresponding to tagged arcs. We use algebraic identities for cluster variables to reduce the proofs of Theorem \ref{thm single} and Theorem \ref{thm double} to combinatorial statements about perfect matchings, $\zg$-symmetric matchings, and $\zg$-compatible pairs of matchings.

In particular, for the case of a tagged arc $\zg^{(p)}$ with a single notch at puncture $p$ (Theorem \ref{thm single}), we use the equation $x_{\ell_p} = x_{\zg} x_{\zg^{(p)}}$ and the fact that Theorem \ref{thm main} gives us matching formulas for two out of three of these terms.  For the case of a tagged arc $\zg^{(pq)}$ with a notch at both ends, punctures $p$ and $q$ (Theorem \ref{thm double}), we use an identity (described in Section \ref{subsect algdouble}) involving $x_{\zg^{(pq)}}$ and three other cluster variables, where all other terms except $x_{\zg^{(pq)}}$ have matching formulas from Theorems \ref{thm main} and \ref{thm single}.  In both of these cases, the fact that the desired matching formulas for $x_{\zg^{(p)}}$ and $x_{\zg^{(pq)}}$ satisfy combinatorial identities analogous to  the algebraic identities coming from the cluster algebra completes the proofs of Theorems \ref{thm single} and \ref{thm double}.  Before giving these proofs, we introduce some notation and auxiliary lemmas.   We begin by describing the shape of the graph $G_{T^\circ,\ell_p}$ in more detail.   

\begin{definition} 
Let $H_{\zeta}$ be the connected subgraph of $G_{T^\circ, \ell_p}$ consisting of the union of the tiles  $G_{\zeta_1}$ through $G_{\zeta_{e_p}}$  
(see the notation of Section \ref{sect notched} and Figure \ref{LoopSnakeGraph}).
\end{definition}

\begin{remark}\label{Hzeta}
It follows from the construction of $G_{T^\circ,\ell_p}$ in Section \ref{sect graph} and the fact that $\zeta_1$ through $\zeta_{e_p}$ all share a single endpoint, that $H_{\zeta}$ contains no consecutive triple of tiles all of which lie in the same row or column.  
\end{remark}

\begin{remark} \label{BendingAssumpt} Since the arcs $\tau_{i_d}, \zeta_1, \zeta_{e_p}$ are the sides of a triangle in $T^\circ$, and $\tau_{i_{d-1}}$
and  $\tau_{i_d}$ share a vertex, it follows that in the graph $G_{T^\circ,\ell_p}$ either the three tiles $G_{\tau_{i_{d-1}}}, G_{\tau_{i_d}}$, and $G_{\zeta_1}$ or the three tiles $G_{\tau_{i_{d-1}}}, G_{\tau_{i_d}}$ and $G_{\zeta_{e_p}}$ lie in a single row or column.
Thus, we may assume without loss of generality that tiles $G_{\tau_{i_{d-1}}}$, $G_{\tau_{i_d}}$, and $G_{\zeta_1}$ lie in a single row and tiles $G_{\tau_{i_{d-1}}}$, $G_{\tau_{i_d}}$, and $G_{\zeta_{e_p}}$ do not.  See Figure \ref{LoopSnakeGraphWV}. 
\end{remark}

\begin{lemma} \label{twoends} If $P$ is a perfect matching of $G_{T^\circ,\ell_p}$ then $P$ restricts to a perfect matching on at least one of its two ends.  More precisely,  
$P|_{G_{T^\circ,\gamma_p,1}}$ is a perfect matching
of  
${G_{T^\circ,\gamma_p,1}}$,
or the analogous condition must hold for $P|_{G_{T^\circ,\gamma_p,2}}$.
\end{lemma}

\begin{proof}
See Figure \ref{LoopSnakeGraphWV}.  
We let $w_1$ (respectively $w_2$) denote the other vertex
of the edge labeled $\zeta_{e_p}$ (respectively $\zeta_1$) 
incident to $v_1$ (respectively $v_2$).
Suppose that $P$ is a perfect matching of $G_{T^\circ,\ell_p}$ whose 
restriction to each of the subgraphs $G_{T^\circ,\gamma_p,i}$ is not a perfect matching.  The restriction of $P$ to $G_{T^\circ,\gamma_p,1}$ is not a perfect matching if and only if $P$ contains the edge labeled $\zeta_2$ incident to vertex $v_1$.  Then $P$ must also contain the edge labeled $\tau_{i_{d}}$ on the same tile because otherwise the vertex $w_1$ could only be covered by the edge labeled $\tau_{i_{d-1}}$ and this would leave a connected component with an odd number of vertices to match together.

\begin{figure}
\input{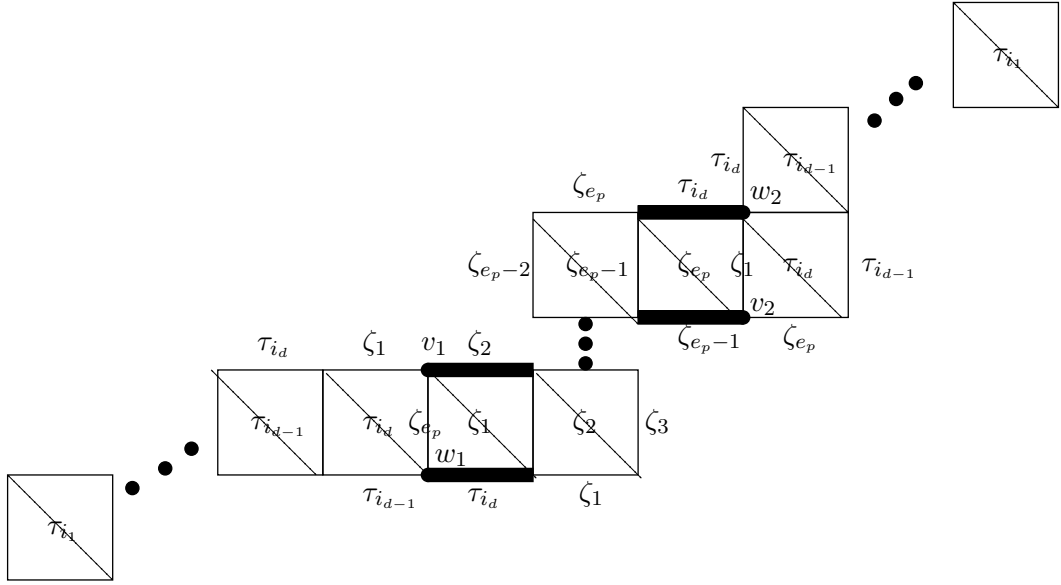}
\caption{
The graph $G_{T^\circ,\zLp}$. 
There exists no perfect matching of $G_{T^\circ,\ell_p}$ containing the highlighted edges.  Here $e_p$ is even.}
\label{LoopSnakeGraphWV}
\end{figure}

Similarly, the restriction of $P$ to $G_{T^\circ,\gamma_p,2}$ is not a perfect matching if and only if $P$ contains the edge labeled $\zeta_{e_p-1}$ incident to vertex $v_2$.  Then $P$ must also contain the edge labeled $\tau_{i_d}$ incident to $w_2$ on this same tile.  However, no perfect matching $P$ can contain all four of these edges since by Remark \ref{Hzeta}, $H_\zeta$ contains no consecutive triple of tiles lying in a single row or column.
Thus we have a contradiction.
\end{proof}

\subsection{Proof of the expansion formula for arcs with a single notch} \label{SingleNotchedProofs}

For the proof of Theorem \ref{thm single}, we also need the following fact.

\begin{lemma}  \label{yheight}
The minimal matching $P_-$ 
of $G_{T^\circ,\ell_p}$ 
is a $\gamma$-symmetric matching.
\end{lemma}

\begin{figure}
\input{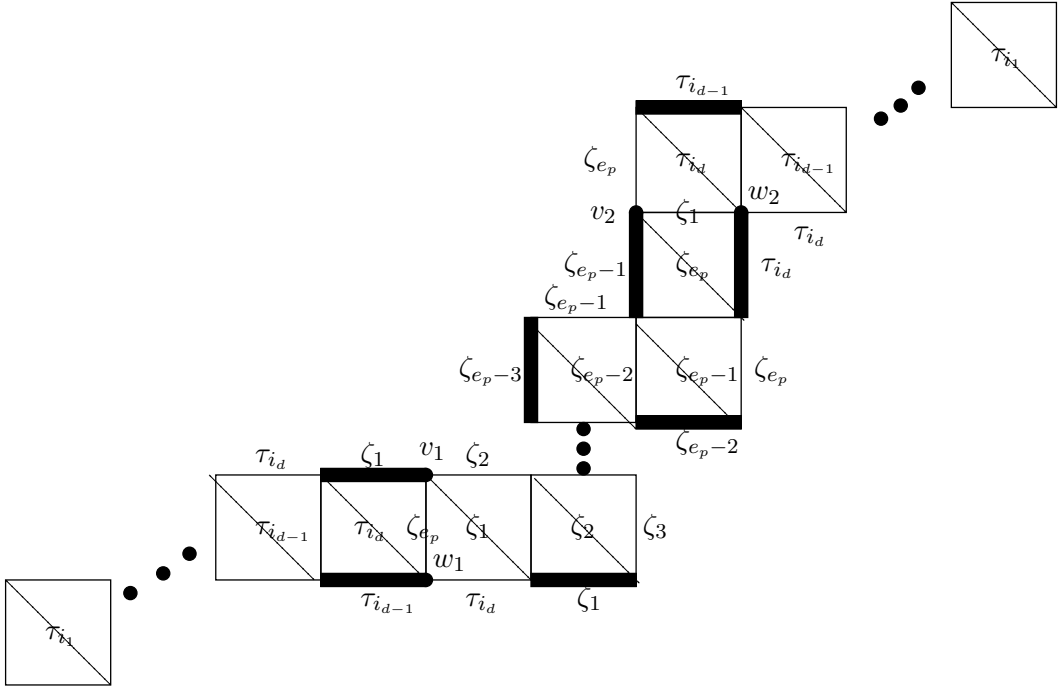}
\caption{One of the matchings $P_-$ and $P_+$ of $G_{T^\circ,\ell_p}$ must contain the highlighted edges and is therefore $\gamma$-symmetric.  Here  $e_p$ 
is odd.}
\label{figMinM}
\end{figure}

\begin{proof}
Since $P_-$ and $P_+$ are the unique perfect matchings of $G_{T^\circ,\ell_p}$ using only boundary edges, it follows that exactly one out of $\{P_-, \, P_+\}$, say $P_\epsilon$, contains the edge labeled $\tau_{i_{d-1}}$ on the tile $G_{\tau_{i_d}}$ containing $v_1$.  The perfect matching $P_\epsilon$ cannot contain the edge labeled $\tau_{i_d}$ on the adjacent tiles. As shown in Figure \ref{figMinM}, the perfect matching $P_\epsilon$ contains other edges on the boundary in an alternating fashion.
Since the two ends of $G_{T^\circ,\ell_p}$ are isomorphic, continuing along the boundary in an alternating pattern, we obtain that $P_\epsilon$ is $\gamma$-symmetric.  Its complement must 
 also be $\gamma$-symmetric,  so  both $P_-$ and $P_+$ are $\gamma$-symmetric.
\end{proof}

We need to introduce a few more definitions before proving Theorem \ref{thm single}.

\begin{definition}
Let $H_{\zeta}^{(i)}$ 
denote the induced subgraph obtained 
after deleting the vertices $v_i$, $w_i$ 
of $H_{\zeta}$ and all edges incident to those vertices.
Let $G_{\zeta}^{(1)}$ (resp. $G_{\zeta}^{(2)}$) denote the subgraph of $G_{T^\circ,\ell_p}$ which is the union $H_{\zeta}^{(1)} \cup G_{T^\circ,\zg_p,2}$ (resp. $G_{T^\circ,\zg_p,1} \cup H_{\zeta}^{(2)}$).
That is, we use a superscript $(i)$ to denote the removal of the $i$th side of a graph.
\end{definition}

\begin{definition} [\emph{Symmetric completion}] Fix a perfect matching $P$ of $G_{T^\circ,\ell_p}$, and by Lemma \ref{twoends}, assume without loss of generality that $P$ restricts to a perfect matching on $G_{T^\circ,\gamma_p,1}$.
Therefore $P$ also restricts to a perfect matching on the complement, graph
$G_{\zeta}^{(1)}$.
We define the \emph{symmetric completion} $\overline{P}=\overline{P|_{G_{\zeta}^{(1)}}}$ of $P|_{G_{\zeta}^{(1)}}$ to be the unique extension of $P|_{G_{\zeta}^{(1)}}$ to $G_{T^\circ,\ell_p}$ such that $\overline{P}|_{H_{T^\circ,\gamma_p,1}} \cong \overline{P}|_{H_{T^\circ,\gamma_p,2}}$.  
(Note that after adding edges to $H_{T^\circ,\gamma_p,1}$, only vertex $v_1$ is not covered.  We 
add an edge incident to $v_1$ based on whether the edge incident to $w_1$ labeled $\tau_{i_{d-1}}$
is included so far.)  
It follows from this construction that the restriction $\overline{P}|_{G_{T^\circ,\zg_p,1}}$ is a perfect matching. \end{definition}

\begin{definition} [\emph{Sets $\mathcal{P}(\zg)$ and $\mathcal{SP}(\zg^{(p)})$}] \label{ppANDsp} For an ordinary arc $\zg$ (including loops 
cutting out once-punctured monogons)
we let $\mathcal{P}(\gamma)$ denote the set of perfect matchings of $G_{T^\circ,\gamma}$, and let $\mathcal{SP}(\gamma^{(p)})$ denote the set of $\gamma$-symmetric matchings of $G_{T^\circ,\ell_p}$.
\end{definition}

We now prove Theorem \ref{thm single} by constructing a bijection $\psi$ between pairs $(P_1,P_2)$ in $\mathcal{P}(\zg) \times \mathcal{SP}(\zg^{(p)})$ and perfect matchings $P_3$ in $\mathcal{P}(\zLp)$.  
This bijection will be weight-preserving and height-preserving, in the sense that
if $\psi(P_1,P_2) = P_3$, then $x(P_1)\overline{x}(P_2) = x(P_3)$ and $h(P_1)\overline{h}(P_2) = h(P_3)$. 
This gives
\begin{equation}\label{preserve}
\sum_{P_3 \in \mathcal{P}(\zLp)} x(P_3) h(P_3)
= 
\left(\sum_{P_1 \in \mathcal{P}(\zg)} x(P_1) h(P_1)\right) 
\, \left(\sum_{P_2 \in \mathcal{SP}(\zg^{(p)})} \overline{x}(P_2) \overline{h}(P_2)\right).
\end{equation}
After applying $\Phi$, the left-hand-side and first term on the right are the numerators for $x_{\ell_p}$ and $x_\zg$ given by Theorem \ref{thm main}, which allows us to express 
$x_{\zg^{(p)}} = \frac{x_{\ell_p}}{x_\zg}$ in terms of 
$\sum_{P_2 \in \mathcal{SP}(\zg^{(p)})} \overline{x}(P_2) \overline{h}(P_2)$.

\begin{proof} [Proof of Theorem \ref{thm single}]
As indicated above, we define a 
map $$\psi : \mathcal{P}(\gamma) \times \mathcal{SP}(\zg^{(p)}) \rightarrow \mathcal{P}(\zLp) \mathrm{~~by}$$
$$\psi(P_1,P_2) = \begin{cases}
P_1 ~\bigcup~ P_2|_{G_{\zeta}^{(1)}} \mathrm{~~~~if~}P_2|_{G_{T^\circ,\zg_p,1}} \mathrm{~is~a~perfect~matching} \\
P_2|_{G_{\zeta}^{(2)}}
~\bigcup~ P_1 \mathrm{~otherwise~}
\end{cases}$$
where the edges of $P_1$ are placed on the subgraph $G_{T^\circ,\zg_p,1}$ or $G_{T^\circ,\zg_p,2}$, respectively.
In words, $\psi$ removes all of the edges from one of the two ends of the $\zg$-symmetric matching $P_2$, 
and replaces those edges with edges from the perfect matching $P_1$, thereby constructing a perfect matching $P_3$ of $\mathcal{P}(\zLp)$ that it is not necessarily $\zg$-symmetric.
By Lemma \ref{twoends}, either $P_2|_{G_{T^\circ,\zg_p,1}}$ or
$P_2|_{G_{T^\circ,\zg_p,2}}$ is a perfect matching and so $\psi$ is well-defined.
Thus $\psi(P_1,P_2)$ is a perfect matching of $\mathcal{P}(\zLp)$.

We show that $\psi$ is a bijection by exhibiting its inverse.  For $P_3 \in \mathcal{P}(\zLp)$, define
$$\varphi(P_3) = \begin{cases}  ( P_3|_{G_{T^\circ,\zg_p,1}}~~, ~~\overline{P_3|_{G_{\zeta}^{(1)}}}) \mathrm{~~~if~}P_3|_{G_{T^\circ,\zg_p,1}} \mathrm{~is~a~perfect~matching}, \\
( P_3|_{G_{T^\circ,\zg_p,2}}~~, ~~\overline{P_3|_{G_{\zeta}^{(2)}}}) \mathrm{~otherwise.}
\end{cases}$$
A little thought shows that these two maps are inverses.

We now show that the bijection $\psi$ is weight-preserving.  
Without loss of generality, $P_2|_{G_{T^\circ,\zg_p,1}}$ is a perfect matching. 
If $\psi(P_1,P_2) = P_3$, then 
$P_3 = P_1 ~\cup~ P_2|_{G_{\zeta}^{(1)}}$.  
We  obtain $$x(P_3) = x(P_1) \, x(P_2|_{G_{\zeta}^{(1)}})
= x(P_1)
\, \frac{x(P_2)}{x(P_2|_{G_{T^\circ,\zg_p,1}})}.$$ 
Since $\overline{x}(P_2)$ is defined to be
$\frac{x(P_2)}{x(P_2|_{G_{T^\circ,\zg_p,1}})}$, we conclude that $\psi$ is weight-preserving.

To see that $\psi$ is height-preserving, we use Lemma \ref{yheight}, which states that $P_-(G_{T^\circ,\zLp})$ is a $\zg$-symmetric matching.  Consequently, using the same partitioning that showed that $\psi$ was weight-preserving, we can consider the following equation describing the symmetric difference of $P_3$ and $P_-(G_{T^\circ,\zLp})$:
$$P_3 \ominus P_-(G_{T^\circ,\zLp}) =
(P_1 \ominus P_-(G_{T^\circ,\zg})) \cup
(P_2 \ominus P_-(G_{T^\circ,\zLp})|_{G_{\zeta}^{(1)}}).$$
Since the cycles appearing in the symmetric difference determine the height monomials, this decomposition implies that
$$h(P_3) = h(P_1) \, h(P_2|_{G_{\zeta}^{(1)}})
= h(P_1)
\, \frac{h(P_2)}{h(P_2|_{G_{T^\circ,\zg_p,1}})} = h(P_1) \, \overline{h}(P_2),$$hence $\psi$ is height-preserving.

Because $\phi$ is weight- and height-preserving, we have
(\ref{preserve}).  Applying $\Phi$ gives
\begin{eqnarray}
\label{almost-single-proof2} 
\sum_{P \in \mathcal{P}(\zLp)}
x(P) y(P) 
&=& \left(
\sum_{P_1 \in \mathcal{P}(\gamma)}
x(P_1) y(P_1)
\right)\left(
\sum_{P_2 \in \mathcal{SP}(\gamma^{(p)}) }
\overline{x}(P_2)
\overline{y}(P_2)
\right).
\end{eqnarray} 

We now use the identity $x_{\zLp} = x_{\gamma} x_{\gamma^{(p)}}$ and obtain
\begin{eqnarray} \label{almost-single-proof} 
x_{\gamma^{(p)}} &=&
\frac{
\mathrm{cross}(T^\circ,\gamma) \sum_{P \in \mathcal{P}(\zLp)}
x(P) y(P)}
{\mathrm{cross}(T^\circ, \zLp)
\sum_{P  \in \mathcal{P}(\zg)} x(P) y(P)}.
\end{eqnarray}

Comparing (\ref{almost-single-proof}) and (\ref{almost-single-proof2})
yields the desired formula.
\end{proof}

\subsection{An algebraic identity for arcs with two notches}\label{subsect algdouble}

We now give the algebraic portion of the proof of Theorem \ref{thm double}.  For the purpose of computing the Laurent expansion of $x_{\rho^{(pq)}}$ with respect to $T$, we can assume that no tagged arc in $T$ is notched at either $p$ or $q$, see Remark \ref{assumption}.  In the statement below, the notation $\chi$ indicates $1$ or $0$, 
based on whether it's argument is true or false.

\begin{theorem}\label{double-identity}
Fix a tagged triangulation $T$ of $(S,M)$ which comes from
an ideal triangulation, and
let $\Acal$ be the cluster algebra associated
to $(S,M)$ with principal coefficients with respect to $T$.
Let $p$ and $q$ be punctures in $S$, and let
$\rho$ be an ordinary arc between $p$ and $q$.  Assume that
no tagged arc in $T$ is notched at either $p$ or $q$. Then
\begin{equation*}
x_{\rho} x_{\rho^{(pq)}} - x_{\rho^{(p)}} x_{\rho^{(q)}}  y_{\rho}^{\chi(\rho \in T)} =
(1-\prod_{\tau\in T} y_{\tau}^{e_p(\tau)})
(1-\prod_{\tau\in T} y_{\tau}^{e_q(\tau)})
\prod_{\tau \in T} y_{\tau}^{e(\tau,\rho)}.
\end{equation*}
\end{theorem}

\begin{proof}
For simplicity, we assume for now that $\rho \notin T$.
(Later we will lift the assumption.)
Choose a quadrilateral in $S$ with simply connected interior
such that one of its diagonals is $\rho$.  (This involves
the choice of two more marked points, say $v$ and $w$.) 
Label the (ordinary) arcs of the quadrilateral by
$\alpha, \beta, \gamma, \delta$ and the other diagonal by
$\rho'$, as in Figure \ref{quad}.  Note that there are four ways of changing the taggings around $p$ and $q$, and for each
we get a Ptolemy relation.
\begin{figure}
\input{quad.pstex_t}
\caption{}
\label{quad}
\end{figure}
\begin{align}\label{P0}
x_{\rho} x_{\rho'} &= Y^+ Y_q^+ Y_p^+ x_{\beta} x_{\delta} +
                       Y^- x_{\alpha} x_{\gamma}  \\ \label{P1}
x_{\rho^{(p)}} x_{\rho'} &= Y^+ Y_q^+  x_{\beta} x_{\delta^{(p)}} +
                       Y^- Y_p^- x_{\alpha^{(p)}} x_{\gamma} \\ \label{P2}
x_{\rho^{(q)}} x_{\rho'} &= Y^+ Y_p^+  x_{\beta^{(q)}} x_{\delta} +
                       Y^- Y_q^- x_{\alpha} x_{\gamma^{(q)}} \\ \label{P3}
x_{\rho^{(pq)}} x_{\rho'} &= Y^+  x_{\beta^{(q)}} x_{\delta^{(p)}} +
                       Y^- Y_q^- Y_p^- x_{\alpha^{(p)}} x_{\gamma^{(q)}}
\end{align}
Here, $Y^+$ (respectively $Y^-$) is the monomial (in coefficient
variables) coming from all
laminations which do not spiral into $p$ or $q$ and which
give a shear coordinate of $1$ (respectively $-1$) with $\rho$,
as in Figure \ref{partition}.  We use Definition 12.1 of \cite{FT} to compute shear coordinates with respect to tagged arcs $\rho^{(p)}$, $\rho^{(q)}$, and $\rho^{(pq)}$.  

$Y_p^{\pm}$ and $Y_q^{\pm}$ are monomials coming from
laminations which spiral into either the puncture $p$ or $q$, respectively.
Since we have assumed that $T$ does not contain arcs with a notch
at $p$ or $q$,  all laminates which spiral into
$p$ or $q$ spiral counterclockwise.
$Y_p^+$ is the monomial coming from laminations that spiral into
$p$ giving a shear coordinate of $1$ to $\rho$ (equivalently,
a shear coordinate of $1$ to $\rho^{(q)}$).
$Y_q^+$ is the monomial coming from laminations that spiral into
$q$ giving a shear coordinate of $1$ to $\rho$ (equivalently,
a shear coordinate of $1$ to $\rho^{(p)}$).
$Y_p^-$ is the monomial coming from laminations that spiral into
$p$ giving a shear coordinate of $-1$ to
$\rho^{(p)}$ (equivalently, to $\rho^{(pq)}$).
Finally, $Y_q^-$ is the monomial coming from laminations that
spiral into $q$ giving a shear coordinate of $-1$ to
$\rho^{(q)}$ (equivalently, to $\rho^{(pq)}$).
See Figure \ref{laminate-spiral}.
\begin{figure}
\input{laminate-spiral.pstex_t}
\caption{}
\label{laminate-spiral}
\end{figure}

When we multiply equations (\ref{P0}) and (\ref{P3})
and subtract the product of (\ref{P1}) and (\ref{P2}),
some terms cancel.  Factoring the remaining
terms, we find that
\begin{equation*}
(x_{\rho'})^2 (x_{\rho} x_{\rho^{(pq)}} - x_{\rho^{(p)}} x_{\rho^{(q)}}) =
Y^+ Y^- (Y_p^+ Y_p^- x_{\alpha^{(p)}} x_{\delta} - x_{\alpha} x_{\delta^{(p)}})
(Y_q^+ Y_q^- x_{\gamma^{(q)}} x_{\beta} - x_{\beta^{(q)}} x_{\gamma}).
\end{equation*}

We now want to interpret each of the terms
$x_{\alpha^{(p)}} x_{\delta}$,
$x_{\alpha} x_{\delta^{(p)}}$,
$x_{\gamma^{(q)}} x_{\beta}$, and
$x_{\beta^{(q)}} x_{\gamma}$
as the left-hand-side of a Ptolemy relation.
To this end, let $\epsilon$ be the arc between $v$ and $w$
which is homotopic to the concatenation of $\alpha$ and $\delta$,
so that $\epsilon$ and $\rho'$ are opposite sides of a bigon with
vertices $v$ and $w$ and internal vertex $p$.  See Figure \ref{bigon}.
\begin{figure}
\input{bigon.pstex_t}
\caption{}
\label{bigon}
\end{figure}

The Ptolemy relations concerning this bigon are
\begin{align*}
x_{\alpha} x_{\delta^{(p)}} &= Y_2 Y_4 x_{\epsilon} + Y_1 x_{\rho'} \\
x_{\delta} x_{\alpha^{(p)}} &=Y_1 Y_3 x_{\rho'} + Y_2 x_{\epsilon}.
\end{align*}

Here $Y_1, Y_2, Y_3$, and $Y_4$ are monomials coming from laminations
that intersect $\alpha$ and $\delta$ as in Figure \ref{bigon-shear}.
\begin{figure}
\input{bigon-sheer.pstex_t}
\caption{}
\label{bigon-shear}
\end{figure}
(See also \cite[Figure 32]{FT}.)    Note that by our assumptions on $T$,
we do not have to worry about laminations that spiral clockwise
into $p$.

A laminate crossing $\rho'$ and spiraling to $p$ must cross $\rho$, so $Y_p^+ Y_p^- = Y_4$.  Therefore
\begin{align*}
Y_p^+ Y_p^- x_{\alpha^{(p)}} x_{\delta} -x_{\alpha} x_{\delta^{(p)}}&=
         Y_4(Y_1 Y_3 x_{\rho'} + Y_2 x_{\epsilon}) -
            ( Y_2 Y_4 x_{\epsilon} + Y_1 x_{\rho'}) \\
   &= Y_1 x_{\rho'}(Y_3 Y_4 - 1) \\
   &= Y_1 x_{\rho'}(\prod_{\tau \in T} y_{\tau}^{e_p(\tau)} - 1),
\end{align*}
since laminates spiraling to $p$ correspond to tagged arcs incident to $p$.

Similarly, letting $\eta$ be the arc between $v$ and $w$ homotopic
to the concatenation of $\beta$ and $\gamma$, so that
$\rho'$ and $\eta$ are opposite sides of a bigon with the interior
point $q$, we get the following Ptolemy relations.
\begin{align*}
x_{\beta} x_{\gamma^{(q)}} &= Y_2' Y_4' x_{\rho'} + Y_1' x_{\eta} \\
x_{\gamma} x_{\beta^{(q)}} &= Y_1' Y_3' x_{\eta} + Y_2' x_{\rho'}.
\end{align*}
Here, $Y_1', Y_2', Y_3', Y_4'$ are defined just as
$Y_1, Y_2, Y_3, Y_4$ were, with $q$ replacing $p$.

Similar to above,  $Y_q^+ Y_q^- = Y_3'$, and
\begin{equation*}
Y_q^+ Y_q^- x_{\gamma^{(q)}} x_{\beta} - x_{\beta^{(q)}} x_{\gamma} =
Y_2' x_{\rho'} (\prod_{\tau \in T} y_{\tau}^{e_q(\tau)} - 1).
\end{equation*}

We now have that
\begin{equation*}
(x_{\rho'})^2 (x_{\rho} x_{\rho^{(pq)}} - x_{\rho^{(p)}} x_{\rho^{(q)}}) =
Y^+ Y^-
 Y_1 x_{\rho'}(\prod_{\tau \in T} y_{\tau}^{e_p(\tau)} - 1)
Y_2' x_{\rho'} (\prod_{\tau \in T} y_{\tau}^{e_q(\tau)} - 1),
\end{equation*}
so
\begin{equation*}
x_{\rho} x_{\rho^{(pq)}} - x_{\rho^{(p)}} x_{\rho^{(q)}} =
Y^+ Y^-
 Y_1 Y_2' (\prod_{\tau \in T} y_{\tau}^{e_p(\tau)} - 1)
 (\prod_{\tau \in T} y_{\tau}^{e_q(\tau)} - 1).
\end{equation*}

Since the monomials $Y^{\pm}, Y_1$ and $Y_2'$ are defined by
laminates crossing the quadrilateral as in Figure \ref{partition}
(which in turn come from tagged arcs of $T$ that have the same
local configuration), it follows that
$$Y^+ Y^- Y_1 Y_2' = \prod_{\tau \in T} y_{\tau}^{e(\tau,\rho)}.$$
This completes the proof when $\rho \notin T$.
\begin{figure}
\input{partition.pstex_t}
\caption{}
\label{partition}
\end{figure}

If $\rho \in T$, the proof is nearly the same.  In this case,
one gets a contribution to the shear coordinates from the laminate
$L_{\rho}$ associated to $\rho$, see Figure \ref{quad}.  Equations
(\ref{P1}) and (\ref{P2}) remain the same,
and equations (\ref{P0}) and (\ref{P3}) become
\begin{align} \label{P0''}
x_{\rho} x_{\rho'} &= Y^+ Y_q^+ Y_p^+ y_{\rho} x_{\beta} x_{\delta} +
                       Y^- x_{\alpha} x_{\gamma}  \\ \label{P3''}
x_{\rho^{(pq)}} x_{\rho'} &= Y^+  x_{\beta^{(q)}} x_{\delta^{(p)}} +
                       Y^- Y_q^- Y_p^- y_{\rho} x_{\alpha^{(p)}} x_{\gamma^{(q)}}.
\end{align}
Using the four Ptolemy relations, i.e. (\ref{P0''})(\ref{P3''}) $- y_\rho$(\ref{P1})(\ref{P2}), we get
\begin{equation*}
x_{\rho'}^2 (x_{\rho} x_{\rho^{(pq)}} - y_\rho x_{\rho^{(p)}} x_{\rho^{(q)}}) =
Y^+ Y^- (y_{\rho} Y_p^+ Y_p^- x_{\alpha^{(p)}} x_{\delta} - x_{\alpha} x_{\delta^{(p)}})
(y_{\rho} Y_q^+ Y_q^- x_{\gamma^{(q)}} x_{\beta} - x_{\beta^{(q)}} x_{\gamma}).
\end{equation*}
In this case $y_{\rho} Y_p^+ Y_p^- = Y_4$ and
$y_{\rho} Y_q^+ Y_q^- = Y_3'$, and the proof continues as before.
\end{proof}

There is a version of Theorem \ref{double-identity}
which makes no assumptions on the notching of arcs in $T$
at $p$ or $q$.  Although we won't use it later, we record
the statement.

\begin{theorem}
Fix a tagged triangulation $T$ of $(S,M)$ which comes from
an ideal triangulation, and
let $\Acal$ be the cluster algebra associated
to $(S,M)$ with principal coefficients with respect to $T$.
Let $p$ and $q$ be punctures in $S$, and let
$\rho$ be an ordinary arc between $p$ and $q$.
Then $$x_{\rho} x_{\rho^{(pq)}}
y_{\rho^{(p)}}^{\chi(\rho^{(p)}\in T)}
y_{\rho^{(q)}}^{\chi(\rho^{(q)}\in T)}
- x_{\rho^{(p)}} x_{\rho^{(q)}}
y_{\rho}^{\chi(\rho \in T)}
y_{\rho^{(pq)}}^{\chi(\rho^{(pq)} \in T)}$$ is equal to
$$\prod_{\tau \in T} y_{\tau}^{e(\tau,\rho)}
(\prod_{\tau\in T} y_{\tau}^{e_p^{\bowtie}(\tau)} -\prod_{\tau\in T} y_{\tau}^{e_p(\tau)})
(\prod_{\tau\in T} y_{\tau}^{e_q^{\bowtie}(\tau)} -\prod_{\tau\in T} y_{\tau}^{e_q(\tau)}),$$ where
$e_p(\tau)$ (respectively, $e_p^{\bowtie}(\tau)$) is the number of ends
of $\tau$ that are incident to the puncture $p$ with an ordinary (respectively, notched) tagging.
\end{theorem}

\begin{remark}
In the degenerate case of
a bordered surface with two punctures $p$ and $q$ and only one other  marked 
point $v$, Theorem \ref{double-identity} still holds and the proof
is analogous.  Here
we let $\rho'$ be a loop based at $v$ crossing $\rho$ exactly once, 
and define $\alpha$, $\beta=\gamma$, and $\delta$ as in Figure \ref{figquad-mark3}.
Note that we can view $\alpha, \beta,\gamma,$ and $\delta$ as the four sides
of a degenerate quadrilateral with diagonals $\rho$ and $\rho'$.
We then obtain the analogues of equations 
(\ref{P0})-(\ref{P3''}), replacing all instances of vertex $w$ with $v$, $\gamma$ with $\beta$, $x_{\beta}x_{\beta^{(p)}}$ with $x_{\rho'}$, $Y_2'$ with $1$, and $Y_q^+ Y_q^-$ with $\prod_{\tau \in T} y_{\tau}^{e_q(\tau)}$.
\end{remark}

\begin{figure}
\input{figquad-mark3.pstex_t}
\caption{}
\label{figquad-mark3}
\end{figure}

\begin{remark}
In the degenerate case when $p=q$, Theorem \ref{double-identity} also holds,
but we need to make sense of notation such as $x_{\rho^{(p)}}$ when $\rho$
is a loop.  See Section \ref{sec DNL}.
\end{remark}

\subsection{Combinatorial identities satisfied by $\zg$-compatible pairs of matchings}

\label{DoubleNotchedProofs}

We now use Theorem \ref{double-identity}
to prove Proposition \ref{prop double-special}, where $\zg \in T^\circ$, and then Theorem \ref{thm double}, where $\zg \not \in T^\circ$.  
In both  proofs, we will use Theorems \ref{thm main} and \ref{thm single} to replace appearances of cluster variables $x_\zg$, $x_{\zg^{(p)}}$, and $x_{\zg^{(q)}}$ with generating functions of perfect (and $\zg$-symmetric) matchings of graphs $G_{T^\circ,\zg}$, $G_{T^\circ,\ell_p}$ and $G_{T^\circ,\ell_q}$.  We are then reduced to proving combinatorial identities concerning these sets of matchings.

\begin{lemma} \label{MinMaxProdInT} Assume that the ideal triangulation $T^\circ$ contains the arc $\zg$ between the punctures $q$ and $p$ ($p\not = q$).  
Let $\ell_p$ denote the loop based at puncture $q$ enclosing the arc $\gamma$ and puncture $p$, but no other marked points. Let $P_-(\ell_p)$ and $P_+(\ell_p)$ denote the minimal and maximal matchings of $G_{T^\circ, \ell_p}$, respectively.  Define $\ell_q$, $P_-(\ell_q)$, and $P_+(\ell_q)$ analogously.  Assume the local configuration around arc $\zg$ and punctures $p$ and $q$ is as in Figure \ref{figLocalQuad}.  Let $\zeta_1 = \zg$ and $\zeta_2$ through $\zeta_{e_p}$ label the arcs that $\ell_p$ crosses as we follow it clockwise around puncture $p$.  Analogously, let $\eta_1 = \zg$ and $\eta_2$ through $\eta_{e_q}$ label the arcs that $\ell_q$ crosses as we follow it clockwise around puncture $q$.  Then we have the following.

\begin{figure}
\input{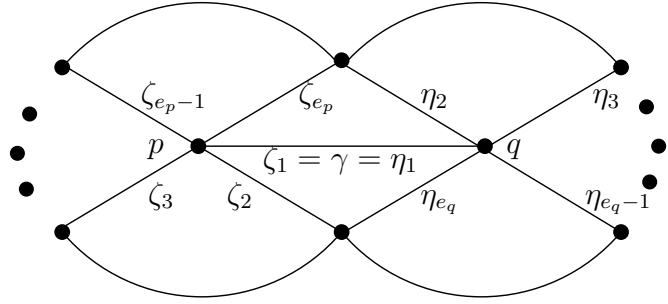}
\caption{The local configuration around arc $\zg$ between  $p$ and $q$.} \label{figLocalQuad}
\end{figure}

\begin{eqnarray}
\label{loopMinP}
x(P_-(\ell_p)) \, h(P_-(\ell_p))  &=&
x_{\zg} \big( \prod_{j=2}^{e_p-1} x_{\zeta_j}\big) x_{\eta_{2}},
\\
\label{loopMaxP}
x(P_+(\ell_p)) \, h(P_+(\ell_p))    &=&
x_{\eta_{e_q}} \big( \prod_{j=3}^{e_p} x_{\zeta_j}\big) x_{\zg}
\big(\prod_{j=2}^{e_p} h_{\zeta_j}\big).
\end{eqnarray}
There are analogous identies for $\ell_q$, which we get by
replacing $p$ with $q$
and switching the $\eta$'s and $\zeta$'s.
\end{lemma}

\begin{proof}
The minimal and maximal matchings are precisely those that
contain only boundary edges.  We distinguish between the two based on the fact that arcs $\zeta_{e_p}$, $\zg$, and $\zeta_2$ are assumed to be given in clockwise order, as are $\eta_{e_q}$, $\zg$, and $\eta_2$.  The edges of the minimal and maximal matchings both have a regular alternating pattern on the interior of $H_\zeta$ (resp. $H_\eta$).  See Figure \ref{figInT} for the verification of equation (\ref{loopMinP}).  The weights in the other  equation are analogous.

\begin{figure}
\input{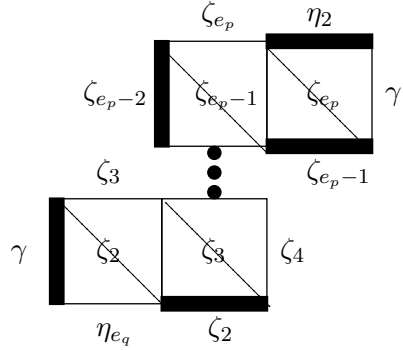}
\caption{The graph $G_{T^\circ,\ell_p}$ with the minimal matching $P_-(\ell_p)$ highlighted.} \label{figInT}
\end{figure}

The height monomial of a minimal matching is $1$, and the height monomial of a maximal matching of a graph is the product of $h_{\tau_i}$'s, one for each label of a tile in the graph.  Thus, looking at the diagonals (i.e. labels) of the tiles in $G_{T^\circ,\ell_p}$ and $G_{T^\circ,\ell_q}$ completes the proof.
\end{proof}

\begin{proof} [Proof of Proposition \ref{prop double-special}]
We define $\zeta_1$ through $\zeta_{e_p}$ and $\eta_1$ through $\eta_{e_q}$ 
as in Lemma \ref{MinMaxProdInT}.  Based on Lemma \ref{MinMaxProdInT}, it follows that
\begin{eqnarray*}
x(P_-(\ell_q))\, h(P_-(\ell_q)) \, x(P_+(\ell_p)) \, h(P_+(\ell_p)) \, h_{\zg} &= &
x_{\zg}^2 h_{\zg} \prod_{j=2}^{e_p} (x_{\zeta_j} h_{\zeta_j})  \prod_{j=2}^{e_q} x_{\eta_j}
\hspace{1.5em}\mathrm{~and} \\
x(P_-(\ell_p))\, h(P_-(\ell_p)) \, x(P_+(\ell_q)) \, h(P_+(\ell_q)) \, h_{\zg} &= &
x_{\zg}^2 h_{\zg} \prod_{j=2}^{e_p} x_{\zeta_j}  \prod_{j=2}^{e_q} (x_{\eta_j} h_{\eta_j}).
\end{eqnarray*}
\begin{eqnarray} \label{matchInT}
&&\hspace{-3em}\text{Therefore }\frac{\sum_{P_p \in \mathcal{P}(\ell_p)} x(P_p) \, h(P_p)}{x_{\zg}x_{\zeta_2}x_{\zeta_3} \cdots x_{\zeta_{e_p}}}
\cdot
\frac{ \sum_{P_q \in \mathcal{P}(\ell_q)} x(P_q) \, h(P_q)}{x_{\zg}x_{\eta_2}x_{\eta_3} \cdots x_{\eta_{e_q}}}
\cdot h_{\zg} \\
\nonumber &-&
h_\zg \big(\prod_{j=2}^{e_p} h_{\zeta_j}\big) - h_\zg \big(\prod_{j=2}^{e_q}
h_{\eta_j}\big)
+ 1 +
h_\zg^2 \big(\prod_{j=2}^{e_p} h_{\zeta_j}\big)\big( \prod_{j=2}^{e_q}
h_{\eta_j}\big)
\end{eqnarray}
is positive,
since $P_\pm(\ell_p) \in \mathcal{P}(\ell_p)$, $P_\pm(\ell_q) \in \mathcal{P}(\ell_q)$, and thus the two negative terms cancel with terms coming from the product of Laurent polynomials.

Since we assumed that 
$T$ does not contain any arcs with notches at $p$ or $q$, it follows that $\Phi(h_\zg) = y_\zg$,  $\Phi(h(P_p)) = y(P_p)$, 
$\prod_{j=2}^{e_p} \Phi(h_{\zeta_j}) = \prod_{\tau \in T} y_\tau^{e_p(\tau)}$, and $\prod_{j=2}^{e_q} \Phi(h_{\eta_j}) = \prod_{\tau \in T} y_\tau^{e_q(\tau)}$. Applying $\Phi$ to (\ref{matchInT}), we obtain
that
\begin{equation}\label{matchInT2}
\frac{\sum_{P_p \in \mathcal{P}(\ell_p)} x(P_p) \, y(P_p)}{x_{\zg}
x_{\zeta_2}x_{\zeta_3} \cdots x_{\zeta_{e_p}}}
\cdot
\frac{ \sum_{P_q \in \mathcal{P}(\ell_q)} x(P_q) \, y(P_q)}{x_{\zg}  
x_{\eta_2}x_{\eta_3} \cdots x_{\eta_{e_q}}}
\cdot y_{\zg} 
+ (1- \prod_{\tau \in T} y_\tau^{e_p(\tau)})(1 - \prod_{\tau \in T} y_\tau^{e_q(\tau)})
\end{equation}
is positive.  Since $\zg \in T$, $x_{\zg}$ is an initial cluster variable and the left-hand-side of (\ref{matchInT2}) can be rewritten using Remark \ref{rem:enough}.  Theorem \ref{double-identity} then gives
\begin{eqnarray*}
x_{\zg^{(pq)}}
\label{TurnToMatchings} &=& \frac{x_{\zg^{(p)}}x_{\zg^{(q)}} y_\zg + \big(1- \prod_{\tau \in T} y_\tau^{e_p(\tau)}\big)\big(1 - \prod_{\tau \in T} y_\tau^{e_q(\tau)}\big)}{x_\zg}.
\end{eqnarray*}  
It follows that the cluster expansion of $x_{\zg^{(pq)}}$ is
positive.
\end{proof}
For the remainder of this section, we assume that $\zg$ (as well as $\zg^{(p)}, \zg^{(q)}$ and $\zg^{(pq)}$) is not in the tagged triangulation $T$.  We use the notation of Definition \ref{ppANDsp}, and additionally we let
$\mathcal{CP}(\zg^{(pq)})$ denote the set of pairs of $\zg$-compatible matchings $(P_p,P_q)$ of $G_{T^\circ,\zLp} \sqcup G_{T^\circ,\zLq}$, and let $\mathcal{P}(\zeta)$ (resp. $\mathcal{P}(\zeta^{(i)})$, $\mathcal{P}(\eta)$, and $\mathcal{P}(\eta^{(i)})$) denote the set of perfect matchings of $H_\zeta$ (resp. $H_{\zeta}^{(i)}$, $H_\eta$, and $H_{\eta}^{(i)}$).
Following Section \ref{sect notched}, we label the tiles of $G_{T^\circ,\ell_p}$ so that they match the labels of the arcs crossed as we travel along $\ell_p$ in clockwise order:
$$G_{\tau_{i_1}}, G_{\tau_{i_2}}, \dots, G_{\tau_{i_d}}, G_{\zeta_1}, G_{\zeta_2},\dots, G_{\zeta_{e_p-1}},G_{\zeta_{e_p}},G_{\tau_{i_d}},G_{\tau_{i_{d-1}}},\dots, G_{\tau_{i_2}}, G_{\tau_{i_1}}.$$
See Figure \ref{LoopSnakeGraph}.
Analogously, the tiles of $G_{T^\circ, \ell_q}$ are labeled so that they match the arcs crossed as we travel along $\ell_q$ in clockwise order:
$$G_{\tau_{i_d}}, G_{\tau_{i_{d-1}}}, \dots, G_{\tau_{i_1}}, G_{\eta_1}, G_{\eta_2},\dots, G_{\eta_{e_q-1}},G_{\eta_{e_q}},G_{\tau_{i_1}},G_{\tau_{i_2}},\dots, G_{\tau_{i_{d-1}}}, G_{\tau_{i_d}}.$$
It follows that the tiles $G_{\tau_{i_d}}$ in both $G_{T^\circ, \ell_p}$ and $G_{T^\circ, \ell_q}$ have two adjacent sides labeled $\zeta_1$ and $\zeta_{e_p}$, and the tiles $G_{\tau_{i_1}}$ contain two adjacent sides labeled $\eta_1$ and $\eta_{e_q}$.

We let $J_{T^\circ,\zg_p,i}$ denote 
the induced subgraph obtained from $G_{T^\circ,\zg_p,i}$ by deleting vertices $v_i$ and $w_i$, and all edges incident to either of these two vertices.

\begin{lemma} \label{twoendsSymmetric} If $P$ is a $\gamma$-symmetric matching of $G_{T^\circ,\ell_p}$ then $P$ can be partitioned into three perfect matchings of subgraphs in exactly one of the two following ways:
\begin{enumerate}
\item $P = P|_{G_{T^\circ, \zg_p, 1}} \sqcup
P|_{H_\zeta^{(1)}}
\sqcup P|_{J_{T^\circ,\zg_p,2}}$, or
\item $P = P|_{J_{T^\circ, \zg_p, 1}} \sqcup
P|_{H_\zeta^{(2)}}
\sqcup P|_{G_{T^\circ,\zg_p,2}}$.
\end{enumerate}
\end{lemma}
\begin{proof}
See Figures \ref{LoopSnakeGraphWV} and \ref{figMinM}.  We will divide the set of $\gamma$-symmetric matchings  of $G_{T^\circ,\ell_p}$ into two classes, depending on whether or not they contain one of the edges labeled $\tau_{i_{d-1}}$ on the tiles containing vertex $v_1$ and $v_2$.  
By definition,
 a $\gamma$-symmetric matching must contain both of these edges or neither.

(1) If $P$ contains the specified edges, then $P$ must also contain the edge labeled $\zeta_1$ that is incident to vertex $v_1$.  (Otherwise, vertex $v_1$ could only be covered by the edge labeled $\zeta_2$ and this would leave a connected component with an odd number of vertices to match together.)  Filling in the rest of the edges on tiles $G_{\tau_{i_1}}$ through $G_{\tau_{i_{d-1}}}$, we see that  $P|_{G_{T^\circ,\gamma_p,1}}$ is a perfect matching.  Such a $P$ does not contain the edge labeled $\tau_{i_d}$ on the tile 
$G_{\tau_{i_{d-1}}}$ since that would also leave a connected component with an odd number of vertices.  Consequently, vertices $v_2$ and $w_2$ must be covered by edges from the tile $G_{\zeta_{e_p}}$.  We conclude that the remainder of the set $P$ can be decomposed disjointly as the perfect matchings $P|_{H_\zeta^{(1)}}$ and $P|_{J_{T^\circ,\gamma_p,2}}$.

(2) If $P$ does not contain the specified edges, then $P$ must contain the edge labeled $\zeta_{e_p}$ that is incident to $v_2$. (Otherwise the vertex where edges labeled $\zeta_{e_p}$ and $\tau_{i_{d-1}}$ meet on that tile would not be covered by an edge of $P$.)  Filling in the rest of $P$, we see that it restricts to a perfect matching on $G_{T^\circ,\gamma_p,2}$.  Since the edge labeled $\tau_{i_{d-1}}$ incident to $w_1$
is not in $P$, the edge $\zeta_1$ incident to $v_1$ cannot be contained in $P$. (Otherwise vertex $w_1$ could only be covered by the edge labeled $\tau_{i_d}$ and this also leaves an odd number of vertices to match together.)  We conclude that the rest of the set $P$ can be decomposed disjointly as the perfect matchings $P|_{J_{T^\circ,\gamma_p,1}}$ and $P|_{H_\zeta^{(2)}}$.

As $P$ either contains or does not contain the specified edges, the proof is complete.
\end{proof}

\begin{remark}\label{rem:partition}
By Lemma \ref{twoendsSymmetric}, 
it is impossible for both the edge labeled $\zeta_1$ incident
to $v_1$ (resp. $v_2$) {\it and} the edge labeled
$\zeta_{e_p}$ incident to $v_2$ (resp. $v_1$) to appear in a
$\zg$-symmetric matching of $G_{T^\circ,\ell_p}$.  
Furthermore Case (1) of Lemma \ref{twoendsSymmetric} corresponds to the case where $P$ contains one edge labeled $\zeta_1$ incident to $v_1$ or $v_2$, but does not contain either edge labeled $\zeta_{e_p}$ incident to $v_1$ or $v_2$.  Case (2) is the reverse, and analogous statements hold for edges labeled $\eta_1$ and $\eta_{e_q}$ in $G_{T^\circ,\ell_q}$. 
\end{remark}
We use this observation to 
partition various sets of matchings into  disjoint sets.
\begin{definition} [\emph{$\mathcal{P}_{a,b}(\zg)$, $\mathcal{SP}_{a,b}(\zg^{(p)})$, and $\mathcal{CP}_{a,b}(\zg^{(pq)})$}] \label{partss}
For $a \in \{1,e_p\}$ and $b \in \{1,e_q\}$, let
$\mathcal{P}_{a,b}(\zg)$ (resp. $\mathcal{SP}_{a,b}(\zg^{(p)})$ and $\mathcal{CP}_{a,b}(\zg^{(pq)}))$ denote the set of matchings 
in $\mathcal{P}(\zg)$ (resp. $\mathcal{SP}(\zg^{(p)})$ and $\mathcal{CP}(\zg^{(pq)})$) that contains at least one edge labeled $\zeta_a$ and at least one edge labeled $\eta_b$.
\end{definition}

By Remark \ref{rem:partition}, we have the following.
\begin{align} \label{eq-part}
\mathcal{P}(\zg) &=
\mathcal{P}_{1,1}(\zg) \sqcup
\mathcal{P}_{1,e_q}(\zg) \sqcup
\mathcal{P}_{e_p,1}(\zg) \sqcup
\mathcal{P}_{e_p,e_q}(\zg), \\
\label{eq-part2}
\mathcal{SP}(\zg^{(p)}) &=
\mathcal{SP}_{1,1}(\zg^{(p)}) \sqcup
\mathcal{SP}_{1,e_q}(\zg^{(p)}) \sqcup
\mathcal{SP}_{e_p,1}(\zg^{(p)}) \sqcup
\mathcal{SP}_{e_p,e_q}(\zg^{(p)}),\\
\label{eq-part3}
\mathcal{CP}(\zg^{(pq)}) &=
\mathcal{CP}_{1,1}(\zg^{(pq)}) \sqcup
\mathcal{CP}_{1,e_q}(\zg^{(pq)}) \sqcup
\mathcal{CP}_{e_p,1}(\zg^{(pq)}) \sqcup
\mathcal{CP}_{e_p,e_q}(\zg^{(pq)}).
\end{align}

\noindent We let $\mathcal{P}_1(\zeta)$ (resp. $\mathcal{P}_{e_p}(\zeta)$) denote the subset of perfect matchings of $H_{\zeta}$ that contains the edge labeled $\zeta_1$ (resp. $\zeta_{e_p}$) on the tile $G_{\zeta_{e_p}}$ (resp. $G_{\zeta_1})$.  We define $\mathcal{P}_b(\eta)$, $\mathcal{P}_a(\zeta^{(i)})$, and $\mathcal{P}_{b}(\eta^{(i)})$ analogously for graphs $H_\eta$, $H_\zeta^{(i)}$, and $H_\eta^{(i)}$.  We also define the following.
\begin{eqnarray} 
\mathcal{M}_{a,b}(\zg) &=& \sum_{P \in \mathcal{P}_{a,b}(\zg)}
x(P)h(P), \\
\mathcal{M}(\zg) &=&  \sum_{P \in \mathcal{P}(\zg)} x(P) h(P), \\
\mathcal{SM}(\zg^{(p)}) &=& \sum_{P_p \in \mathcal{SP}(\zg^{(p)})} \overline{x}(P_p)\overline{h}(P_p), \\
\mathcal{CM}(\zg^{(pq)}) &=& \sum_{(P_p,P_q) \in \mathcal{CP}(\zg^{(pq)})}
\overline{\overline{x}}(P_p,P_q)\overline{\overline{h}}(P_p,P_q), \\
\label{eqZ1} \mathcal{M}_a(\zeta) &=& \sum_{P \in \mathcal{P}_a(\zeta)} x(P)h(P), \\
\label{eqZ2} \mathcal{M}(\zeta^{(1)}) &=&
\frac{\sum_{P \in \mathcal{P}(\zeta)} x(P)h(P)}{x_{\zeta_1}}, \mathrm{~and} \\
\label{eqZ3} \mathcal{M}(\zeta^{(2)}) &=&
\frac{h_{\zeta_{e_p}}\, \sum_{P \in \mathcal{P}(\zeta)} x(P)h(P)}{x_{\zeta_{e_p}}}.
\end{eqnarray}
We define $\mathcal{M}_b(\eta)$, $\mathcal{M}(\eta^{(1)})$, and $\mathcal{M}(\eta^{(2)})$ analogously.  In equations (\ref{eqZ1})-(\ref{eqZ3}), $h(P)$ is the height monomial with respect to the relevant subgraph.

\begin{lemma} \label{DisjointPartition}
By Remark  \ref{BendingAssumpt}, we can assume without loss
of generality that the tiles $G_{\tau_{i_{d-1}}}, G_{\tau_{i_d}}$, and $G_{\zeta_1}$ (resp. $G_{\tau_{i_2}}, G_{\tau_{i_1}}$, and $G_{\eta_1}$) all lie in a single row or column, while the tiles $G_{\zeta_{e_p}}, G_{\tau_{i_d}}$, and $G_{\tau_{i_{d-1}}}$(resp. $G_{\eta_{e_q}}, G_{\tau_{i_1}}$, and $G_{\tau_{i_2}}$) do not.
Then
\begin{align*}
\mathcal{M}(\zg)& =
\mathcal{M}_{1,1}(\zg) +
\mathcal{M}_{1,e_q}(\zg) +
\mathcal{M}_{e_p,1}(\zg) +
\mathcal{M}_{e_p,e_q}(\zg); \\
\mathcal{SM}(\zg^{(p)})
&=
(\mathcal{M}_{1,1}(\zg) +
\mathcal{M}_{1,e_q}(\zg))\mathcal{M}(\zeta^{(1)}) 
+ (\mathcal{M}_{e_p,1}(\zg) +
\mathcal{M}_{e_p,e_q}(\zg))\mathcal{M}(\zeta^{(2)}); \\
\mathcal{SM}(\zg^{(q)})
&=
(\mathcal{M}_{1,1}(\zg) + \mathcal{M}_{e_p,1}(\zg))
\mathcal{M}(\eta^{(1)}) 
+(\mathcal{M}_{1,e_q}(\zg) + \mathcal{M}_{e_p,e_q}(\zg))
\mathcal{M}(\eta^{(2)}); \\
\mathcal{CM}(\zg^{(pq)}) &=
\mathcal{M}_{1,1}(\zg)\mathcal{M}(\zeta^{(1)})\mathcal{M}(\eta^{(1)}) +
\mathcal{M}_{1,e_q}(\zg)\mathcal{M}(\zeta^{(1)})\mathcal{M}(\eta^{(2)})  \\ 
\nonumber
&+
\mathcal{M}_{e_p,1}(\zg)\mathcal{M}(\zeta^{(2)})\mathcal{M}(\eta^{(1)})  +
\mathcal{M}_{e_p,e_q}(\zg)\mathcal{M}(\zeta^{(2)})\mathcal{M}(\eta^{(2)}) .
\end{align*}
\end{lemma}

\begin{proof}
The identity for $\mathcal{M}(\zg)$ follows directly from 
(\ref{eq-part}).  We use (\ref{eq-part2}) to get 
\begin{align}
\label{MMstar}
\mathcal{SM}(\zg^{(p)}) &= \sum_{P_p \in \mathcal{SP}_{1,1}(\zg^{(p)})} \overline{x}(P_p)\overline{h}(P_p) +
\sum_{P_p \in \mathcal{SP}_{e_p,1}(\zg^{(p)})} \overline{x}(P_p)\overline{h}(P_p) \\
\nonumber
&+
\sum_{P_p \in \mathcal{SP}_{1,e_q}(\zg^{(p)})} \overline{x}(P_p)\overline{h}(P_p) +
\sum_{P_p \in \mathcal{SP}_{e_p,e_q}(\zg^{(p)})} \overline{x}(P_p)\overline{h}(P_p).
\end{align}
By Lemma \ref{twoendsSymmetric}, a $\zg$-symmetric matching $P_p$ of $G_{T^\circ, \zLp}$ restricts to the disjoint union 
of  perfect matchings of $$G_{T^\circ,\zg_p,1} \sqcup H_{\zeta}^{(1)} \sqcup J_{T^\circ,\zg_p,2} \hspace{1em} \mathrm{~or~}
\hspace{1em}
 J_{T^\circ,\zg_p,1} \sqcup H_{\zeta}^{(2)} \sqcup G_{T^\circ,\zg_p,2}
,$$
depending on whether $P_p$ contains an edge labeled $\zeta_1$ or $\zeta_{e_p}$, respectively.  
Accordingly, the weight 
\begin{align*}
\overline{x}(P_p) &= 
\frac{x(P_p)}{x(P_p|_{G_{T^\circ,\zg_p,1}})} = 
x(P_p|_{H_{\zeta}^{(1)}}) \, x(P_p|_{J_{T^\circ,\zg_p,2}})
= 
x(P_p|_{H_{\zeta}^{(1)}}) \, \frac{x(P_p|_{G_{T^\circ,\zg_p,2}})}{x_{\zeta_1}} \hspace{1em} \mathrm{~or} \\ 
&= 
\frac{x(P_p)}{x(P_p|_{G_{T^\circ,\zg_p,2}})} = 
x(P_p|_{H_{\zeta}^{(2)}}) \, x(P_p|_{J_{T^\circ,\zg_p,1}})
= 
x(P_p|_{H_{\zeta}^{(2)}}) \, \frac{x(P_p|_{G_{T^\circ,\zg_p,2}})}{x_{\zeta_{e_p}}}
,
\end{align*}
respectively.

To calculate the height, we note that the minimal matching $P_-(\ell_p)$ appears in the subset $\mathcal{SP}_{1,1}(\ell_p) \sqcup \mathcal{SP}_{1,e_q}(\ell_p)$,
so $$\overline{h}(P_p) = 
\frac{h(P_p)}{h(P_p|_{G_{T^\circ,\zg_p,1}})} = 
h(P_p|_{H_{\zeta}^{(1)}}) \, h(P_p|_{J_{T^\circ,\zg_p,2}})
= 
h(P_p|_{H_{\zeta}^{(1)}}) \, h(P_p|_{G_{T^\circ,\zg_p,2}})$$ in the case that $P_p \in 
\mathcal{SP}_{1,1}(\ell_p) \sqcup \mathcal{SP}_{1,e_q}(\ell_p)$.
On the other hand, any $\zg$-symmetric matching in 
$\mathcal{SP}_{e_p,1}(\zg^{(p)}) \sqcup \mathcal{SP}_{e_p,e_q}(\zg^{(p)})$ 
has a height monomial scaled by a factor of $h_{\zeta_{e_p}}$.
Thus
$$\overline{h}(P_p) = 
\frac{h(P_p)}{h(P_p|_{G_{T^\circ,\zg_p,2}})} = 
h(P_p|_{H_{\zeta}^{(2)}}) \, h(P_p|_{J_{T^\circ,\zg_p,1}})
= 
h_{\zeta_{e_p}} \, h(P_p|_{H_{\zeta}^{(2)}}) \, h(P_p|_{G_{T^\circ,\zg_p,1}})$$ in the case that $P_p \in 
\mathcal{SP}_{e_p,1}(\ell_p) \sqcup \mathcal{SP}_{e_p,e_q}(\ell_p)$.  We thus can rewrite (\ref{MMstar}) as 
\begin{align}
\mathcal{SM}(\gamma^{(p)}) &= 
\sum_{P_1 \in \mathcal{P}_{1,1}(\gamma)\sqcup \mathcal{P}_{1,e_q}(\gamma)} \, \, \sum_{P_2 \in \mathcal{P}_{1}(\zeta^{(1)})}  
\frac{x(P_1)}{x_{\zeta_1}}h(P_1) x(P_2) h(P_2)
\\
\nonumber &+ \sum_{P_1 \in \mathcal{P}_{e_p,1}(\gamma)\sqcup \mathcal{P}_{e_p,e_q}(\gamma)}\,\,\sum_{P_2 \in \mathcal{P}_{e_p}(\zeta^{(2)})} 
\frac{x(P_1)}{x_{\zeta_{e_p}}}h(P_1) x(P_2) h(P_2) h_{\zeta_{e_p}},
\end{align}
thus showing the identity for $\mathcal{SM}(\zg^{(p)})$ (and $\mathcal{SM}(\zg^{(q)})$).

The formula for $\mathcal{CM}(\zg^{(pq)})$ follows by similar logic since specifying the four ends of a $\zg$-compatible pair of matchings of $G_{T^\circ, \zLp}$ and $G_{T^\circ, \zLq}$ also specifies which of the two cases of Lemma \ref{twoendsSymmetric} we are in for both $G_{T^\circ,\zLp}$ and $G_{T^\circ,\zLq}$.
\end{proof}

Lemma \ref{DisjointPartition} immediately implies the following.

\begin{lemma} \label{mainLemma}
The expression $\mathcal{CM}(\zg^{(pq)}) \mathcal{M}(\zg) - \mathcal{SM}(\zg^{(p)})
\mathcal{SM}(\zg^{(q)})$ equals
\begin{equation*}
\label{threefactors}
\big(\mathcal{M}_{1,1}(\zg) \mathcal{M}_{e_p,e_q}(\zg) -
\mathcal{M}_{1,e_q}(\zg) \mathcal{M}_{e_p,1}(\zg) \big) 
 \big(\mathcal{M}(\zeta^{(1)})-\mathcal{M}(\zeta^{(2)})\big)
\, \big(\mathcal{M}(\eta^{(1)})-\mathcal{M}(\eta^{(2)})\big).
\end{equation*}

\end{lemma}

The next two results describe how to simplify the three factors in (\ref{threefactors}).

\begin{lemma} \label{firstfactor} We have
\begin{equation} \label{eqnfirstfactor} \Phi\big(\mathcal{M}_{1,1}(\zg) \mathcal{M}_{e_p,e_q}(\zg) -
\mathcal{M}_{1,e_q}(\zg) \mathcal{M}_{e_p,1}(\zg) \big)
=
x_{\tau_{i_1}} x_{\tau_{i_d}} x_{\zeta_1}x_{\zeta_{e_p}}x_{\eta_1}x_{\eta_{e_q}}\prod_{j=2}^{d-1} x_{\tau_{i_j}}^2
\prod_{\tau \in T} y_\tau^{e(\tau,\zg)}.
\end{equation}
\end{lemma}

\begin{proof}
The main idea is that a superposition of two matchings 
corresponding to the first term on the left-hand-side of (\ref{eqnfirstfactor}) 
can be decomposed into a superposition of two matchings 
corresponding to the second term on the left-hand-side of (\ref{eqnfirstfactor}) in all cases except for one. This  case corresponds to the right-hand-side of (\ref{eqnfirstfactor}).

We denote by $P_1 + P_2$ the multigraph given by the superposition of $P_1$ and $P_2$.  Let $P_1$ be an element of $\mathcal{P}_{1,1}(\zg)$ and $P_2$ an element of $\mathcal{P}_{e_p,e_q}(\zg)$.
Since $G_{T^\circ,\zg}$ is bipartite, it follows that $P_1+P_2$ consists of a disjoint union of cycles of even length (including doubled edges which we treat as cycles of length two).

By definition, $P_1$ contains the edge labeled $\zeta_{1}$ on the tile $G_{\tau_{i_d}}$ while $P_2$ contains the edge labeled $\zeta_{e_p}$ on $G_{\tau_{i_d}}$.  Similarly, $P_1$ contains the edge labeled $\eta_1$ on $G_{\tau_{i_1}}$ while $P_2$ contains the edge labeled $\eta_{e_q}$ on $G_{\tau_{i_1}}$.  
Consequently, the superposition $P_1+P_2$ contains at least one cycle of length greater than two, and one such cycle must contain the edges labeled $\zeta_1$ and $\zeta_{e_p}$ on the tile $G_{\tau_{i_d}}$, and one must contain the edges labeled $\eta_1$ and $\eta_{e_q}$ on the tile $G_{\tau_{i_1}}$.

Let $k$ be the number of cycles in $P_1+P_2$ of length greater than $2$ which do not involve edges on tiles $G_{\tau_{i_d}}$ or $G_{\tau_{i_1}}$.  Then there are $2^k$ ways of decomposing $P_1+P_2$ into the superposition of two matchings, one from $\mathcal{P}_{1,1}(\zg)$ and one from $\mathcal{P}_{e_p,e_q}(\zg)$.
When $P_1+P_2$ has at least two cycles of length greater than $2$,
there are also $2^k$ ways to decompose $P_1+P_2$ into the superposition of two matchings with
one from set 
$\mathcal{P}_{e_p,1}(\zg)$ and one from $\mathcal{P}_{1,e_q}(\zg)$.
Thus we have a weight-preserving and height-preserving bijection
between such superpositions.

The superposition of the minimal matching $P_-(\zg) \in \mathcal{P}_{1,1}(\zg)$ and the maximal matching $P_+(\zg) \in \mathcal{P}_{e_p,e_q}(\zg)$ is of the form $P_1+P_2$, but consists of a single cycle including all edges on the boundary of $G_{T^\circ,\zg}$.
Recall that the sets  $\mathcal{P}_{1,1}(\zg), \mathcal{P}_{e_p,1}(\zg), \mathcal{P}_{1,e_q}(\zg),$ and $\mathcal{P}_{e_p,e_q}(\zg)$ are disjoint.  Accordingly, a single cycle cannot decompose into a superposition of an element of $\mathcal{P}_{e_p,1}(\zg)$ and an element of $\mathcal{P}_{1,e_q}(\zg)$ because $P_-(\zg)$ and $P_+(\zg)$ are the unique two perfect matchings of a single cycle including all edges on the boundary of $G_{T^\circ,\zg}$.  It follows that any superposition of an element in $\mathcal{P}_{e_p,1}(\zg)$ and an element in $\mathcal{P}_{1,e_q}(\zg)$ must contain at least two cycles, and 
is also of the form $P_1+P_2$, with $P_1 \in \mathcal{P}_{1,1}(\zg)$ and $P_2 \in \mathcal{P}_{e_p,e_q}(\zg)$.

In conclusion, the only monomial not canceled in the difference on the left-hand-side of (\ref{eqnfirstfactor}) corresponds to the superposition of  $P_-(\zg)$ and $P_+(\zg)$, which includes all edges on the boundary.  
To calculate the weight, we note that on every tile $G_{\tau_{i_j}}$, for $2 \leq j \leq d-1$, there are exactly two adjacent tiles that include edges on the boundary with weight $x_{\tau_{i_j}}$, see Figure \ref{figglue}. 
On the other hand, the tiles $G_{\tau_{i_1}}$ and $G_{\tau_{i_d}}$ only have one adjacent tile each with an edge on the boundary with weight $x_{\tau_{i_1}}$ (resp. $x_{\tau_{i_d}}$).  The remaining two boundary edges of $G_{\tau_{i_1}}$ have weights $x_{\eta_1}$ and $x_{\eta_{e_q}}$, while the other two boundary edges of $G_{\tau_{i_d}}$ have weights $x_{\zeta_1}$ and $x_{\zeta_{e_p}}$.

The product of heights is $1\cdot \prod_{j=1}^{d} h_{\tau_{i_j}}$, the height monomial for the minimal matching multiplied by the height monomial for the maximal matching.  This specializes to $\prod_{\tau \in T} y_\tau^{e(\tau,\zg)}$ under the map $\Phi$.
\end{proof}

\begin{lemma} \label{secondthirdfactor}
We have the following two identities:
$$\Phi\big(\mathcal{M}(\zeta^{(1)})-\mathcal{M}(\zeta^{(2)})\big) =
x_{\tau_{i_d}}\big(\prod_{j=2}^{e_p-1} x_{\zeta_j}\big)
\big(1- \prod_{\tau \in T} y_\tau^{e_p(\tau)}\big) \text{ and }$$
$$\Phi\big(\mathcal{M}(\eta^{(1)})-\mathcal{M}(\eta^{(2)})\big) =
x_{\tau_{i_1}}\big(\prod_{j=2}^{e_q-1} x_{\eta_j}\big)
\big(1- \prod_{\tau \in T} y_\tau^{e_q(\tau)}\big).$$
\end{lemma}

\begin{proof} It suffices to prove the first identity.
The idea is to show that almost all terms on the left-hand-side cancel except for two, which correspond to the two monomials on the right.  Recall the notation preceding Lemma \ref{DisjointPartition}.

The union of a perfect matching of $H_{\zeta}^{(1)}$ (resp. $H_{\zeta}^{(2)}$) and the edge labeled $\zeta_{e_p}$ (resp. $\zeta_1$) on $G_{\zeta_1}$ (resp. $G_{\zeta_{e_p}}$) is an element of the set $\mathcal{P}_1(\zeta)$ (resp. $\mathcal{P}_{e_p}(\zeta)$).  The minimal height of a matching in $\mathcal{P}_1(\zeta)$ is $h_{\zeta_{e_p}}$ while subset $\mathcal{P}_{e_p}(\zeta)$ contains the perfect matching of $H_\zeta$ with a height monomial of $1$.  We accordingly obtain the identities 
\begin{eqnarray*}
\mathcal{M}_1(\zeta) &=& x_{\zeta_1} (x_{\zeta_{e_p}}
\mathcal{M}(\zeta^{(2)})) \text{ and }
\mathcal{M}_{e_p}(\zeta) = x_{\zeta_{e_p}} (x_{\zeta_1} \mathcal{M}(\zeta^{(1)})), \text{ and so } \end{eqnarray*}  
\begin{eqnarray} \label{PhiPhi} 
\hspace{1em} &&
\Phi\big(\mathcal{M}(\zeta^{(1)})-\mathcal{M}(\zeta^{(2)})\big) =
\frac{\Phi\big(\mathcal{M}_{e_p}(\zeta) - \mathcal{M}_1(\zeta)\big)}
{x_{\zeta_1} x_{\zeta_{e_p}}} 
= \frac{\Phi\big( \sum_P x(P) h(P)\big)}
{x_{\zeta_1} x_{\zeta_{e_p}}},
\end{eqnarray}
where the sum is over $P \in (\mathcal{P}_1(\zeta) \cup \mathcal{P}_{e_p}(\zeta)) \setminus 
(\mathcal{P}_1(\zeta) \cap \mathcal{P}_{e_p}(\zeta))$.
There are exactly two perfect matchings of $H_\zeta$ not in this intersection,
therefore (\ref{PhiPhi}) equals
$$\frac{x(P_-(H_\zeta))y(P_-(H_\zeta)) -
x(P_+(H_\zeta))y(P_+(H_\zeta))}{x_{\zeta_1} x_{\zeta_{e_p}}}.$$
By inspection (see the central subgraphs of Figures \ref{LoopSnakeGraphWV} and \ref{figMinM}), 
$x(P_-(H_\zeta)) = x(P_+(H_\zeta)) =
x_{\tau_{i_d}}\big(\prod_{j=1}^{e_p} x_{\zeta_j}\big)$, $y(P_-(H_\zeta))=1$, and $y(P_+(H_\zeta)) =
\prod_{\tau \in T} y_\tau^{e_p(\tau)}$.
\end{proof}

We can now prove Theorem \ref{thm double}.

\begin{proof} [Proof of Theorem \ref{thm double}]
We conclude from Lemmas \ref{mainLemma}, \ref{firstfactor}, and \ref{secondthirdfactor} that
\begin{eqnarray}
\label{LemmaStar}
&& \hspace{2em} \Phi\big(\mathcal{CM}(\zg^{(pq)})\mathcal{M}(\zg) -
\mathcal{SM}(\zg^{(p)})\mathcal{SM}(\zg^{(q)})\big) =
\\ \nonumber
&& 
\big(\prod_{j=1}^{d} x_{\tau_{i_j}}^2\big)
\big(\prod_{j=1}^{e_p} x_{\zeta_j}\big)
\big(\prod_{j=1}^{e_q} x_{\eta_j}\big)
 \big(\prod_{\tau \in T}
y_\tau^{e(\tau,\zg)}\big)
 \big(1-\prod_{\tau \in T}
y_\tau^{e_p(\tau)}\big)
 \big(1-\prod_{\tau \in T}
y_\tau^{e_q(\tau)}\big).
\end{eqnarray}

Using Theorem \ref{thm single}, we have that
$x_{\zg} = \frac{\Phi(\mathcal{M}(\zg))}
{\prod_{j=1}^d x_{\tau_{i_j}}}$ and 
$x_{\zg^{(p)}} x_{\zg^{(q)}}$ is equal to  
$$\frac{\Phi(\mathcal{SM}(\zg^{(p)}))}
{\prod_{\tau \in T} x_{\tau}^{e(\zg,\tau) ~+~ e_p(\tau)}} \cdot
\frac{\Phi(\mathcal{SM}(\zg^{(q)}))}
{\prod_{\tau \in T} x_{\tau}^{e(\zg,\tau) ~+~ e_q(\tau)}}
= \frac{\Phi(\mathcal{SM}(\zg^{(p)}))}
{\prod_{j=1}^d x_{\tau_{i_j}} \prod_{j=1}^{e_p} x_{\zeta_j} } \cdot
\frac{\Phi(\mathcal{SM}(\zg^{(q)}))}
{\prod_{j=1}^d x_{\tau_{i_j}}\prod_{j=1}^{e_q} x_{\eta_j}}.$$

Using (\ref{LemmaStar}), we obtain
\begin{equation*}
\label{almost-there}
\frac{\Phi(\mathcal{CM}(\zg^{(pq)})) \,\Phi(\mathcal{M}(\zg))}
{\prod_{j=1}^d x_{\tau_{i_j}}^2\prod_{j=1}^{e_p} x_{\zeta_j}\prod_{j=1}^{e_q} x_{\eta_j}}
- x_{\zg^{(p)}} \,x_{\zg^{(q)}} 
= (1-\prod_{\tau \in T}
y_\tau^{e_p(\tau)})
 (1-\prod_{\tau \in T}
y_\tau^{e_q(\tau)})
\prod_{\tau \in T}
y_\tau^{e(\tau,\zg)}.
\end{equation*}

Comparing this to 
Theorem \ref{double-identity} 
and using $x_{\zg} = \frac{\Phi(\mathcal{M}(\zg))}
{\prod_{j=1}^d x_{\tau_{i_j}}}$ 
yields 
$$x_{\zg^{(pq)}} = \frac{\Phi(\mathcal{CM}(\zg^{(pq)}))}
{\prod_{j=1}^d x_{\tau_{i_j}}\prod_{j=1}^{e_p} x_{\zeta_j}\prod_{j=1}^{e_q} x_{\eta_j}}.$$
\end{proof}


\subsection{The case of a doubly-notched loop}\label{sec DNL}

The previous section proved our formula for cluster expansions
of cluster variables corresponding to doubly-notched arcs between
two distinct punctures $p$ and $q$.   It remains to understand
the cluster variables
corresponding to doubly-notched loops.

In fact we will use the same strategy for doubly-notched loops
as we used for doubly-notched arcs between two punctures, namely,
we will show that our combinatorial formula for doubly-notched loops
satisfies the identity of Theorem
\ref{double-identity}.  However, we need to explain how to
interpret Theorem \ref{double-identity} when $\rho$ is
a loop, namely an arc between points $p$ and $q$ where $p$ and $q$
happen to coincide.  In this case it is not immediately clear
how to interpret the symbols $x_{\rho^{(p)}}$ and $x_{\rho^{(q)}}$;
a ``singly-notched loop" does not represent a cluster variable.

Before defining the symbol $x_{\rho^{(p)}}$, we need to introduce
a simple operation on bordered surfaces and triangulations.

\begin{definition}
[\emph{Augmentation}]
Fix a bordered surface $(S,M)$, an ideal triangulation $T^\circ$ of $S$,
a puncture $p$, and a loop $\rho$ based at $p$ with a choice of
orientation.
Let $\Delta$ be the first triangle of $T^\circ$ which $\rho$ crosses.
We assume that $\Delta$ is not self-folded, so we
can denote the arcs of $\Delta$ by $a, b$ and $c$ (in clockwise order),
with $a$ and $c$ incident to $p$, $a$ and $b$ incident to a marked
point $u$, and $b$ and $c$ incident to a marked point $v$.
We then define the
{\it augmented bordered surface} $(\widehat{S},\widehat{M})$ by
adding a single puncture $q$ to $(S,M)$, placing it inside $\Delta$.
And we construct the
{\it augmented triangulation} $\widehat{T}^\circ$
from $T^\circ$ by adding three new arcs inside
$\Delta$: an arc $\hat{a}$ from $q$ to $u$,
an arc $\hat{c}$ from $q$ to $v$, and an arc $\hat{b}$ from $u$ to $v$
(so that $\hat{b}$ and $b$ form a bigon with the puncture $q$
inside).    See Figure \ref{Augment}.
\end{definition}
\begin{figure}
\begin{center}
\input{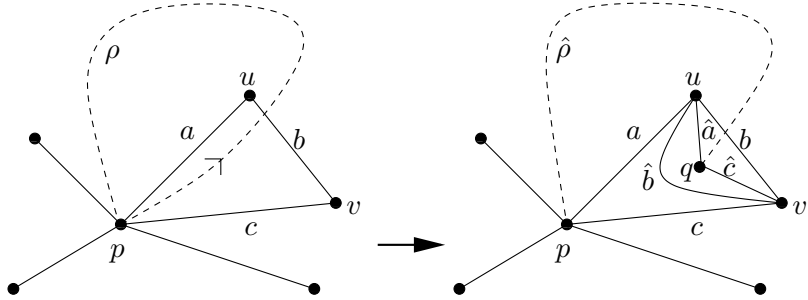}
\end{center}
\caption{Augmenting a bordered surface and triangulation}
\label{Augment}
\end{figure}

\begin{definition}
[\emph{The cluster algebra element corresponding to a singly-notched loop}]
\label{SLoop}
Fix a bordered surface $(S,M)$ and a tagged triangulation
$T = \iota(T^\circ)$ corresponding to an ideal
triangulation $T^\circ$, and
let $\Acal$ be the corresponding cluster algebra with
principal coefficients
with respect to $T$.
Let $\rho$ be an ordinary loop based at $p$,
with a choice of orientation,
and let $\rho^{(p)}$ denote the ``tagged arc" obtained from $\rho$ by
notching at the final end of $\rho$.
We represent this ``tagged arc" by the curve (with a self-intersection) $\ell_p$  obtained by
following the loop $\rho$ along its orientation, but then looping around
the puncture $p$ and doubling back, again following $\rho$.
See Figure \ref{DNLoopReplacements}.
Let $G_{T^\circ, \ell_p}$ be the graph associated to $\ell_p$
in Section \ref{sect graph}.
Then we define $x_{\rho^{(p)}}$  to be
$$\frac{1}{\mathrm{cross}({T}^\circ,{\rho}^{(p)})}
\sum_P \overline{x}(P)\, \overline{y}(P),$$ where the sum is over all
$\rho$-symmetric matchings $P$ of $G_{{T}^\circ,\ell_p}.$
\end{definition}

\begin{prop}\label{augment-specialize}
Using the notation of Definition \ref{SLoop},  let
$\widehat{T}^\circ$ denote the augmented triangulation
corresponding to $T^\circ$ and $\rho$, and let $\hat{\rho}$ denote the arc
in $\widehat{T}^\circ$ from $q$ to $p$ which is equal to $\rho$ after  identification of
$p$ and $q$.   We set $x_{\hat{a}}=x_a$, $x_{\hat{b}}=x_b$,
$x_{\hat{c}}=x_c$, $y_{\hat{a}}=y_a$, $y_{\hat{b}}=y_b$, and $y_{\hat{c}}=y_c$.
Let $\hat{\ell}_p$ denote the loop which is the ideal arc
representing $\hat{\rho}^{(p)}$.
Then $x_{\rho^{(p)}}$ is equal to
$$\frac{1}{\mathrm{cross}(\widehat{T}^\circ,\hat{\rho}^{(p)})}
\sum_P \overline{x}(P)\, \overline{y}(P),$$ where the sum is over all
$\hat{\rho}$-symmetric matchings $P$ of $G_{\widehat{T}^\circ,\hat{\ell}_p}.$
\end{prop}

\begin{remark}
In other words,
$x_{\rho^{(p)}}$ can be obtained by taking the formula for
$x_{\hat{\rho}}$ given by Theorem \ref{thm single}, and making a simple substitution of
variables.
\end{remark}

\begin{proof}
By Remark \ref{assumption}, we can assume that the first triangle
which $\ell_p$ crosses is not self-folded; therefore we can 
augment $T^\circ$.
We defined the augmented triangulation $\widehat{T}^\circ$ so that the
the sequence of diagonals crossed by the loop $\hat{\ell}_p$ in $\widehat{S}$
is identical
to the sequence of diagonals crossed by the curve $\ell_p$
in $S$.
Moreover, the local configurations of all triangles crossed
is the same for both $\hat{\ell}_p$ and $\ell_p$, and after the substitution
$\hat{a}=a$,
$\hat{b}=b$, and
$\hat{c}=c$, even their labels coincide.
(Note that it was essential for us to define
$\widehat{T}^\circ$ so that in the neighborhoods around $p$ in both $S$
and $\widehat{S}$, the two triangulations coincide.)
Therefore, after this substitution, the labeled graphs
$G_{\widehat{T}^\circ,\hat{\ell}_p}$
and $G_{{T}^\circ,\ell_p}$ are equal.
Additionally, the notions of $\hat{\rho}$-symmetric and $\rho$-symmetric matchings
coincide, as do the crossing monomials.
This proves the proposition.
\end{proof}

\begin{prop}\label{Ptolemy-DLoopspecialize}
Fix a bordered surface $(S,M)$ and a tagged triangulation $T=\iota(T^\circ)$,
and let $\Acal$ be the corresponding cluster algebra.
Let $\rho$ be a loop based at a puncture $p$ in $S$.
Choose a marked point $w$ enclosed by $\rho$ and a marked point
$v$ which is not enclosed by $\rho$.  Choose arcs
$\alpha$ and $\beta$ between $p$ and $v$, and arcs $\gamma$ and $\delta$
between $p$ and $w$, so that $\alpha, \beta, \gamma,$
and $\delta$ are the
four sides (in clockwise order) of a quadrilateral with simply-connected interior.
Let $\rho'$ denote the arc between $v$ and $w$ so that $\rho'$ and $\rho$
are the two diagonals of this quadrilateral.
Choose the orientation for $\rho$ which starts at the corner
of the quadrilateral between $\beta$ and $\gamma$, and ends
at the corner between $\alpha$ and $\delta$, and define $x_{\rho^{(p)}}$
as in Definition \ref{SLoop}.
Define $x_{\rho^{(q)}}$ in the same way, but using 
the opposite orientation for $\rho$.  
Let $Y_q^{\pm}$, $Y_p^{\pm}$, and $Y^\pm$ be the monomials of shear coordinates
coming from laminations
as in  Figure \ref{Bigon-Laminate} (which shows a degeneration
of Figure \ref{laminate-spiral} and \ref{partition}).
Then we have (\ref{P1}) and (\ref{P2}).
\end{prop}
\begin{figure}
\begin{center}
\input{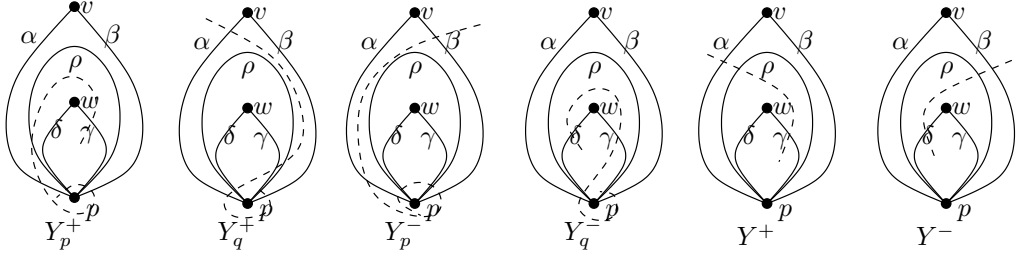}
\end{center}
\caption{Laminations for a quadrilateral in a bigon.}
\label{Bigon-Laminate}
\end{figure}

\begin{proof}
It suffices to prove (\ref{P1}).
We augment $S$ and $T^\circ$, and define arcs
$\hat{\alpha}, \hat{\beta}, \hat{\gamma},$
$\hat{\delta}, \hat{\rho}$, and $\hat{\rho'}$ in $\widehat{S}$,
so that they are the same as  the corresponding arcs
in $S$ except that the endpoints of $\hat{\beta}, \hat{\gamma}$
and $\hat{\rho}$ are moved from $p$ to $q$.  See Figure \ref{Open}.
The underlying triangulations are indicated by
thin  lines, and the sides of the quadrilateral 
are bold.

By Proposition \ref{augment-specialize}, after a simple specialization of variables
(obtained by equating $\hat{a}$, $\hat{b}$, $\hat{c}$ with $a$, $b$, $c$),
$x_{\rho^{(p)}}$ is equal to $x_{\hat{\rho}}$.  Similarly,
$x_{\rho'} = x_{\hat{\rho'}}$, $x_{\beta}=x_{\hat{\beta}}$,
$x_{\delta^{(p)}}=x_{\hat{\delta}^{(p)}}$,
                       $x_{\alpha^{(p)}}=x_{\hat{\alpha}^{(p)}}$,
and $x_{\gamma}=x_{\hat{\gamma}}$.  Note that we are using the special
nature of the augmentation $\widehat{T^\circ}$ of $T^\circ$ and the fact
that it preserves the neighborhood around the puncture $p$.
Finally, we know that in $\widehat{S}$ the equation (\ref{P1}) holds, so 
after the
simple specialization above, the proposition holds.
\end{proof}

\begin{figure}
\begin{center}
\input{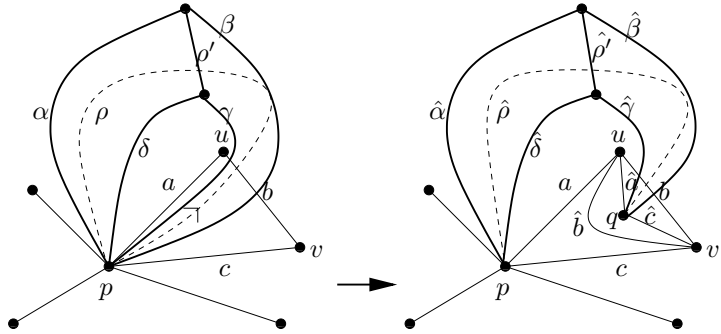}
\end{center}
\caption{Opening the quadrilateral in the bigon}
\label{Open}
\end{figure}

The proof of Theorem \ref{thm double} for doubly-notched arcs 
can now be extended to loops.

\begin{proof}
We've now defined $x_{\rho^{(p)}}$ and 
 $x_{\rho^{(q)}}$, so   
the statement of Theorem \ref{double-identity} makes sense.
(Here $e_p(\tau)=e_q(\tau)$ is the number of ends of arcs 
of $T$ which are incident to $p$.)
Moreover by Proposition \ref{Ptolemy-DLoopspecialize}, equations
(\ref{P1}) and (\ref{P2}) hold, and the proof of Theorem
\ref{double-identity} goes through with minimal modifications.
We now have an algebraic counterpart for
a singly-notched loop given by Proposition \ref{SLoop},
which is analogous to our formula for singly-notched loops.
Using this, the proofs of Section \ref{DoubleNotchedProofs}
go through for doubly-notched loops just as they did for
doubly-notched arcs, with no changes necessary.
This proves our combinatorial formula for cluster variables
of doubly-notched loops as sums over $\gamma$-compatible pairs of matchings.
\end{proof}

\begin{question}
When $\rho$ is a loop,
is $x_{\rho^{(p)}}$ an element of the cluster algebra $\Acal$, or merely an element of
$\Frac(\Acal)$?
\end{question}


\section{Applications to F-polynomials, g-vectors, Euler characteristics}\label{sect applications}
\subsection{$F$-polynomials and $g$-vectors}\label{sect fg}

Fomin and Zelevinsky showed \cite{FZ4} that the Laurent expansions of  
cluster variables can be computed from the somewhat simpler $F$-polynomials and 
$g$-vectors.  In this section we invert this line of thought and compute the $F$-polynomials
and $g$-vectors from our Laurent expansion formulas.  
$F$-polynomials are obtained from Laurent expansions of cluster variables with principal coefficients  by setting all cluster variables equal to $1$. Thus the $F$-polynomial $F_\zg$ 
of a tagged arc $\zg$ is obtained from Theorems \ref{thm main}, \ref{thm single} and \ref{thm double} by deleting the weight and crossing monomials,
 and summing up only the specialized height monomials.
For example, if $\zg$ is an ordinary arc then
\[ F_\zg =\sum_P y(P),\]
where the sum is over all perfect matchings $P$ of $G_{T^\circ,\zg}$.

Note that this shows that $F$-polynomials have constant term $1$, since the minimal matching $P_-$ is the only matching for which the 
specialized height monomial is $1$.

It has been shown \cite{FZ4} that
the  Laurent expansion of any cluster variable $x_\zg$ with respect to a seed 
$(\mathbf{x},\mathbf{y},B)$ is homogeneous
 with respect to the grading given by
$\textup{deg}(x_i)=\mathbf{e}_i$ and
$\textup{deg}(y_i)=B\mathbf{e}_i$, where
$\mathbf{e}_i=(0,\ldots,0,1,0,\ldots,0) \in\mathbb{Z}^n$ with $1$ at
position $i$. By
definition, the \emph{$g$-vector} $g_\zg$ of a cluster variable $x_\zg$ is the
degree of its Laurent expansion  with respect to this grading.
Using the fact that the minimal matching $P_-$ is of height $1$, we see from Theorem \ref{thm main}
that the $g$-vector is given by
$$g_\zg = 
\textup{deg} \displaystyle\left(\frac{{x}(P_-)}{\textup{cross}(T^\circ,\zg)}\right),$$
if $\zg $ is an ordinary arc.  
The same formula works for arcs with one or two notches, replacing
$x(P_-)$ by 
$\overline{x}(P_-)$ or 
$\overline{\overline{x}}(P_-)$, respectively.

\subsection{Euler-Poincar\'e characteristics}\label{sect EPC}

In this section we compare our cluster expansion formula to the formulas
of \cite{DWZ}, which apply to all skew-symmetric cluster algebras, hence in particular, to all cluster algebras that are associated to surfaces.

Let $\mathcal{A} =\mathcal{A}(\mathbf{x},\mathbf{y},B)$ be a rank $n$ cluster algebra 
with principal coefficients associated to a surface. 
We associate to the skew-symmetric matrix 
$B=(b_{ij})$ a quiver $Q(B)$ 
without loops or oriented 2-cycles, by letting $\{1,2,\ldots,n\}$ be the vertex set and 
by drawing $b_{ij}$ arrows from $i$ to $j$ if and only 
if $b_{ij}>0$.

Let $S$ be a potential on $Q(B)$, and consider the corresponding Jacobian algebra. This algebra is the quotient of the complete path algebra of $Q(B)$ by the Jacobian ideal, which is the closure of the ideal generated by the partial cyclic derivatives of the potential. In \cite{DWZ}, the authors associate to any cluster variable $x_\zg$ in $\mathcal{A} $ a finite-dimensional module $M_\zg$ over the Jacobian algebra (thus $M_\zg$ is a representation of the quiver $Q(B)$ whose maps satisfy the relations given by the Jacobian ideal).  Furthermore, they prove that the
$F$-polynomial of $x_{\zg}$ is given by the formula

\[ F_\zg= \sum_e \chi(Gr_e(M_\zg) )\,\prod_{i=1}^n y_i^{e_i}, \]
where the sum is over all dimension vectors $e=(e_1,e_2,\ldots,e_n)$,
$\chi$ denotes the Euler-Poincar\'e characteristic, 
and $Gr_e(M_\zg)$ is the $e$-Grassmannian of
$M_\zg$, that is, the variety of subrepresentations of dimension vector $e$.
Comparing this formula to 
our combinatorial formulas for $F_\zg$,
we get the following results.

\begin{theorem}\label{thm EPC}
\begin{enumerate}
\item For an ordinary  arc $\zg$, the Euler-Poincar\'e characteristic $ \chi(Gr_e(M_\zg))$ is the number of
perfect matchings $P$ of $G_{T^\circ,\zg}$ such that
 the specialized height monomial $y(P)$ is equal to $\prod_{i=1}^n y_i^{e_i}$.
 \item For an arc $\zg=\zg^{(p)}$ with one notched  end,  $\chi(Gr_e(M_\zg))$ is the number of $\zg$-symmetric matchings $P$ of $G_{T^\circ,\ell_p}$ such that
 $\overline y(P)=\prod_{i=1}^n y_i^{e_i}$.
 \item For an arc $\zg=\zg^{(pq)}$ with two notched ends, $ \chi(Gr_e(M_\zg))$ is the number of $\zg$-compatible pairs  $(P_1,P_2)$ of $G_{T^\circ,\ell_p}\sqcup G_{T^\circ,\ell_q}$ such that
 $\overline{\overline{y}}(P_1,P_2)=\prod_{i=1}^n y_i^{e_i}$.
\end{enumerate}
\end{theorem}

\begin{cor}\label{cor EPC} For any cluster variable $x_\zg$ in a cluster algebra associated to a surface, the Euler-Poincar\'e characteristic $ \chi(Gr_e(M_\zg))$ is a non-negative integer.
\end{cor}

\begin{remark} In the case where $Q(B)$ has no oriented cycles, Corollary \ref{cor EPC} was already proved in \cite{CalRein}, and for unpunctured surfaces in \cite{S2}.
\end{remark}

{} 

\end{document}